\documentclass[a4paper,12pt]{amsart}

\usepackage{fullpage}

\usepackage[a4paper, twoside=false, margin=1.1in]{geometry}
\usepackage{mathrsfs}
\usepackage[utf8]{inputenc}
\usepackage[T1]{fontenc}
\usepackage[english]{babel}
\usepackage{amsfonts}
\usepackage{bm}
\usepackage{amsthm}
\usepackage{amssymb}
\usepackage{mathcomp}
\usepackage{mathrsfs}
\usepackage[bb=boondox]{mathalfa}
\usepackage{graphicx}
\usepackage{epstopdf}
\usepackage[centertags]{amsmath}
\usepackage[pdftex]{hyperref}
\usepackage[capitalize]{cleveref}
\usepackage{stmaryrd}

\makeatletter
\def\section{\@startsection{section}{1}{\z@}%
  {25pt}
  {15pt}
  {\normalfont\scshape\centering}}
\makeatother

\usepackage{tikz}
\usepackage{tikz-cd}
\usepackage{dsfont}
\usepackage{cases}
\usepackage{mathtools}
\usepackage{shuffle}
\usepackage{adjustbox}
\usepackage[all]{xy}
\newcommand{\antishriek}{!}
\usepackage{ebgaramond}
\usepackage{euscript}
\renewcommand{\mathcal}{\EuScript}

\usepackage{enumitem}

\allowdisplaybreaks
\definecolor{corail}{rgb}{0.9882,0.4627,0.4157} 
\definecolor{viola}{RGB}{166,146,186}
\definecolor{mocha}{RGB}{164,120,100}
\definecolor{peachfuzz}{rgb}{0.8, 0.5451, 0.3961}
\definecolor{clouddancer}{RGB}{240, 238, 233}
\definecolor{bluefusion}{RGB}{83, 104, 124}

\hypersetup{
    colorlinks,
    linkcolor= bluefusion ,
    citecolor= bluefusion ,
    urlcolor= bluefusion ,
}

\title{Unifying Koszul dualities via point-set models}
\author{Dan Petersen, \quad Victor Roca i Lucio, \quad Sinan Yalin}

\address{Dan Petersen, Stockholm University}
\email{\href{mailto:dan.petersen@math.su.se}{dan.petersen@math.su.se}}

\address{Victor Roca i Lucio, Université Paris Cité and Sorbonne Université, CNRS, IMJ-PRG, F-75013 Paris, France}
\email{\href{mailto:rocalucio@imj-prg.fr}{rocalucio@imj-prg.fr}}

\address{Sinan Yalin, Université d'Angers}
\email{\href{mailto:sinan.yalin@univ-angers.fr}{sinan.yalin@univ-angers.fr}}

\date{\today}

\thanks{The first named author was supported by the National Science Foundation
under Grant No.~DMS-2424441, and a Wallenberg Scholar fellowship. The second named author was supported by the ANR project SHoCos ANR-22-CE40-0008. The third named author acknowledges the support of the ANR projects LieDG ANR-24-CE40-5367 and KAsH ANR-25-CE40-2861.}


\newtheorem{theoremintro}{Theorem}

\theoremstyle{plain}
\newtheorem{theorem}{Theorem}[section]
\newtheorem{lemma}[theorem]{Lemma}
\newtheorem{proposition}[theorem]{Proposition}
\newtheorem{assumption}{Assumption}
\newtheorem{corollary}[theorem]{Corollary}

\theoremstyle{definition}
\newtheorem{definition}[theorem]{Definition}

\theoremstyle{remark}
\newtheorem{remark}[theorem]{\sc Remark}
\newtheorem{example}[theorem]{\sc Example}
\newtheorem{notation}[theorem]{\sc Notation}
\newtheorem{counterexample}[theorem]{\sc Counter-Example}

\newcommand{\qi}{\xrightarrow{ \,\smash{\raisebox{-0.65ex}{\ensuremath{\scriptstyle\sim}}}\,}}

\newcommand{\operad}[1]{\mathcal{#1}}

\newcommand{\kk}{\Bbbk}
\newcommand{\PP}{\mathcal{P}}
\newcommand{\C}{\mathcal{C}}

\DeclareMathOperator*{\colim}{colim}

\newcommand{\poubelle}[1]{}

\begin{document}
\pagecolor{clouddancer}

\newcommand{\Victor}[1]{\textcolor{red}{#1}}
\newcommand{\Sinan}[1]{\textcolor{blue}{#1}}
\newcommand{\Dan}[1]{\textcolor{orange}{#1}}
\newcommand{\Z}{\mathbb{Z}}

\begin{abstract}
The classical bar-cobar adjunction between dg algebras and dg coalgebras goes back to the origins of \emph{differential homological algebra} as developed by Cartan, Eilenberg, Moore, and many others, and is part of the broader framework of \emph{Koszul duality}. In recent years, several $\infty$-categorical analogues of this adjunction have been developed, notably by Lurie, Francis--Gaitsgory, and Heuts. However, there is no comparison in the literature between the classical chain-level constructions and their higher-categorical counterparts, and in fact the two constructions are not quite compatible.

In this paper we provide a unified framework relating these different forms of Koszul duality in the differential graded setting. We construct a commutative square of adjunctions---called the inclusion-restriction square---intertwining the classical operadic bar-cobar adjunction with its completed variant due to Le Grignou--Lejay. We show that this square induces an $\infty$-categorical adjunction between algebras and their Koszul dual coalgebras, recovering in particular the differential graded case of Lurie's bar-cobar adjunction, and explain precisely how our constructions relate to those of Francis--Gaitsgory and Heuts.
\end{abstract}

\maketitle

\tableofcontents

\section*{Introduction}

\subsection{Bar-cobar duality}
A fundamental object of study in homological and homotopical algebra is the \emph{bar-cobar adjunction} \cite{CartanEilenberg56, adams}:
\begin{equation}\label{algebraic bar/cobar}
	\begin{tikzcd}
		(\text{augmented dg algebras}) \arrow[r,shift right=1ex,"\mathrm{B}"']&\arrow[l, shift right=1ex,"\Omega"'](\text{conilpotent coaugmented dg coalgebras}). \arrow[phantom, from=1-1, to=1-2, "{\scriptstyle\perp}"]
	\end{tikzcd}
\end{equation} 
	These functors have their origin in algebraic topology --- they are algebraic shadows of the formation of \emph{classifying spaces} and \emph{loop spaces}. If $M$ is a topological monoid, then its chains $C_*(M;\Z)$ is an augmented dg algebra under Pontryagin product, and there is a quasi-isomorphism of dg coalgebras \cite{stasheff2,mayfibrations}:
\[ \mathrm{B} C_*(M;\Z) \simeq C_*(\mathrm{B} M;\Z). \]
Similarly if $X$ is a connected based space, then $C_*(X;\Z)$ is a coaugmented dg coalgebra, and there is a quasi-isomorphism of dg algebras \cite{adams,riverazeinalian}:
\[ \Omega C_*(X;\Z) \simeq C_*(\Omega X;\Z).\]
These theorems justify the notations $\mathrm B$ and $\Omega$ for the bar and cobar functors. The topological analogue of \eqref{algebraic bar/cobar} is the adjunction
\begin{equation}\label{topological bar/cobar}
	\begin{tikzcd}
		(\text{topological monoids}) \arrow[r,shift left=1ex,"\mathrm{B}"]&\arrow[l, shift left=1ex,"\Omega"](\text{based spaces}). \arrow[phantom, from=1-1, to=1-2, "{\scriptstyle\perp}"]
	\end{tikzcd}
\end{equation}	
Up to homotopy, \eqref{topological bar/cobar} is nearly an equivalence: it restricts to an equivalence of $\infty$-categories between group-like topological monoids and connected based spaces. Its algebraic shadow \eqref{algebraic bar/cobar} is likewise an equivalence up to quasi-isomorphism, if one restricts to suitably connective objects. This almost-equivalence between the homotopy theories of algebras and coalgebras,
	 and a number of closely related algebraic dualities of a similar flavor, are often collectively referred to as \emph{Koszul duality}. 

\medskip

	In fact, there are a bewildering number of notions\footnote{To add to the confusion, some authors reserve the phrase ``Koszul duality'' only for a more restrictive notion involving \emph{small resolutions}. This is also the only part of the story which relates to the work of Koszul, via the \emph{Koszul complex} of \cite{Priddy70}. A quadratic algebra $A$ has a \emph{quadratic dual} coalgebra $A^\antishriek$ and a canonical map $A^\antishriek \to \mathrm BA $, and one says that $A$ is a \emph{Koszul algebra} if this canonical map is a quasi-isomorphism. In this case, its adjoint $\Omega A^\antishriek \to A$ is also a quasi-isomorphism, and a minimal resolution of $A$.
    
    While there are several generalizations of this version of Koszul duality to other algebraic objects, one always constructs first a suitable bar-cobar adjunction that establishes a \textit{homotopical} Koszul duality between algebras and coalgebras, with respect to which the story about minimal resolutions becomes meaningful. The present paper is concerned only with the underlying homotopy-theoretic duality.  }
    %
    in the literature which by now are called Koszul duality. A full historical account is beyond the scope of this introduction, but we wish to mention \cite{Priddy70,EM,HMS,quillenrationalhomotopytheory,sullivaninfinitesimal,BGG,ginzburgkapranov,BGS,hinichdgcoalgebras,LodayVallette}.

\medskip

A more recent, less classical approach to the bar-cobar adjunction is due to Jacob Lurie \cite[Section 5.2.2]{HigherAlgebra}. Let $\mathbf C$ be a monoidal $\infty$-category which admits geometric realizations of simplicial objects and totalizations of cosimplicial objects. Under these assumptions, Lurie constructs an adjunction
\begin{equation}\label{lurie bar/cobar}
	\begin{tikzcd}
		(\text{augmented $\mathbb E_1$-algebras in $\mathbf C$}) \arrow[r,shift left=1ex,"\mathbf{bar}"]&\arrow[l, shift left=1ex,"\mathbf{cobar}"](\text{coaugmented $\mathbb E_1$-coalgebras in $\mathbf C$}). \arrow[phantom, from=1-1, to=1-2, "{\scriptstyle\perp}"]
	\end{tikzcd}
\end{equation}	
When we specialize $\mathbf C$ to the $\infty$-category of spaces $\mathbf{Spc}$, the adjunction \eqref{lurie bar/cobar} specializes to \eqref{topological bar/cobar}. Indeed Lurie proves that $\mathbf{bar}(A)$ can be computed as the realization of a simplicial object
\begin{equation}\label{simplicial description of bar}
\begin{tikzcd}[row sep=huge, column sep=huge]
  \dots\hspace{3ex} A \otimes A
	\arrow[r, shift left=3]
	\arrow[r]
	\arrow[r, shift right=3]
	& 
	A
	\arrow[r, shift left=1.5]
	\arrow[r,  shift right=1.5]
	\arrow[l, shift left=1.5]
	\arrow[l, shift right=1.5]
	& \mathbf 1 \arrow[l]
\end{tikzcd}
\end{equation}which gives back the familiar formula for the classifying space of a topological monoid. Note that every space is an $\mathbb E_1$-coalgebra in a unique way (with comultiplication the diagonal), and a coaugmentation is just a basepoint; also, every $\mathbb E_1$-algebra in spaces is automatically augmented, and can be rectified to a topological monoid. 

\medskip

Let us take $\mathbf C=\mathbf D(\kk)$, the derived category of a ring $\kk$. It would then be natural to think that for this choice of $\mathbf C$, the adjunction \eqref{lurie bar/cobar} specializes to \eqref{algebraic bar/cobar}. This is \emph{not} the case. In fact, there are several immediate problems:
\begin{enumerate}[label=(\roman*)]
	\item Lurie's bar functor is \emph{left} adjoint; the algebraic bar functor is \emph{right} adjoint.

	\item Lurie does not impose conilpotence anywhere; algebraically, conilpotence is essential in order for $\mathrm{B}$ and $\Omega$ to be adjoints.

	\item \label{itemiii} The category of dg coalgebras localized at quasi-isomorphisms does not present the $\infty$-category of $\mathbb E_1$-coalgebras in the derived $\infty$-category of $\kk$, no matter what ring $\kk$ we choose.
\end{enumerate}

Nevertheless, \eqref{simplicial description of bar} shows that Lurie's $\mathbf{bar}(A)$ is quasi-isomorphic to the classical algebraic bar construction $\mathrm{B}A$ for any augmented dg algebra $A$. So, when $\mathbf C=\mathbf D(\kk)$, if we consider the adjunction \eqref{lurie bar/cobar} in this case, there are a number of natural questions:

\begin{enumerate}[label=(\alph*)]
	\item \label{itema} Can we construct explicit chain-level constructions of {the functors} $\mathbf{bar}$ and $\mathbf{cobar}$?

	\item \label{itemc}What, if any, is the relationship between these chain-level constructions and the adjunction \eqref{algebraic bar/cobar}?

	\item \label{itemb}Can we present \eqref{lurie bar/cobar} by a Quillen adjunction?

\end{enumerate} 

The goal of this paper is in particular to answer all of these questions, when $\kk$ is a field.

\subsection{The associative case} Let $\kk$ be a field.  Before we can even meaningfully consider question \ref{itema}, we need to find suitable point-set models for $\mathbb{E}_1$-coalgebras in $\mathbf{D}(\kk)$, which was the main goal of \cite{rectification}. From now on we shall find it convenient to restrict our attention to non-unital algebras and coalgebras, rather than keeping track of augmentations and coaugmentations. In \cite[Theorem A]{rectification}, we showed that there exists an equivalence 
 \[ 
 \mathbf{Coalg}_{\mathbb{E}_1}^{\text{non-unital}}(\mathbf D(\kk)) \simeq \mathcal{A}_\infty \mbox{-}\mathsf{coalg}~[\mathsf{Q.iso}^{-1}]
\]
between the $\infty$-categories of non-unital $\mathbb{E}_1$-coalgebras in $\mathbf{D}(\kk)$, and (not necessarily conilpotent) $\mathcal{A}_\infty$-coalgebras up to quasi-isomorph\-ism. Informally, this says that the reason for problem \ref{itemiii} is that coassociative coalgebras are \emph{too strict} to exhibit homotopically correct behavior, and need to be replaced by a flabbier notion.

\medskip

Before proceeding further, let us dwell on the definitions of $\mathcal{A}_\infty$-algebra and $\mathcal{A}_\infty$-coalgebra. Usually, an $\mathcal A_\infty$-algebra is defined as a graded vector space $A$ with a family of maps
\[ \mu_n : A^{\otimes n} \longrightarrow A\]
of degree $n-2$, satisfying an infinite hierarchy of quadratic relations: these express that $\mu_1^2=0$, so that $(A,\mu_1)$ is a chain complex; $\mu_2$ is a multiplication satisfying the Leibniz rule with respect to the differential $\mu_1$; $\mu_3$ is a homotopy with respect to $\mu_1$ between $\mu_2 \circ (\mu_2 \otimes \mathrm{id})$ and $\mu_2 \circ (\mathrm{id} \otimes \mu_2)$, correcting the failure of $\mu_2$ to be strictly associative, and so on. An efficient way to package all this data (and the signs involved) is to observe that the maps $\mu_i$ determine, and are determined by, a coderivation $\boldsymbol\mu$ of degree $-1$ of the (reduced) tensor coalgebra $\bar{T}^c(sA)$ on the suspension of $A$, since such a coderivation is freely determined by its composition with the projection to the cogenerators:
\begin{equation}\label{a-infinity structure as coderivation}
    \bigoplus_{n \geq 1} (sA)^{\otimes n} = \bar{T}^c(sA) \stackrel{\boldsymbol{\mu}}\longrightarrow \bar{T}^c(sA) \twoheadrightarrow sA.
\end{equation} 
In these terms, the infinite hierarchy of quadratic relations becomes $\boldsymbol{\mu}^2=0$. An advantage of this definition is that it becomes tautologically obvious how to extend the bar construction $\mathrm B$ from associative algebras to $\mathcal A_\infty$-algebras: if $A$ is an $\mathcal A_\infty$-algebra, then $\mathrm BA$ is the dg coalgebra given by $\bar{T}^c(sA)$, with the differential determined by $\boldsymbol\mu$. 

\medskip

An $\mathcal A_\infty$-coalgebra can likewise be defined as a graded vector space $C$ and a family of maps 
\[ \delta_n : C \longrightarrow C^{\otimes n}\]
of degree $n-2$, satisfying a dual hierarchy of quadratic relations. The dual statement of \eqref{a-infinity structure as coderivation} is that any family of maps $\delta_n : C \longrightarrow C^{\otimes n}$ uniquely determines a degree $-1$ derivation $\boldsymbol{\delta}$ 
\begin{equation}\label{a-infinity structure as derivation}
    s^{-1}C \hookrightarrow \bar{T}^{\wedge}(s^{-1}C) \stackrel {\boldsymbol{\delta}}\longrightarrow \bar{T}^{\wedge}(s^{-1} C) = \prod_{n \geq 1} (s^{-1}C)^{\otimes n}
\end{equation}
of the \emph{completion} of the tensor algebra on the desuspension of $C$, since such a derivation is freely determined by where it sends the generators. This description makes it clear that the classical cobar functor $\Omega$ for coassociative coalgebras (defined as in \eqref{a-infinity structure as derivation}, but with a direct sum) does not extend to a well-defined functor on $\mathcal A_\infty$-coalgebras, but the \emph{completed cobar functor}, which we shall denote $\widehat\Omega$, \emph{does} extend in a tautological manner. 

\medskip
In fact, it will turn out that Lurie's functor $\mathbf{cobar}$ is given (on the chain level) by $\widehat{\Omega}$, the completion of the classical cobar functor. There are further reasons that $\smash{\widehat\Omega}$ may be more natural than $\Omega$:
\begin{itemize}
    \item The classical cobar functor $\Omega$ does not preserve quasi-isomorphisms, unless the coalgebras involved are suitable connective or coconnective. The reason is that $\Omega$ sends quasi-isomorphisms to filtered quasi-isomorphisms\footnote{A filtered quasi-isomorphism is not in general a quasi-isomorphism!}, but only for a filtration which is generally not complete. Passing to the completion repairs this defect, and $\smash{\widehat\Omega}$ always preserves quasi-isomorphisms.
    \item Lurie's $\mathbf{bar}$ and $\mathbf{cobar}$ are interchanged under switching $\mathbf C$ and $\mathbf C^{\mathrm{op}}$. Then it is natural to expect that $\mathbf{cobar}$ should be given by a formula dual to that for $\mathbf{bar}$, which entails switching $\bigoplus$ for $\prod$. 
\end{itemize}
The drawback of $\widehat{\Omega}$ compared to $\Omega$ is that it is in no sense an adjoint of the bar construction $\mathrm B$, working 1-categorically. This makes the following theorem surprising.

\begin{theoremintro}\label{theorem a}Let $\kk$ be a field of any characteristic. Let $\mathsf{dg}~\mathsf{assoc}~\mathsf{alg}$ and $\mathcal{A}_\infty \mbox{-}\mathsf{coalg} $ denote the relative categories of dg associative algebras, and $\mathcal A_\infty$-coalgebras, respectively, with weak equivalences given by quasi-isomorphisms. The functors 
\[ \begin{tikzcd}[column sep= 3pc]
	\mathsf{dg}~\mathsf{assoc}~\mathsf{alg} \arrow[r,shift left=.5ex,"\mathrm{B}"]&\arrow[l, shift left=.5ex,"\widehat{\Omega}"]\mathcal{A}_\infty \mbox{-}\mathsf{coalg} \arrow[phantom, from=1-1, to=1-2, ""]
	\end{tikzcd}\]
are both relative functors, and hence induce functors between the respective localizations, which we continue to denote by the same name. Even though $\mathrm B$ and $\smash{\widehat \Omega}$ are not adjoints 1-categorically, the localized functors 
\begin{equation}\label{theorem A adjunction}
\begin{tikzcd}[column sep=5pc,row sep=3pc]
            \mathsf{dg}~\mathsf{assoc}~\mathsf{alg}~[\mathsf{Q.iso}^{-1}] \arrow[r,"\mathrm{B}"{name=F},shift left=1.1ex ] & \mathcal{A}_\infty \mbox{-}\mathsf{coalg}~[\mathsf{Q.iso}^{-1}] \arrow[l,"\widehat{\Omega}"{name=U},shift left=1.1ex ]
            \arrow[phantom, from=F, to=U, , "\dashv" rotate=-90]
\end{tikzcd}
    \end{equation}
are adjoints, and \eqref{theorem A adjunction} is equivalent to Lurie's adjunction \eqref{lurie bar/cobar}.
\end{theoremintro}

Let us explain how to understand the \emph{unit} of the adjunction \eqref{theorem A adjunction}. Let $A$ be an associative dg algebra. There exists a zig-zag
\[ A \stackrel \sim \longleftarrow \Omega\mathrm BA \longrightarrow \widehat\Omega \mathrm B A\]
in which the first arrow is the \emph{counit} of the classical adjunction \eqref{algebraic bar/cobar}, which is classically known to be a quasi-isomorphism for any $A$, and the second arrow comes from the fact that the ordinary cobar construction sits inside its completion. Working $\infty$-categorically, the first arrow is ``invisible'', and it looks as if we just have a map $A\to\smash{\widehat\Omega}\mathrm B A$, which explains the reversal of handedness between \eqref{algebraic bar/cobar} and \eqref{theorem A adjunction}. 

\medskip

The above description also identifies the unit with a derived functor: $A \stackrel\sim\longleftarrow \Omega \mathrm B A$ is a cofibrant replacement of $A$, so the unit is given by first cofibrantly replacing and then taking the completion, which is precisely the \emph{derived completion} (this notion has appeared in \cite{HarperHess13,CamposPetersen24,absolutealgebras}). One can completely similarly describe the counit of the adjunction \eqref{theorem A adjunction} as a zig-zag
\[ C \stackrel \sim \longrightarrow \widehat{\mathrm{B}}\widehat{\Omega} C \leftarrow \mathrm B \widehat \Omega C,\]
where $\widehat{\mathrm{B}}$ is a non-conilpotent enlargement of the classical bar construction; we shall describe it further shortly. Now the first arrow is a \emph{fibrant} replacement, and the similar consequence is that the counit of \eqref{theorem A adjunction} is the derived functor of taking the maximal conilpotent subcoalgebra. We thus obtain:

\begin{theoremintro}\label{description of unit}
    The unit of the adjunction \eqref{theorem A adjunction} is the derived completion functor, and the counit of the adjunction is the derived functor of taking the maximal conilpotent subcoalgebra.
\end{theoremintro}

\cref{description of unit} is a special case of a more general theorem of Heuts \cite{heuts2024}, but we see it here very explicitly. If we define \textit{homotopy complete} dg algebras as those for which the unit of \eqref{theorem A adjunction} is a quasi-isomorphism, and \textit{homotopy cocomplete} $\mathcal A_\infty$-coalgebras as those for which the counit of \eqref{theorem A adjunction} is a quasi-isomorphism, then \eqref{theorem A adjunction} restricts tautologically to an equivalence of $\infty$-categories between homotopy complete associative algebras and homotopy cocomplete $\mathcal A_\infty$-coalgebras. 
 
\subsection{The inclusion-restriction square}

Our \cref{theorem a} answers, in a sense, both questions \ref{itema} and \ref{itemc}, but not \ref{itemb}. In fact, relative categories are an unwieldy model of homotopy theory, and it would be much preferable to have a description of Lurie's bar-cobar adjunction as a Quillen adjunction. What we shall do is something almost as good: we will present it as a composite of \emph{two} Quillen adjunctions, with opposite handedness. 

\medskip

To carry this out, we need to realize that there another 1-categorical bar-cobar adjunction, dual to (\ref{algebraic bar/cobar}), which relates (non-necessarily conilpotent) coalgebras and a new type of algebraic objects called \textit{absolute algebras}. These are associative algebras where infinite sums of structural operations are well-defined \textit{by definition}, without assuming an underlying topology. Absolute algebras are closely related to complete algebras, but subtly different. The free absolute algebra on a module $M$ is however isomorphic to the completion of the free algebra on $M$. We refer to \cite{linearcoalgebras,absolutealgebras} for more details on these structures. There is a \textit{complete bar-cobar adjunction}  \cite{linearcoalgebras}:
\begin{equation}\label{complete algebraic bar/cobar}
\begin{tikzcd}[column sep= 3pc]
	\mathsf{abs}~\mathsf{dg}~\mathsf{assoc}~\mathsf{alg} \arrow[r,shift right=1ex,"\widehat{\mathrm{B}}"']&\arrow[l, shift right=1ex,"\widehat{\Omega}"']\mathcal{A}_\infty \mbox{-}\mathsf{coalg} \arrow[phantom, from=1-1, to=1-2, "{\scriptstyle\perp}"]
	\end{tikzcd}
\end{equation} 
between (non-necessarily conilpotent) $\mathcal{A}_\infty$-coalgebras and dg absolute associative algebras. Here the complete cobar functor $\smash{\widehat{\Omega}}$ is defined as before. The complete bar functor  is informally defined  by replacing the tensor coalgebra used in the functor $\mathrm B$ with  the (genuine, non-conilpotent) cofree $\mathcal{A}_\infty$-coalgebra. 

\medskip

The final step is to combine the bar-cobar adjunction (\ref{algebraic bar/cobar}) with its dual completed version (\ref{complete algebraic bar/cobar}) to get the following \textit{inclusion-restriction} square of adjunctions
\[
\begin{tikzcd}[column sep=5pc,row sep=4pc]
\mathsf{dg}~\mathsf{assoc}~\mathsf{alg} \arrow[r,"\mathrm{B}"{name=B},shift left=1.1ex] \arrow[d,"c\mathrm{Abs}"{name=SD},shift left=1.1ex ]
&\mathcal{A}_\infty \mbox{-}\mathsf{coalg}^{\mathsf{conil}}, \arrow[d,"\mathrm{Inc}"{name=LDC},shift left=1.1ex ] \arrow[l,"\Omega"{name=C},,shift left=1.1ex]  \\
\mathsf{abs}~\mathsf{dg}~\mathsf{assoc}~\mathsf{alg}\arrow[r,"\widehat{\mathrm{B}}"{name=CC},shift left=1.1ex]  \arrow[u,"\mathrm{Res}"{name=LD},shift left=1.1ex ]
&\mathcal{A}_\infty \mbox{-}\mathsf{coalg}~. \arrow[l,"\widehat{\Omega}"{name=CB},shift left=1.1ex] \arrow[u,"\mathrm{Sub}"{name=TD},shift left=1.1ex] \arrow[phantom, from=SD, to=LD, , "\dashv" rotate=180] \arrow[phantom, from=C, to=B, , "\dashv" rotate=90]\arrow[phantom, from=TD, to=LDC, , "\dashv" rotate=180] \arrow[phantom, from=CC, to=CB, , "\dashv" rotate=90]
\end{tikzcd}
\]
Here the functor $\mathrm{Res}$ restricts the structure maps of any absolute dg associative algebra to finite sums of operations, and thus gives its underlying associative algebra; it admits a left adjoint $c\mathrm{Abs}$.  The functor $\mathrm{Inc}$ includes conilpotent $\mathcal{A}_\infty$-coalgebras into all coalgebras, and it has a right adjoint $\mathrm{Sub}$. 

\begin{theoremintro}\label{thoorem c}
    The projective model structure on chain complexes may be right-transferred to a model structure on $\mathsf{dg}~\mathsf{assoc}~\mathsf{alg}$, and then further left-transferred along $\Omega$ to a model structure on the category $\mathcal{A}_\infty \mbox{-}\mathsf{coalg}^{\mathsf{conil}}$. Dually, the injective model structure on chain complexes may be left-transferred to a model structure on $\mathcal{A}_\infty \mbox{-}\mathsf{coalg}$, and then further right-transferred to a model structure on the category $\mathsf{abs}~\mathsf{dg}~\mathsf{assoc}~\mathsf{alg}$. With respect to these model structures, all the above adjunctions are Quillen adjunctions. The functor $\mathrm{B} \circ \mathrm{Inc}$ is the functor denoted $\mathrm B$ in \cref{theorem a}, and the functor $\mathrm{Res} \circ \widehat\Omega$ is the functor denoted $\widehat\Omega$ in \cref{theorem a}. 
\end{theoremintro}

\subsection{Main results} The situation explained before is just an example of more general phenomena in the particular case of associative algebras. In the general context, we will work with a dg operad $\mathcal{P}$. We assume: \begin{itemize}
    \item $\mathcal{P}$ is reduced,
    \item $\mathcal P$ is non-symmetric or $\operatorname{char}(\kk)=0$,
    \item each $\mathcal{P}(n)$ is a dualizable chain complex, that is, a complex of total finite dimension.
\end{itemize}Given this data, we construct a general inclusion-restriction square
\begin{equation}\label{general inclusion-restriction square}
\begin{tikzcd}[column sep=5pc,row sep=4pc]
\mathcal{P}\mbox{-}\mathsf{alg}\arrow[r,"\mathrm{B}_\pi"{name=B},shift left=1.1ex] \arrow[d,"c\mathrm{Abs}"{name=SD},shift left=1.1ex ]
&\mathrm{B}\mathcal{P}\mbox{-}\mathsf{coalg}, \arrow[d,"\mathrm{Inc}"{name=LDC},shift left=1.1ex ] \arrow[l,"\Omega_\pi"{name=C},,shift left=1.1ex]  \\
\mathcal{P}^*\mbox{-}\mathsf{alg}^{\mathsf{comp}} \arrow[r,"\widehat{\mathrm{B}}_\iota"{name=CC},shift left=1.1ex]  \arrow[u,"\mathrm{Res}"{name=LD},shift left=1.1ex ]
&\Omega\mathcal{P}^*\mbox{-}\mathsf{coalg} ~. \arrow[l,"\widehat{\Omega}_\iota"{name=CB},shift left=1.1ex] \arrow[u,"\mathrm{Sub}"{name=TD},shift left=1.1ex] \arrow[phantom, from=SD, to=LD, , "\dashv" rotate=180] \arrow[phantom, from=C, to=B, , "\dashv" rotate=90]\arrow[phantom, from=TD, to=LDC, , "\dashv" rotate=180] \arrow[phantom, from=CC, to=CB, , "\dashv" rotate=90]
\end{tikzcd}
\end{equation}
The operadic bar-cobar adjunction $\mathrm{B}_\pi \vdash \Omega_\pi$, generalizing \eqref{algebraic bar/cobar}, is the one studied in \cite{LodayVallette}; its dual completed version $\smash{\widehat{\mathrm{B}}_\iota \vdash \widehat \Omega_\iota}$, generalizing \eqref{complete algebraic bar/cobar}, is studied in \cite{linearcoalgebras}.

\begin{theoremintro}
The categories in the square \eqref{general inclusion-restriction square} may be given transferred model structure in the same way as in \cref{thoorem c}, making all adjunctions Quillen adjunctions. If we set
\[
\mathbf{bar}_{\mathcal{P}} \coloneqq \mathrm{Inc} \circ \mathrm{B}_\pi~, \quad \mathbf{cobar}_{\mathcal{P}} \coloneqq \mathrm{Res} \circ \widehat{\Omega}_\iota~, 
\]then there is an adjunction of $\infty$-categories 
\begin{equation}\label{theorem d adjunction}
\begin{tikzcd}[column sep=5pc,row sep=3pc]
            \mathcal{P}\mbox{-}\mathsf{alg}~[\mathsf{Q.iso}^{-1}] \arrow[r,"\mathbf{bar}_{\mathcal{P}}"{name=F},shift left=1.1ex ] & \Omega\mathcal{P}^*\mbox{-}\mathsf{coalg}~[\mathsf{Q.iso}^{-1}] \arrow[l,"\mathbf{cobar}_{\mathcal{P}}"{name=U},shift left=1.1ex ]
            \arrow[phantom, from=F, to=U, , "\dashv" rotate=-90]
\end{tikzcd}
\end{equation}
between the $\infty$-category of dg $\mathcal{P}$-algebras localized at quasi-isomorphisms and the $\infty$-category of dg $\Omega\mathcal{P}^*$-coalgebras localized at quasi-isomorphisms. The resulting adjunction \eqref{theorem d adjunction} coincides with the differential graded case of the adjunction considered by Heuts in \cite{heuts2024}.
\end{theoremintro}


As in the discussion around \cref{description of unit}, we can define \textit{homotopy complete} $\mathcal{P}$-algebras as those which are fully faithfully included into absolute $\mathcal{P}$-algebras, that is, $\mathcal{P}$-algebras $A$ such that the derived unit $
A \longrightarrow \mathrm{Res} \circ \mathbb{L}c\mathrm{Abs}(A)
$
is a quasi-isomorphism. We can also define \textit{homotopy cocomplete} $\Omega\mathcal{P}^*$-coalgebras as those which are fully faithfully included into conilpotent $\mathrm{B}\mathcal{P}$-coalgebras, that is, $\Omega\mathcal{P}^*$-coalgebras $C$ such that the derived counit 
\(
\mathrm{Inc} \circ \mathbb{R}\mathrm{Sub}(C) \longrightarrow C
\)
is a quasi-isomorphism. As in \cref{description of unit} these are also the unit and counit of the $\infty$-categorical adjunction $\mathbf{bar}_{\mathcal{P}} \dashv \mathbf{cobar}_{\mathcal{P}}$, and we obtain an equivalence of $\infty$-categories between homotopy complete $\mathcal{P}$-algebras and homotopy cocomplete $\Omega\mathcal{P}^*$-coalgebras. This gives a new proof of  \cite[Theorem 2.1]{heuts2024} in the special case that $\mathbf C = \mathbf D(\kk)$, and moreover gives an explicit point-set description of each of these categories. 

\medskip

An important class of operad are operads \textit{with good completions}, which are operads such that trivial algebras are homotopy complete. These have been introduced and studied by Heuts in \cite{heuts2024}. In Subsection \ref{subsection: operads with good completions}, we give explicit criteria for an operad to have good completions in terms of our point-set descripions, and we also construct a first example of an operad without good completions. We also introduce the dual notion of a cofibrant operad \textit{with good cocompletions}, that is, such that trivial coalgebras over it are homotopy cocomplete, and we study examples. 

\medskip

\subsection{Francis--Gaitsgory's enhanced bar-cobar adjunction.} There is yet another bar-cobar adjunction that can be constructed at the $\infty$-categorical level \cite{FrancisGaitsgory,GaitsgoryRozenblyumVolII}, and which is different from all those mentioned before. Let $\mathscr{P}$ be an augmented $\infty$-operad enriched in $\mathbf{D}(\kk)$. The augmentation $\mathscr{P} \longrightarrow \mathbb{1}$ induces an adjunction 
\[
\begin{tikzcd}[column sep=5pc,row sep=3pc]
         \mathbf{Alg}_{\mathscr{P}}(\mathbf{D}(\kk)) \arrow[r, shift left=1.1ex, "\mathbf{indec}"{name=F}] &\mathbf{D}(\kk) \arrow[l, shift left=.75ex, "\mathbf{triv}"{name=U}]
            \arrow[phantom, from=F, to=U, , "\dashv" rotate=-90]
\end{tikzcd}
\]
between the $\infty$-category of $\mathscr{P}$-algebras and the underlying $\infty$-category of chain complexes over $\kk$. The functor $\mathbf{triv}$ endows any chain complex with a trivial $\mathscr{P}$-algebra structure, and its left adjoint $\mathbf{indec}$ is the functor of derived indecomposables of a $\mathscr{P}$-algebra. The \textit{homology} of $\mathbf{indec}$ is called the (topological) André--Quillen homology, see \cite{Basterra99, Mandell03}. When $\mathscr{P}$ is the $\mathbb{E}_\infty$-operad in $\mathbf{D}(\kk)$, the functor $\mathbf{indec}$ is called the cotangent functor. 

\medskip

The point now is that this adjunction factors universally through coalgebras over the induced comonad $\mathbf{indec} \circ \mathbf{triv}$ on $\mathbf{D}(\kk)$, and that these coalgebras can be identified with conilpotent $\mathrm{B}\mathscr{P}$-coalgebras. Therefore, we get an adjunction 
\[
\begin{tikzcd}[column sep=5pc,row sep=1pc]
         \mathbf{Alg}_{\mathscr{P}}(\mathbf{D}(\kk)) \arrow[r, shift left=1.1ex, "\mathbf{B}^{\mathbf{enh}}"{name=F}] &\mathbf{Coalg}_{\mathbf{bar}(\mathscr{P})}^{\mathbf{conil}}(\mathbf{D}(\kk))~, \arrow[l, shift left=.75ex, "\bm{\Omega}^{\mathbf{enh}}"{name=U}] \arrow[phantom, from=F, to=U, , "\dashv" rotate=-90]
\end{tikzcd}
\]
called the \textit{enhanced bar-cobar adjunction} in \cite{FrancisGaitsgory}, between $\mathscr{P}$-algebras and conilpotent divided powers $\mathbf{bar}(\mathscr{P})$-coalgebras. 

\medskip

Naively, one would like to say that the top horizontal adjunction in (\ref{general inclusion-restriction square}) — which generalizes (\ref{algebraic bar/cobar}) — induces the enhanced bar-cobar adjunction at the $\infty$-categorical level. However, as pointed out in  \cite[Theorem 5.3.1]{chen25}, this is not true in general. The main reason is that point-set conilpotent $\mathrm{B}\mathcal{P}$-coalgebras up to quasi-isomorphisms are \emph{not} equivalent to their $\infty$-categorical counterparts. What we get is the following diagram of adjunctions at the $\infty$-categorical level 
\[
\begin{tikzcd}[column sep=3pc,row sep=3.5pc]
\mathcal{P}\mbox{-}\mathsf{alg}~[\mathsf{Q.iso}^{-1}] \arrow[r, shift left=1.1ex, "\mathrm{B}_\pi"{name=A}] 
&\mathrm{B}\mathcal{P}\mbox{-}\mathsf{coalg}~[\mathsf{Q.iso}^{-1}] \arrow[l, shift left=1.1ex, "\mathbb{R}\Omega_\pi"{name=A}]
\arrow[r, shift left=1.5ex, "\mathbf{loose}"{name=A}]  \arrow[rd, shift left=0ex, shorten >=10pt, shorten <=10pt, "\mathrm{Inc}"{name=B}] 
&\mathbf{Coalg}_{\mathbf{bar}(\mathcal{P})}^{\mathbf{conil}}(\mathbf{D}(\kk)) \arrow[d, shift left=2ex, "\mathbf{Inc}"{name=F}] \arrow[l, shift left=.75ex, "\mathbf{strict}"{name=C}] \\
&
&\Omega\mathcal{P}^*\mbox{-}\mathsf{coalg}~[\mathsf{Q.iso}^{-1}]\arrow[u, shift left=1.1ex, "\mathbf{Sub}"{name=U}] \arrow[lu, shift left=2.5ex, shorten >=-5pt, shorten <=30pt, "\mathbb{R}\mathrm{Sub}"{name=D}] 
\end{tikzcd}
\]
where left adjoints are written on top. The composite horizontal adjunction does give back the enhanced bar-cobar adjunction $\mathbf{B}^{\mathbf{enh}} \dashv \bm{\Omega}^{\mathbf{enh}}$ but, surprisingly, what happens is that the adjunction $\mathbf{loose} \dashv \mathbf{strict}$ is not in general an equivalence. In \cref{subsection: derived inclusions and equivalences}, we construct counterexamples to the fully faithfulness of some of these functors, and we show however that the functor $\mathbf{loose}$ is indeed fully faithful on the essential image of homotopy complete $\mathcal{P}$-algebras by the bar construction $\mathrm{B}_\pi$, see Corollary \ref{cor: fully faithfulness of loose on homotopy complete}. This puts together several partial rectification results for conilpotent coalgebras present in the literature, for example what happens when $\mathcal{P}$ is connective and conilpotent $\mathrm{B}\mathcal{P}$-coalgebras are in degrees $\geq 1$. 

\subsection{A dual version of the enhanced bar-cobar adjunction.}
In order to make sense of the bottom horizontal adjunction in (\ref{general inclusion-restriction square}) at the $\infty$-categorical level, we construct a dual version of the enhanced bar-cobar adjunction for non-conilpotent coalgebras. Let $\mathscr{P}$ be an augmented $\infty$-operad enriched in $\mathbf{D}(\kk)$. In the same way that we can use point-set operads to encode non-conilpotent coalgebras, the same happens with $\infty$-operads. We refer to \cite[Section 3]{rectification} for more details about this. 

\medskip

Here also, the augmentation map $\mathscr{P} \longrightarrow \mathbb{1}$ induces an adjunction $\mathbf{triv} \dashv \mathbf{prim}$ between the $\infty$-category of $\mathscr{P}$-coalgebras and the derived $\infty$-category of chain complexes over $\kk$. The functor $\mathbf{triv}$ endows any chain complex with a trivial $\mathscr{P}$-coalgebra structure and its right adjoint $\mathbf{prim}$ is the functor of derived primitives of a $\mathscr{P}$-coalgebra. Its homology groups should be understood as (topological) coAndré-Quillen homology groups for types of coalgebras, of which (topological) coHochschild homology groups in the sense of \cite{HessShipley21} and \cite{BayindirPeroux23} are particular examples. This adjunction also factors universally through algebras over the induced monad $\mathbf{Q}_{\mathscr{P}} = \mathbf{prim} \circ \mathbf{triv}$ on $\mathbf{D}(\kk)$, giving the adjunction
\[
\begin{tikzcd}[column sep=5pc,row sep=3pc]
         \mathbf{Coalg}_{\mathscr{P}}(\mathbf{D}(\kk)) \arrow[r, shift left=1.1ex, "\bm{\widehat{\mathrm{B}}}^{\mathbf{enh}}"{name=F}] &\mathbf{Alg}_{\mathbf{Q}_{\mathscr{P}}}(\mathbf{D}(\kk))~, \arrow[l, shift left=.9ex, "\bm{\widehat{\Omega}}^{\mathbf{enh}}"{name=U}] \arrow[phantom, from=F, to=U, , "\dashv" rotate=90]
\end{tikzcd}
\]
which we call the \textit{complete enhanced bar-cobar adjunction}. However, unlike in \cite{FrancisGaitsgory}, it is not straightforward to identify what this monad is. In order to do so, we first give a definition (\ref{def: infinity cat C-algebra}) of algebras over an enriched $\infty$-cooperad — which are the $\infty$-categorical analogue of the absolute algebras considered in \cite{absolutealgebras}. Then, we show that point-set algebras over a conilpotent dg cooperad up to quasi-isomorphisms do present their $\infty$-categorical counterparts in Corollary \ref{cor: rectifictation of dg C algebras}. This combination of new $\infty$-categorical definitions and point-set methods allows to show the following result. 

\begin{theoremintro}[Theorems \ref{thm: identification of the infinity monad Q} and \ref{thm: point-set models for complete enhanced bar-cobar adjunctions}]
Let $\C$ be a conilpotent dg cooperad satisfying finiteness assumptions \ref{assumption 1}. The monad $\mathbf{Q}_{\Omega\mathcal{C}}=\mathbf{prim} \circ \mathbf{triv}$ induced by $\mathbf{triv} \dashv \mathbf{prim}$ 
can be indentified with the monad
\[\widehat{\mathbf{S}}^c(\C) \simeq \prod_{n \geq 0} \left[\C(n), (-)^{\otimes n}\right]^{h\mathbb{S}_n}
\] induced by the underlying enriched $\infty$-cooperad of $\C$, which encodes $\infty$-categorical $\C$-algebras. Furthermore, the $\infty$-categorical adjunction induced by the complete bar-cobar adjunction 
composed with the right Bousfield localization with respect to quasi-isomorphisms
\[
\begin{tikzcd}[column sep=4pc,row sep=3pc]
           \Omega\mathcal{C}\mbox{-}\mathsf{coalg}~[\mathsf{Q.iso}^{-1}] \arrow[r, shift left=1.1ex, "\widehat{\Omega}_{\iota}"{name=F}] 
          &\C\mbox{-}\mathsf{alg}^{\mathsf{comp}}~[\mathsf{W}^{-1}] \arrow[l, shift left=.75ex, "\widehat{\mathrm{B}}_{\iota}"{name=U}] \arrow[r, shift left=1.1ex, "\mathrm{Id}"{name=A}]
          &\C\mbox{-}\mathsf{alg}^{\mathsf{comp}}~[\mathsf{Q.iso}^{-1}] \arrow[l, shift left=.75ex, "\mathbb{L}\mathrm{Id}"{name=B}]
            \arrow[phantom, from=F, to=U, , "\simeq" rotate=0] \arrow[phantom, from=A, to=B, , "\dashv" rotate=90]
\end{tikzcd}
\]
is naturally weakly equivalent to the $\infty$-categorical complete enhanced bar-cobar adjunction. 
\end{theoremintro}

\subsection{An \texorpdfstring{$\infty$}{infinity}-categorical inclusion-restriction square} Finally, we put together all of our previous constructions at the $\infty$-categorical level to obtain another inclusion-restriction square.

\begin{theoremintro}[Theorem \ref{thm: infinity categorical inclusion-restriction square}]
Let $\mathscr{P}$ be a reduced $\infty$-operad enriched in $\mathbf{D}(\kk)$, with arity-wise bounded and finite dimensional homology groups. The following diagram of adjunctions 

\[
\begin{tikzcd}[column sep=5pc,row sep=5pc]
\mathbf{Alg}_{\mathscr{P}}(\mathbf{D}(\kk)) \arrow[r,"\bm{\mathrm{B}}^{\mathbf{enh}}"{name=B},shift left=1.1ex] \arrow[d,"\mathbf{cAbs}"{name=SD},shift left=1.1ex ]
&\mathbf{Coalg}_{\mathbf{bar}(\mathscr{P})}^{\mathbf{conil}}(\mathbf{D}(\kk)), \arrow[d,"\mathbf{Inc}"{name=LDC},shift left=1.1ex ] \arrow[l,"\bm{\Omega}^{\mathbf{enh}}"{name=C},,shift left=1.1ex]  \\
\mathbf{Alg}_{\mathscr{P}^*}(\mathbf{D}(\kk)) \arrow[r,"\bm{\widehat{\mathrm{B}}}^{\mathbf{enh}}"{name=CC},shift left=1.1ex]  \arrow[u,"\mathbf{Res}"{name=LD},shift left=1.1ex ]
&\mathbf{Coalg}_{\mathbf{cobar}(\mathscr{P}^*)}(\mathbf{D}(\kk)) ~, \arrow[l,"\bm{\widehat{\Omega}}^{\mathbf{enh}}"{name=CB},shift left=1.1ex] \arrow[u,"\mathbf{Sub}"{name=TD},shift left=1.1ex] \arrow[phantom, from=SD, to=LD, , "\dashv" rotate=180] \arrow[phantom, from=C, to=B, , "\dashv" rotate=-90]\arrow[phantom, from=TD, to=LDC, , "\dashv" rotate=180] \arrow[phantom, from=CC, to=CB, , "\dashv" rotate=-90]
\end{tikzcd}
\] 
commutes in the following sense: left adjoints from the top left to the bottom right are naturally weakly equivalent. 
\end{theoremintro}

Interestingly, while this commutative square does not coincide with the commutative square induced at the $\infty$-categorical level by (\ref{general inclusion-restriction square}), its commutativity does follow from the point-set methods that we have developed throughout this paper. And while we expect this square to hold in a much more general context than $\mathbf{D}(\kk)$, so far the only way we know to reach such a purely $\infty$-categorical statement about Koszul duality is via point-set methods.

\subsection{Acknowledgements}
We wish to warmly thank Grégory Ginot for several discussions that led to this paper. The first author also wishes to thank Gijs Heuts and Brice Le Grignou for useful discusions on these topics.

\subsection{Conventions}
Regarding notation, we will adopt the following conventions. 

\medskip

\begin{itemize}
    \item Let $\kk$ be a field, of any characteristic unless stated otherwise. Our base 1-category will be the category of chain complexes over $\kk$, together with its closed symmetric monoidal structure given by the tensor product $-\otimes-$ of chain complexes and the Koszul sign convention. We will denote the internal hom of this category by $[-,-]$. We adopt the \textit{homological convention}, differentials will be of degree $-1$. We denote this 1-category by $\mathsf{Ch}(\kk)$. 
    
    \medskip
    
    \item In general, dg operads in $\mathsf{Ch}(\kk)$ will be denoted by $\mathcal{P}$ and dg cooperads by $\C$. For a dg operad $\mathcal{P}$, we denote the category of dg $\mathcal{P}$-algebras by $\mathcal{P}\mbox{-}\mathsf{alg}$ and of dg $\mathcal{P}$-coalgebras by $\mathcal{P}\mbox{-}\mathsf{coalg}$. Similarly for algebras and coalgebras over a dg cooperad.
    
    \medskip
    
    \item Let $\mathsf{C}$ be a category and let $\mathsf{W}$ be a class of arrows in $\mathsf{C}$. We will denote $\mathsf{C}~[\mathsf{W}^{-1}]$ the $\infty$-category obtained by localizing $\mathsf{C}$ at $\mathsf{W}$. When working at the $\infty$-categorical level, limits and colimits should be understood as meaning homotopy limits and colimits. If
\[
\mathsf{F}: \mathsf{C} \longrightarrow \mathsf{D}
\]
is a functor between categories which sends a class of arrows $\mathsf{W}_{\mathsf{C}}$ in $\mathsf{C}$ to a class of arrows $\mathsf{W}_{\mathsf{D}}$ in $\mathsf{D}$, we still denote 
\[
\mathsf{F}: \mathsf{C}~[\mathsf{W}^{-1}_{\mathsf{C}}] \longrightarrow \mathsf{D}~[\mathsf{W}^{-1}_{\mathsf{D}}]
\]
the induced functor at the $\infty$-categorical level. However, we will add $\mathbb{L}\mathsf{F}$ and $\mathbb{R}\mathsf{F}$ to the left (resp. right) derived functors of $\mathsf{F}$ when it is a left (resp. right) Quillen functor which does not preserve weak equivalences in general. If $\mathsf{C}$ is a 1-category that we consider as an $\infty$-category via the nerve functor $\mathcal{N}$, we will still denote it by $\mathsf{C}$ instead of $\mathcal{N}(\mathsf{C})$.

    \medskip
    
    \item Our base $\infty$-category will be the category of chain complexes over $\kk$ localized at the class of all quasi-isomorphisms $\mathsf{Ch}(\kk)~[\mathsf{Q.iso}^{-1}]$, which a symmetric monoidal $\infty$-category. We will denote this $\infty$-category directly by $\mathbf{D}(\kk)$.
     
    \medskip
    
    \item In general, an enriched $\infty$-operad in $\mathbf{D}(\kk)$ will be denoted by $\mathscr{P}$, and an enriched $\infty$-co\-operad by $\mathscr{C}$. Given any enriched $\infty$-operad $\mathscr{P}$, we denote the $\infty$-category of $\mathscr{P}$-algebras in the base $\infty$-category $\mathbf{D}(\kk)$ by $\mathbf{Alg}_{\mathscr{P}}(\mathbf{D}(\kk))$, and the $\infty$-category of $\mathscr{P}$-coalgebras in $\mathbf{D}(\kk)$ by $\mathbf{Coalg}_{\mathscr{P}}(\mathbf{D}(\kk))$. Similarly for algebras and coalgebras over an enriched $\infty$-co\-operad.

\end{itemize}


\section{Recollections on algebras and coalgebras}


The goal of this section is to introduce both the point-set and the $\infty$-categorical framework that we are going to use in the rest of this paper.

\medskip

On the point-set side, we recall the notions of operads and cooperads, as well as the categories of algebras and coalgebras that one can define over them. Then we explain the operadic bar-cobar adjunction between augmented dg operads and conilpotent dg cooperads and how one obtains induced bar-cobar adjunctions between their respective categories of algebras and coalgebras. Finally, we also explain the different model structures that each of these categories can be endowed with. 

\medskip

On the $\infty$-categorical side, we use the formalism of symmetric sequences developed in \cite[Section 4.1.2]{BrantnerPhD}, where one can define (enriched) $\infty$-operads (resp. (enriched) $\infty$-cooperads) as monoids (resp. comonoids) in the $\infty$-category of symmetric sequences. See also \cite{Haugseng22,Haugseng19} for a different approach to this formalism. We first recall the definitions of algebras over $\infty$-operads and (divided powers, conilpotent) coalgebras over $\infty$-cooperads. Then, we recall the main points of the definition of coalgebras over an operad introduced in \cite{rectification}. And finally, we introduce the notion of an algebra over a $\infty$-cooperad in the $\infty$-categorical context, in a spirit similar to \cite{linearcoalgebras}. See also \cite{absolutealgebras} for more details on algebras over cooperads.

 \subsection{Point-set preliminaries}

As in \cite[Section 1]{rectification}, we consider the monoidal category $\mathsf{sSeq}(\mathsf{Ch}(\kk))$ of dg symmetric sequences, whose objects are collections $\{M(n)\}_{n \geq 0}$ of chain complexes endowed with an action of $\mathbb{S}_n$ for all $n \geq 0$, with monoidal structure given by the \textit{composition product}. 

\medskip

A dg operad is a monoid object $\operad P$ in $\mathsf{sSeq}(\mathsf{Ch}(\kk))$. Dually, a cooperad is a comonoid object in $\mathsf{sSeq}(\mathsf{Ch}(\kk))$. In  \cite[Section 1]{rectification}, we defined algebras and coalgebras over a dg operad $\mathcal P$ in terms of the \emph{endomorphism operad} and \emph{coendomorphism operad} associated with a dg vector space. We now want to firstly give an equivalent, monadic, definition of the category of $\mathcal P$-algebras. The monadic definition naturally dualizes instead to a comonadic definition of a \emph{coalgebra over a cooperad}.

\subsubsection{Algebras over operads and coalgebras over cooperads.} To any dg symmetric sequence one can associate an endofunctor in the category of chain complexes via the \textit{Schur realization functor}:
\[
\begin{tikzcd}[column sep=4pc,row sep=0pc]
\mathrm{S}(-) : \mathsf{sSeq}(\mathsf{Ch}(\kk)) \arrow[r]
&\mathsf{End}(\mathsf{Ch}(\kk)) \\
M \arrow[r,mapsto]
&\mathrm{S}(M) \coloneqq \displaystyle \bigoplus_{n \geq 0} M(n) \otimes_{\mathbb{S}_n} (-)^{\otimes n}~.
\end{tikzcd}
\]

The Schur realization functor $\mathrm{S}(-)$ is strong monoidal. Thus, for any dg operad $\mathcal{P}$ its Schur functor $\mathrm{S}(\mathcal{P})$ is a monad in chain complexes, and for any dg cooperad $\C$ its Schur functor $\mathrm{S}(\C)$ is a comonad in chain complexes. 

\begin{definition}[dg $\PP$-algebra]
Let $\PP$ be a dg operad. A dg $\PP$-\textit{algebra} is an algebra over the monad $\mathrm{S}(\mathcal{P})$. 
\end{definition}

\begin{definition}[dg $\C$-coalgebra]\label{def: C-coalgebra}
Let $\C$ be a dg cooperad. A dg $\C$-\textit{coalgebra} $W$ is a coalgebra over the comonad $\mathrm{S}(\C)$. 
\end{definition}

\begin{remark}[About coalgebras over cooperads]
    The structural map of a dg $\C$-coalgebra $W$ 
    \[
    \Delta_W: W \longrightarrow \bigoplus_{n \geq 0} \C(n) \otimes_{\mathbb S_n} W^{\otimes n}
    \]lands in the direct sum. Therefore any element $w$ in $W$ can only be decomposed into a \textit{finite sum}. This implies that dg $\C$-coalgebras always satisfy a \textit{conilpotency} condition. Furthermore, the fact that $\Delta_W$ lands in the \textit{coinvariants} on the right-hand side implies that \textit{divided power operations} will appear in this type of structure. Therefore the definition of a dg $\C$-coalgebra encodes divided powers conilpotent types of coalgebraic structures.
\end{remark}

\subsubsection{Algebras over a cooperad}\label{subsection: dual Schur functor} There is a \textit{dual Schur realization functor} which is given by
\[
\begin{tikzcd}[column sep=4pc,row sep=0pc]
\widehat{\mathrm{S}}^c : \mathsf{sSeq}(\mathsf{Ch}(\kk))^{\mathsf{op}} \arrow[r]
&\mathsf{End}(\mathsf{Ch}(\kk)) \\
M \arrow[r,mapsto]
&\widehat{\mathrm{S}}^c(M) \coloneqq \displaystyle \prod_{n \geq 0} [M(n),(-)^{\otimes n}]^{\mathbb{S}_n}~.
\end{tikzcd}
\]
The functor $\widehat{\mathrm{S}}^c(-)$ is lax monoidal by \cite[Corollary 3.4]{linearcoalgebras}: therefore any dg cooperad induces a monad on the category of chain complexes. 

\begin{definition}[dg $\C$-algebra]\label{def dg C algebra}
A \textit{dg} $\mathcal{C}$\textit{-algebra} $B$ is an algebra  $(B,\gamma_B,d_B)$ over the monad $\widehat{\mathrm{S}}^c(\C)$. 
\end{definition}

\begin{remark}[About algebras over cooperads, and absolute algebras]
The structural data of a dg $\C$-algebra is a morphism 
\[
\gamma_B: \prod_{n \geq 0}[\C(n),B^{\otimes n}]^{\mathbb{S}_n} \longrightarrow B~. 
\]
This morphism associates to any formal power series of operations in $\C$ labelled by elements of $B$ a well-defined image in $B$, without supposing an underlying topology on $B$. In \cite{absolutealgebras}, the second named author introduced the terminology of \emph{absolute algebras} for what is here called an algebra over a cooperad. For example, an ``absolute Lie algebra'' is in the present terminology simply an algebra over the cooperad $\mathcal{L}\mathit{ie}^*$, the linear dual of $\mathcal{L}\mathit{ie}$. One defines e.g.\ ``absolute associative algebras'' and ``absolute commutative algebras'' similarly. Notice that since the above definition uses \textit{invariants}, these operations have \textit{divided power operations} by default as well. We refer to \cite{absolutealgebras} for more on this type of structure.
\end{remark}

\subsubsection{The canonical filtration on algebras over cooperads.} The notion of an algebra over a cooperad admits a further description in the case where the cooperad is \textit{conilpotent}. Recall that, for any dg cooperad, one can define the subcooperad ${R}_p \C$ of operations in $\C$ which admit at most $p$-fold iterated non-trivial partial decompositions, and that $\C$ is said to be \textit{conilpotent} if
\[
\colim_{p} {R}_p \C \cong \C~. 
\]
See \cite[Section 1.2]{Curvedcalculus} for this particular definition of conilpotency, which is in fact equivalent to the one given in \cite[Chapter 5.8]{LodayVallette}. 

\begin{definition}[Canonical filtration on a dg $\mathcal{C}$-algebra]\label{def canonical filtration}
Let $\C$ be a conilpotent dg cooperad and let $B$ be a dg $\C$-algebra. The \textit{canonical filtration} of $B$ is the decreasing filtration given by 
\[ 
{W}^p B \coloneqq \mathrm{Im}\left(\gamma_B \circ \widehat{\mathrm{S}}^c(\pi_p)(\mathrm{id}_B): \widehat{\mathscr{S}}^c(\C / {R}_p \C)(B) \longrightarrow B \right)
\]
where ${R}_p \C$ denotes the $p$-th term of the coradical filtration, for all $p \geq 0~.$ Notice that we have 
\[
B = {W}^0 B \supseteq {W}^1 B \supseteq {W}^2 B \supseteq \cdots \supseteq  \cdots.
\]
\end{definition}
	
Each step ${W}^p B$ of the canonical filtration of a dg $\C$-algebra $B$ is an ideal, meaning the quotient $B/{W}^k B$ has a canonical dg $\C$-algebra structure induced by the original dg $\C$-algebra structure of $B$. 

\begin{definition}[Completion of a dg $\C$-algebra]
Let $B$ be a dg $\C$-algebra. Its \textit{completion} is given by the limit
\[ 
\widehat{B} \coloneqq \lim_{p} B/{W}^p B~,
\]
taken in the category of dg $\C$-algebras.
\end{definition}

Any dg $\C$-algebra $B$ comes equipped with a canonical morphism of dg $\C$-algebras
\[
\varphi_B: B \longrightarrow \widehat{B}
\]
This morphism is always an epimorphism, see \cite[Proposition~4.24]{linearcoalgebras}. We define \textit{complete} dg $\C$-algebras as the full subcategory of dg $\C$-algebras for which this map is an isomorphism. It is a reflective subcategory of dg $\C$-algebras, where the reflector is given by the completion functor. Notice that any free dg $\C$-algebra is in particular complete. 

\subsubsection{Coalgebras over an operad}
For any dg operad $\PP$, its dual Schur functor $\widehat{\mathrm{S}}^c(\PP)$ fails to be a comonad since the dual Schur functor is only lax monoidal. However, this lax structure endows this endofunctor with a lax comonad structure over which one can define coalgebras. 

\begin{definition}[dg $\mathcal{P}$-coalgebra]
A \textit{dg} $\mathcal{P}$\textit{-coalgebra} $C$ is a chain complex $(C,d_C)$ endowed with a structural map
\[
\Delta_C: C \longrightarrow \displaystyle \prod_{n \geq 0} [\mathcal{P}(n),C^{\otimes n}]^{\mathbb{S}_n}~,
\]
such that the following diagram commutes 
\[
\begin{tikzcd}[column sep=4.5pc,row sep=3pc]
C \arrow[r,"\Delta_C"] \arrow[d,"\Delta_C",swap] 
&\widehat{\mathrm{S}}^c(\PP)(C) \arrow[r,"\widehat{\mathrm{S}}^c(\mathrm{id})(\Delta_C)"]
&\widehat{\mathrm{S}}^c(\PP) \circ \widehat{\mathrm{S}}^c(\PP)(C) \arrow[d,"\varphi_{\PP,\PP}(C)"] \\
\widehat{\mathrm{S}}^c(\PP)(C) \arrow[rr,"\widehat{\mathrm{S}}^c(\gamma)(\mathrm{id})"]
&
&\widehat{\mathrm{S}}^c(\PP \circ \PP)~.
\end{tikzcd}
\]
\end{definition}

The data of a dg $\mathcal{P}$-coalgebra $C$ is equivalent to the data of a morphism of dg operads $\mathcal{P} \longrightarrow \mathrm{coEnd}_C$, where $\mathrm{coEnd}_C$ denotes the coendomorphism operad of $C$, given by the dg $\mathbb{S}_n$-module
\[
\mathrm{coEnd}_C(n) \coloneqq [C,C^{\otimes n}]~,
\]
for all $n \geq 0$, where the operad structure is given by the usual composition of maps.

\begin{remark}[About coalgebras over operads]
The structure morphism of a dg $\mathcal{P}$-coalgebra $C$ is equivalent to a family of maps 
\[
\Delta_C^n: C \longrightarrow [\mathcal{P}(n),C^{\otimes n}]^{\mathbb{S}_n}~,
\]
for all $n \geq 0$. These maps assign to every element $c$ in $C$ its decompositions with respect to all the operations in $\mathcal{P}(n)$, which lie in $C^{\otimes n}$. This definition encodes the usual types of coalgebraic structures one encounters. For instance, if $\mathcal{P}= \mathcal{C}om$ is the commutative operad, dg $\mathcal{C}om$-coalgebras are precisely non-unital dg commutative coalgebras. If $\mathcal{P}= \mathcal{A}ss$ is the associative operad, then dg $\mathcal{A}ss$-coalgebras are non-unital dg associative coalgebras. In both cases, this gives the categories of all coalgebras without imposing a conilpotency condition.
\end{remark}

\begin{theorem}[{\cite[Theorem 2.7.11]{anelcofree2014}}]\label{Thm: existence of the cofree coalgebra}
Let $\PP$ be a dg operad. The category of dg $\PP$-coalgebras is comonadic. There exists a comonad $(\mathrm{L}(\PP), \omega, \zeta)$ in the category of chain complexes such that the category of $\mathrm{L}(\PP)$-coalgebras is equivalent to the category of dg $\PP$-coalgebras.
\end{theorem}

In particular, this entails the existence of a cofree dg $\PP$-coalgebra. While in the general setting of \cite{anelcofree2014}, the construction of the comonad $\mathrm{L}(\PP)$ is given by an infinite recursion, the construction of $\mathrm{L}(\PP)$ in the category of chain complexes over a field $\kk$ stops at the first step. It is thus given by the following pullback 
\[
\begin{tikzcd}[column sep=3pc,row sep=3pc]
\mathrm{L}(\PP) \arrow[r,"p_2"] \arrow[d,"p_1",swap,rightarrowtail] \arrow[dr, phantom, "\ulcorner", very near start]
&\widehat{\mathrm{S}}^c(\PP) \circ \widehat{\mathrm{S}}^c(\PP) \arrow[d,"\varphi_{\PP,\PP}",rightarrowtail] \\
\widehat{\mathrm{S}}^c(\PP) \arrow[r,"\widehat{\mathrm{S}}^c(\gamma)"]
&\widehat{\mathrm{S}}^c(\PP \circ \PP)~.
\end{tikzcd}
\]

\subsubsection{Operadic Koszul duality}\label{subsubsection: Point-set operadic Koszul duality} Let us now recall the bar-cobar adjunction that exists between dg operads and conilpotent dg cooperads.

\begin{definition}[Reduced dg operad]
    A dg operad $\mathcal P$ is \emph{reduced} if the unit morphism $I \to \mathcal P$ is an isomorphism in arities 0 and 1. Similarly, a cooperad $\mathcal C$ is reduced if its counit $\mathcal C \to I $ is an isomorphism in arity 0 and 1.
\end{definition}
    
From now on, we shall always quietly \emph{assume that all dg operads and dg cooperads are reduced}. Under this hypothesis, there is an adjunction, going back to Ginzburg--Kapranov \cite{ginzburgkapranov} and Getzler--Jones \cite{GetzlerJones94}
\[
\begin{tikzcd}[column sep=5pc,row sep=3pc]
          \mathsf{Operads}   \arrow[r, shift left=1.1ex, "\mathrm{B}"{name=F}] & \mathsf{Cooperads} , \arrow[l, shift left=.75ex, "\Omega"{name=U}]
            \arrow[phantom, from=F, to=U, , "\dashv" rotate=90]
\end{tikzcd}
\]

between dg operads and dg cooperads. Notice that reduced dg cooperads are automatically conilpotent.  

\medskip

A \textit{twisting morphism} $\alpha: \C \longrightarrow \mathcal{P}$ is a graded morphism of degree $-1$ which satisfies the equation
\[
\alpha \circ d_\C  -d_{\mathcal{P}} \circ \alpha+ \alpha \star \alpha = 0~, 
\]
where $\alpha \star \alpha$ is given by the composition
\[
\begin{tikzcd}[column sep=3.5pc,row sep=0.5pc]
\C \arrow[r,"\Delta_{(1)}"]
&\C \circ_{(1)} \C \arrow[r,"\alpha~ \circ_{(1)}~ \alpha"] 
&\mathcal{P} \circ_{(1)} \mathcal{P} \arrow[r,"\gamma_{(1)}"]
&\mathcal{P}~. 
\end{tikzcd} 
\]
Here $\Delta_{(1)}$ and $\gamma_{(1)}$ refer to the infinitesimal decomposition of $\C$ and to the infinitesimal composition of $\mathcal{P}$, respectively. The equation above corresponds to the Maurer-Cartan equation in the convolution pre-Lie algebra constructed from $\C$ and $\mathcal{P}$. The bar-cobar adjunction represents and corepresents twisting morphisms between dg cooperads and dg operads:
\[
\mathrm{Hom}_{\mathsf{Operads}}(\Omega \C,\mathcal P) \cong \mathrm{Tw}(\C,\mathcal P)\cong \mathrm{Hom}_{\mathsf{Cooperads}}(\C,\mathrm{B}\mathcal P)~.
\]
We refer to \cite[Chapter 6]{LodayVallette} for more details. 

\subsubsection{Classical bar-cobar adjunction induced by a twisting morphism}\label{subsubsection: classical bar-cobar adjunction}
Any twisting morphism $\alpha: \C \longrightarrow \mathcal{P}$ induces a \textit{bar-cobar adjunction} relative to $\alpha$ 
\[
\begin{tikzcd}[column sep=5pc,row sep=3pc]
          \mathcal{C}\mbox{-}\mathsf{coalg} \arrow[r, shift left=1.1ex, "\Omega_{\alpha}"{name=F}] & \mathcal{P}\mbox{-}\mathsf{alg}, \arrow[l, shift left=.75ex, "\mathrm{B}_{\alpha}"{name=U}]
            \arrow[phantom, from=F, to=U, , "\dashv" rotate=-90]
\end{tikzcd}
\]

between the category of dg $\C$-coalgebras and the category of dg $\mathcal{P}$-algebras. For any dg $\mathcal{P}$-algebra $A$, the bar construction $\mathrm{B}_{\alpha}A$ is given by taking the cofree $\C$-coalgebra on $A$ and building a differential using the algebra structure on $A$ and the twisting morphism $\alpha$. Similarly, for any dg $\C$-coalgebra $W$, the cobar construction $\Omega_{\alpha}W$ is given by taking the free $\mathcal{P}$-algebra on $W$ and building a differential using the coalgebra structure on $W$ and the twisting morphism $\alpha$. We refer to \cite[Chapter 11]{LodayVallette} for more details on these topics. 

\subsubsection{Complete bar-cobar adjunction induced by a twisting morphism}\label{subsubsection: complete bar-cobar adjunction} Any twisting morphism $\alpha: \C \longrightarrow \mathcal{P}$ between a conilpotent dg cooperad $\C$ and a dg operad $\mathcal{P}$ induces a second bar-cobar adjunction, which we call the \textit{complete bar-cobar adjunction} relative to $\alpha$
\[
\begin{tikzcd}[column sep=5pc,row sep=3pc]
            \PP\mbox{-}\mathsf{coalg} \arrow[r, shift left=1.1ex, "\widehat{\Omega}_{\alpha}"{name=F}] & \C\mbox{-}\mathsf{alg}^{\mathsf{comp}} \arrow[l, shift left=.75ex, "\widehat{\mathrm{B}}_{\alpha}"{name=U}]
            \arrow[phantom, from=F, to=U, , "\dashv" rotate=-90]
\end{tikzcd}
\]
between the category of dg $\PP$-coalgebras and the category of complete dg $\C$-algebras. Like in the previous case, this adjunction is given by considering free and cofree objects, and endowing them with the appropriate differential. See \cite{linearcoalgebras} or \cite{premierpapier} for more details.

\medskip

\subsubsection{Admissible and coadmissible operads}
Let $\mathcal{P}$ be a dg operad. It is natural to consider dg $\mathcal{P}$-algebras and dg $\mathcal{P}$-coalgebras up to quasi-isomorphisms. We say that $\mathcal{P}$ is \textit{admissible} if the projective model structure on chain complexes over $\kk$ can be transferred along the free-forgetful adjunction 
\[
\begin{tikzcd}[column sep=5pc,row sep=3pc]
          \mathcal{P}\mbox{-}\mathsf{alg} \arrow[r, shift left=1.1ex, "\mathrm{S}(\mathcal{P})(-)"{name=F}] &\mathsf{Ch}(\kk)~, \arrow[l, shift left=.75ex, "\mathrm{U}"{name=U}]
            \arrow[phantom, from=F, to=U, , "\dashv" rotate=-90]
\end{tikzcd}
\]
thus endowing the category of dg $\mathcal{P}$-algebras with a model structure where weak equivalences are given by quasi-isomorphisms and fibrations are given by degree-wise surjections. Any dg operad over a field of characteristic zero is admissible \cite{HinichModel}. 

\medskip

Dually, we say that $\mathcal{P}$ is \textit{coadmissible} if the injective model structure on chain complexes over $\kk$ can be transferred along the cofree-forgetful adjunction 
\[
\begin{tikzcd}[column sep=5pc,row sep=3pc]
          \mathcal{P}\mbox{-}\mathsf{coalg} \arrow[r, shift left=1.1ex, "\mathrm{L}(\mathcal{P})(-)"{name=F}] &\mathsf{Ch}(\kk)~, \arrow[l, shift left=.75ex, "\mathrm{U}"{name=U}]
            \arrow[phantom, from=F, to=U, , "\dashv" rotate=90]
\end{tikzcd}
\]
thus endowing the category of dg $\mathcal{P}$-coalgebras with a model structure where weak equivalences are given by quasi-isomorphisms and cofibrations are given by degree-wise injections. It is not true, even over a field of characteristic zero, that any dg operad is coadmissible, see \cite[Proposition 8.10]{linearcoalgebras} for a counter-example. However, cofibrant dg operads (and therefore dg operads of the form $\Omega \C$) are always coadmissible by standard arguments as in \cite{BergerMoerdijk}. See \cite{linearcoalgebras, premierpapier} for more details as well. 

\subsubsection{Model structures transferred along the classical bar-cobar adjunction}



Let $\mathcal{P}$ be a dg operad (which is always admissible in our context). The model structure transferred from chain complexes over $\kk$ presents the $\infty$-category of $\mathcal{P}$-algebras up to quasi-isomorphisms \cite{HinichRectification, Haugseng22}. 
It is sometimes convenient to transfer this model structure along a bar-cobar adjunction to a suitable category of coalgebras over a cooperad. In the good cases, when the twisting morphism $\alpha: \C \longrightarrow \mathcal{P}$ is said to be \textit{Koszul}, one obtains a Quillen equivalence between dg $\mathcal{P}$-algebras and dg $\C$-coalgebras with the transferred structure. We refer to \cite{Vallette20,Drummond-ColeHirsh16} for the following results in characteristic zero, and to \cite{premierpapier} to a generalization over a positive characteristic field. 


\begin{theorem}\label{thm: transferred model structure on conilpotent coalgebras}
Let $\C$ be a conilpotent dg cooperad. There exists a combinatorial model category structure on the category of dg $\C$-coalgebras, sometimes called the canonical model structure, given by the following classes of maps:

\begin{enumerate}
\item the class of weak equivalences is given by morphisms $f$ such that $\Omega_\alpha(f)$ is a quasi-isomorphism,

\item the class of cofibrations is given by degree-wise injections,

\item the class of fibrations is given by morphisms with the right lifting property with respect to acyclic cofibrations.
\end{enumerate}

Furthermore, if $\alpha: \C \longrightarrow \mathcal{P}$ is Koszul in the sense of \cite[Chapter 6]{LodayVallette}, the bar-cobar adjunction relative to $\alpha$ is a Quillen equivalence. 
\end{theorem}

\begin{example}
The universal twisting morphisms $\iota: \C \longrightarrow \Omega \C$ and $\pi: \mathrm{B}\mathcal{P} \longrightarrow \mathcal{P}$ (induced by the unit and the counit of the operadic bar-cobar adjunction) are both Koszul twisting morphisms. 
\end{example}

\begin{remark}
Over a positive characteristic field, one needs to add the hypothesis that the conilpotent dg cooperad $\C$ is furthermore \textit{quasi-planar} in the sense of \cite{premierpapier}, and one has to restrict to the twisting morphism $\iota: \C \longrightarrow \Omega \C$. However, if our goal is to understand the $\infty$-category of dg $\mathcal{P}$-algebras on a $\mathbb{S}$-cofibrant operad $\mathcal{P}$, then we can always replace $\mathcal{P}$ with a dg operad of the form $\Omega \C$, where $\C$ is a quasi-planar conilpotent dg cooperad, without changing the underlying $\infty$-category of dg $\mathcal{P}$-algebras. A canonical choice is given by $\C = \mathrm{B}(\mathcal{P} \otimes \mathcal{E})$, which is indeed quasi-planar for any dg operad $\mathcal{P}$, and $\Omega\mathrm{B}(\mathcal{P} \otimes \mathcal{E})$ is always quasi-isomorphic to $\mathcal P$.
\end{remark}

Notice that in Theorem \ref{thm: transferred model structure on conilpotent coalgebras}, the class of weak equivalences is \textit{strictly} included in the class of quasi-isomorphisms of dg $\C$-coalgebras, unless the twisting morphism is the terminal one given by $\C \longrightarrow \mathcal{I}$. Localizing at this class of weak equivalences does not present an intrinsic $\infty$-category of coalgebras in $\mathbf{D}(\kk)$; on the contrary, since the bar-cobar adjunction is (in the good cases) a Quillen equivalence, it provides us with another presentation of the underlying $\infty$-category of dg $\mathcal{P}$-algebras. However, as shown in \cite{Drummond-ColeHirsh16}, the model category of dg $\C$-coalgebras admits a left Bousfield localization with respect to quasi-isomorphisms. 

\begin{proposition}\label{Prop: left Bousfield localization of conilpotent coalgebras}
Let $\C$ be a conilpotent dg cooperad. There exists a combinatorial model structure on dg $\operad C$-coalgebras transferred from chain complexes, determined by the following classes of morphisms
		
    \begin{enumerate}
        \item the class of cofibrations is given by degree-wise injections,

        \item the class of weak equivalences is given by quasi-isomorphisms,

        \item the class of fibrations is determined by the right-lifting property against all acyclic cofibrations.
    \end{enumerate}
    
Moreover, this is a left Bousfield localisation of the canonical model structure transferred from dg $\Omega \C$-algebras (which coincides with the model structure transferred along any Koszul twisting morphism). This means that the identity functor of dg $\C$-coalgebras, where at the source they are endowed with the canonical model structure, and at the target with the quasi-isomorphisms model structure, is a left Quillen functor. 
\end{proposition}

\begin{remark}
Over a positive characteristic field, one needs again to add the extra hypothesis that $\C$ is \textit{quasi-planar}. See \cite[Section 5.6]{premierpapier} for more details. 
\end{remark}

\subsubsection{Model structures transferred along the complete bar-cobar adjunction} As mentioned before, not all operads are coadmissible, even over a field of characteristic zero. Furthermore, even when they are, they need not present their $\infty$-categorical counterparts, as for instance shown in \cite{Péroux22}. Nevertheless, one can solve these problems by working with cofibrant dg operads, as they do present their $\infty$-categorical counterparts, see \cite{rectification}. In fact, without any loss of generality, we may assume that the cofibrant dg operad is of the form $\Omega \C$, where $\C$ is a conilpotent dg cooperad. In this context, dg $\Omega \C$-coalgebras up to quasi-isomorphisms present the correct $\infty$-category. Furthermore, they admit a transferred model structure from chain complexes. Indeed, since $\Omega \C$ is cofibrant it is in particular coadmissible. 
It is sometimes convenient to transfer this model structure along the complete bar-cobar adjunction in order to obtain a Quillen equivalent model structure on dg $\C$-algebras \cite{linearcoalgebras}. 

\begin{theorem}
Let $\C$ be a conilpotent dg cooperad. There exists a combinatorial model category structure on the category of complete dg $\C$-algebras given by the following classes of maps:
\begin{enumerate}
\item the class of weak equivalences is given by morphisms $f$ such that $\widehat{B}_\iota(f)$ is a quasi-isomorphism,

\item the class of fibrations is given by degree-wise surjections,

\item the class of cofibrations is given by morphisms with the left-lifting property with respect to acyclic fibrations.
\end{enumerate}

Furthermore, when endowed with this model structure, the complete bar-cobar adjunction relative to $\iota: \C \longrightarrow \Omega \C$ becomes a Quillen equivalence. 
\end{theorem}

\begin{remark}
This result admits a generalization for a quasi-planar conilpotent dg cooperad over a field of positive characteristic, see \cite[Section 6 and 7]{premierpapier}.
\end{remark}

Again, notice that the class of weak equivalences is \textit{strictly} included in the class of quasi-isomorph\-isms of complete dg $\C$-algebras. Localizing at this class of weak equivalences does not present an $\infty$-category of algebras in $\mathbf{D}(\kk)$. On the contrary, since the complete bar-cobar adjunction is a Quillen equivalence, it gives another presentation of the underlying $\infty$-category of dg $\Omega\C$-coalgebras. However, complete dg $\C$-algebras admit a right Bousfield localization with respect to quasi-iso\-morph\-isms. 

\begin{proposition}\label{Prop: right Bousfield localization of absolute algebras}
Let $\C$ be a conilpotent dg cooperad. There exists a combinatorial model structure on qp-complete dg $\operad C$-algebras transferred from chain complexes, determined by the following classes of morphisms
		
    \begin{enumerate}
        \item the class of fibrations is given by degree-wise surjections,

        \item the class of weak equivalences is given by quasi-isomorphisms,

        \item the class of cofibrations is determined by the left-lifting property against all acyclic fibrations.
    \end{enumerate}
   
Moreover, this is a right Bousfield localisation of the canonical model structure transferred from dg $\Omega \C$-coalgebras. That is, the identity functor of dg $\C$-coalgebras, where at the source they are endowed with the canonical model structure, and at the target with the quasi-isomorphisms model structure, is a right Quillen functor. 
\end{proposition}

\begin{remark}
See \cite[Section 7.5]{premierpapier} for the general statement involving the quasi-planar hypothesis, which reduces to the above proposition over a characteristic zero field. 
\end{remark}

\subsection{$\infty$-categorical preliminaries}
In this section, we work with the formalism of $\infty$-categorical symmetric sequences developed in \cite[Section 4.1.2]{BrantnerPhD}. See also \cite[Section 5.2.4]{ShiPhD} for a detailed exposition of these ideas. We use this formalism to define both algebras and coalgebras over enriched $\infty$-(co)operads. Note that by the recent work of \cite{arakawa2026}, this formalism of enriched $\infty$-operads is equivalent to the one developped in \cite{ChuHaugseng20}.

\subsubsection{Symmetric sequences, operads and cooperads.}
For simplicity, we work over the base symmetric monoidal $\infty$-category of chain complexes over $\kk$ up to quasi-isomorphisms $\mathbf{D}(\kk)$, although these definitions make sense in a more general context of a rigidly compactly generated symmetric monoidal $\infty$-category \cite[Section 6]{rectification}. 

\medskip

Let $\mathsf{Fin}^{\simeq}$ denote the 1-category of finite sets and bijections. We define the $\infty$-category of \textit{symmetric sequences} in $\mathbf{D}(\kk)$ as the $\infty$-category of functors from $\mathsf{Fin}^{\simeq}$ to $\mathbf{D}(\kk)$:
\[
\mathbf{sSeq}(\mathbf{D}(\kk)) \coloneqq \mathbf{Fun}(\mathsf{Fin}^{\simeq},\mathbf{D}(\kk))~. 
\]

For $M$ in $\mathbf{sSeq}(\mathbf{D}(\kk))$, we denote by $M(n)$ the evaluation of $M$ at the set $\underline{n}= \{1,\cdots,n\}$. The $\infty$-category of symmetric sequences in $\mathbf{D}(\kk)$ admits a monoidal structure given by the composition product $\circledcirc$, which for two symmetric sequences $M$ and $N$ is given pointwise by
\[
M \circledcirc N(n) \simeq \bigoplus_{k \geq 0} \left(\bigoplus_{\underline{n} = \sqcup_{i=1}^k S_i} M(k) \otimes N(S_1) \otimes \cdots \otimes N(S_k)\right)_{h\mathbb{S}_k} 
\]
as computed in \cite[Appendix A.1.]{ShiPhD}. Using this composition product, one can define operads as monoids in a monoidal $\infty$-category.

\begin{definition}[$\infty$-operad]
An $\infty$\textit{-operad} is a monoid object in the $\infty$-category of symmetric sequences with respect to the composition product. 
\end{definition}

\begin{definition}[augmented $\infty$-operad]
An \textit{augmented} $\infty$\textit{-operad} is an $\infty$-operad $\mathscr{P}$  together with a map of $\infty$-operads $\epsilon: \mathscr{P} \longrightarrow \mathbb{1}$. 
\end{definition}

\begin{definition}[$\infty$-cooperad]\label{def: infinity cooperad}
An $\infty$\textit{-cooperad} is a comonoid object in the $\infty$-category of symmetric sequences.
\end{definition}

\begin{definition}[coaugmented $\infty$-cooperad]
A \textit{coaugmented $\infty$-cooperad}  is an $\infty$-cooperad $\mathscr{C}$  with a map $\mu: \mathbb{1} \longrightarrow \operad C$ of $\infty$-cooperads.
\end{definition}


\subsubsection{Lurie's bar-cobar adjunction between $\infty$-operads and $\infty$-cooperads}\label{subsubsection: Lurie's bar-cobar adjunction for operads and cooperads}
In \cite{HigherAlgebra}, Lurie constructs a general bar-cobar adjunction between augmented monoids and coaugmented comonoids over a pointed monoidal $\infty$-category $\mathbf{C}$ admitting geometric realizations of simplicial objects and totalizations of cosimplicial objects. See \cite[Section 4.3]{HigherAlgebra} or Section \ref{Section: Lurie's bar-cobar}  
for more details. By applying this general construction to $\mathbf{C} = \mathbf{sSeq}(\mathbf{D}(\kk))$, we get the following adjunction 
\[
\begin{tikzcd}[column sep=5pc,row sep=3pc]
          \infty\mbox{-}\mathbf{Op}^{\mathbf{aug}} \arrow[r, shift left=1.1ex, "\mathbf{bar}"{name=F}] &\infty\mbox{-}\mathbf{Coop}^{\mathbf{coaug}}~, \arrow[l, shift left=.75ex, "\mathbf{cobar}"{name=U}] \arrow[phantom, from=F, to=U, , "\dashv" rotate=-90]
\end{tikzcd}
\]
which can be further restricted to an adjunction between reduced $\infty$-operads and reduced $\infty$-cooperads. The bar construction of an $\infty$-operad $\mathscr{P}$ is computed as the geometric realization of the simplicial object
\[
\mathbf{bar}(\mathscr{P}) \coloneqq \left|~\mathbf{bar}(\mathbb{1},\mathscr{P},\mathbb{1})_\bullet ~ \right| = \left|
\begin{tikzcd}
\mathbb{1} \arrow[r]
&\mathscr{P} \arrow[l,shift left=1.1ex] \arrow[l,shift right=1.1ex] \arrow[r,shift left=1ex] \arrow[r,shift right=1ex] 
&\mathscr{P} \circledcirc \mathscr{P} \arrow[l,shift left=2ex] \arrow[l,shift right=2ex] \arrow[l] \arrow[r,shift left= 2ex] \arrow[r,shift right=2ex]  \arrow[r]
&\cdots \arrow[l,shift left=3ex] \arrow[l,shift right=3ex]  \arrow[l,shift left=1ex] \arrow[l,shift right=1ex]
\end{tikzcd} \right|~. 
\]
Dually, the cobar construction of an $\infty$-operad $\mathscr{C}$ is computed as the totalization of the cosimplicial object
\[
\mathbf{cobar}(\mathscr{C}) \coloneqq \mathbf{Tot}\left(\mathbf{cobar}(\mathbb{1},\mathscr{C},\mathbb{1})_\bullet\right) = \mathbf{Tot}\left(
\begin{tikzcd}
\cdots \arrow[r,shift left= 2ex] \arrow[r,shift right=2ex]  \arrow[r]
&\mathscr{C} \circledcirc \mathscr{C} \arrow[l,shift left=3ex] \arrow[l,shift right=3ex]  \arrow[l,shift left=1ex] \arrow[l,shift right=1ex] \arrow[r,shift left=1ex] \arrow[r,shift right=1ex] 
&\mathscr{C} \arrow[l,shift left=2ex] \arrow[l,shift right=2ex] \arrow[l] \arrow[r]
&\mathbb{1}  \arrow[l,shift left=1.1ex] \arrow[l,shift right=1.1ex] 
\end{tikzcd} \right)~. 
\]
Furthermore, this adjunction becomes an equivalence of $\infty$-categories between reduced $\infty$-operads and reduced $\infty$-cooperads by \cite[Theorem 3.4]{heuts2024}. 

\begin{remark}\label{rmk: point-set bar and infinity cat bar}
The point-set operadic bar construction of Paragraph \ref{subsubsection: Point-set operadic Koszul duality} is a model for the $\infty$-categorical bar construction of the underlying $\infty$-operad. Indeed, the two constructions can be compared using the same arguments as in \cite[Section 5.4.2]{BrantnerPhD}. 
\end{remark} 

\subsubsection{Schur functors associated to enriched $\infty$-(co)operads and (co)algebras}\label{subsub: infinity categorical Schur functor}
As in the point-set setting, there is a functor 
\[
\mathbf{S}(-): \mathbf{sSeq}(\mathbf{D}(\kk)) \longrightarrow \mathbf{End}(\mathbf{D}(\kk))~,
\]
from the $\infty$-category of symmetric sequences in $\mathbf{D}(\kk)$ to the $\infty$-category of endofunctors in $\mathbf{D}(\kk)$. It associates to any symmetric sequence $M$ an endofunctor $\mathbf{S}(M)$, given by the formula
\[
\mathbf{S}(M) = \bigoplus_{n \geq 0} \left(M(n) \otimes (-)^{\otimes n}\right)_{h\mathbb{S}_n}~.
\]
The functor $\mathbf{S}(-)$ is symmetric monoidal, and therefore sends (enriched) $\infty$-operads to monads and (enriched) $\infty$-cooperads to comonads. 

\begin{definition}[$\mathscr{P}$-algebras] 
A \emph{$\mathscr{P}$-algebra} is an algebra over the monad $\mathbf{S}(\mathscr{P})$. 
\end{definition}

\begin{remark}
A $\mathscr{P}$-algebra is an object $A$ in $\mathbf{D}(\kk)$ equipped with a family of structural maps 
\[
\gamma_n^{A}: \left(\mathscr{P}(n) \otimes A^{\otimes n}\right)_{h\mathbb{S}_n} \longrightarrow A 
\]
for all $n \geq 1$ which satisfy associativity and unitality conditions up to higher coherences. 
\end{remark}

\begin{remark}[Point-set models for $\mathscr{P}$-algebras]
Any $\infty$-operad enriched in $\mathbf{D}(\kk)$ admits a point-set model, given by some dg operad in chain complexes over $\kk$. One sees this by comparing their respective $\infty$-categories of algebras in the following way: let $\mathcal{P}$ be a dg operad, projective as a dg symmetric sequence. Then there is an equivalence of $\infty$-categories
\[
\mathcal{P}\mbox{-}\mathsf{alg}~[\mathsf{Q.iso}^{-1}] \simeq \mathbf{Alg}_{\mathcal{P}}(\mathbf{D}(\kk))
\]
between dg $\mathcal{P}$-algebras up to quasi-isomorphism and algebras in $\mathbf{D}(\kk)$ over the induced enriched $\infty$-operad. This follows from the Barr--Beck--Lurie theorem, since the point-set Schur functor presents the $\infty$-categorical Schur functor when the underlying dg symmetric sequence of $\mathcal{P}$ is projective. We refer to \cite{HinichRectification, PavlovScholbach, Haugseng22} for similar statements. 
\end{remark}

\begin{definition}[Divided powers conilpotent $\mathscr{C}$-coalgebras]\label{def: infinity cat C-coalgebras}
A \emph{divided powers conilpotent $\mathscr{C}$-coalgebra} $W$ is a coalgebra over the comonad $\mathbf{S}(\mathscr{C})$. 
\end{definition}

\begin{remark}
A divided powers conilpotent $\mathscr{C}$-coalgebra is an object $W$ in $\mathbf{D}(\kk)$ equipped with a structural map 
\[
\Delta_W : W \longrightarrow \bigoplus_{n \geq 0} \left(\mathscr{C}(n) \otimes W^{\otimes n}\right)_{h\mathbb{S}_n}~,
\]
which again, satisfies some associativity condition up to higher coherences. Since  $\Delta_W$ lands in an infinite coproduct instead of an infinite product, this implies that decomposition has to have some finiteness condition, which explains the adjective conilpotent. Furthermore, this structural map lands on the homotopy orbits, hence this decomposition operation has a divided powers structure as well. 
\end{remark}

\begin{remark}
In \cite{heuts2024}, Heuts proposed a definition of divided powers $\mathscr{C}$-coalgebras for a cooperad $\mathscr{C}$ which, roughly, replaces the infinite coproduct in the above definition with an infinite product, thus removing the conilpotency constraint on the decomposition map. 
\end{remark}

\subsubsection{Dual Schur functor and algebras over an enriched $\infty$-cooperad}\label{subsubsection: infinity cat dual Schur functor}
There is another Schur functor associated to any symmetric sequence, which we call the \textit{dual Schur functor}. Associating to a symmetric sequence $M$ its dual Schur functor defines a functor 
\[
\begin{tikzcd}[column sep=3pc,row sep=0pc]
\widehat{\mathbf{S}}^c(-): \mathbf{sSeq}(\mathbf{D}(\kk))^{\mathsf{op}} \arrow[r]
&\mathbf{End}(\mathbf{D}(\kk)) \\
M \arrow[r,mapsto]
&\displaystyle \prod_{n \geq 0} \left[M(n),(-)^{\otimes n}\right]^{h\mathbb{S}_n} 
\end{tikzcd}
\]

from the $\infty$-category of symmetric sequences in $\mathbf{D}(\kk)$ to the $\infty$-category of endofunctors of $\mathbf{D}(\kk)$. Here $[-,-]$ denotes the self-enrichment of $\mathbf{D}(\kk)$, adjoint to the tensor product.

\begin{lemma}\label{dual schur is lax monoidal}
The functor $\widehat{\mathbf{S}}^c(-)$ is lax monoidal. 
\end{lemma}

\begin{proof}
We adapt the proof given in \cite[Proposition 5.1]{lucio2022integrationtheorycurvedabsolute} to the $\infty$-categorical setting. 
\end{proof}

\begin{remark}
The 1-categorical complete Schur functor $\widehat{\mathrm{S}}^c(-)$ is lax monoidal. The natural map 
\[
\mathsf{Alg}(\mathsf{End}(\mathsf{Ch}(\kk)))[\mathsf{Q.iso}^{-1}]\longrightarrow \mathbf{Alg}_{\mathbb{E}_1}(\mathbf{End}(\mathbf{D}(\kk)))
\]
sends it to $\widehat{\mathbf{S}}^c(-)$, hence its $\infty$-categorical counterpart is also lax monoidal, providing another proof of \cref{dual schur is lax monoidal}. 
\end{remark}

\cref{dual schur is lax monoidal} implies that an (enriched) $\infty$-cooperad $\mathscr{C}$ in $\mathbf{D}(\kk)$ is sent to a monad in $\mathbf{D}(\kk)$. Therefore, it also makes sense to define algebras over an $\infty$-cooperad. 

\begin{definition}[$\mathscr{C}$-algebra]\label{def: infinity cat C-algebra}
A $\mathscr{C}$-algebra is an algebra over the monad $\widehat{\mathbf{S}}^c(\mathscr{C})$.
\end{definition}

\begin{remark}
A $\mathscr{C}$-algebra is an object $B$ in $\mathbf{D}(\kk)$ equipped with a structural map
\[
\gamma^{B}: \prod_{n \geq 0} \left[\mathscr{C}(n),B^{\otimes n} \right]^{h\mathbb{S}_n} \longrightarrow B 
\]
which satisfies associativity and unitality conditions up to higher coherences. This structure is the analogue of the \textit{absolute algebra} structures studied in \cite{absolutealgebras} at the homotopical level. 
\end{remark}

\subsubsection{Coalgebras over operads}\label{subsubsection: infinity cat of coalgebras over an operad} We briefly recall the definition of the $\infty$-category of coalgebras over an (enriched) $\infty$-operad given in \cite[Section 2]{rectification}.

\medskip

The main obstruction to defining $\mathscr{P}$-coalgebras is that $\widehat{\mathbf{S}}^c(-)$ is only lax monoidal, therefore the image of an $\infty$-operad is not a priori a comonad. Following the ideas of \cite[Appendix A]{heuts2024}, we extended this dual Schur functor to the $\infty$-category of pro-objects $\mathbf{pro}(\mathbf{D}(\kk))$, such that the assignment 
\[
\widehat{\mathbf{S}}^c_{\mathrm{pro}}(-): \mathbf{sSeq}(\mathbf{D}(\kk))^{\mathsf{op}} \longrightarrow \mathbf{End}(\mathbf{pro}(\mathbf{D}(\kk))) 
\]
is now a strong monoidal functor. Therefore it makes sense to define $\mathscr{P}$-coalgebras in $\mathbf{pro}(\mathbf{D}(\kk))$ as coalgebras over this comonad. 

\medskip

Finally, if we denote by $c: \mathbf{D}(\kk) \longrightarrow \mathbf{pro}(\mathbf{D}(\kk))$ the canonical fully faithful inclusion given by sending objects in $\mathbf{D}(\kk)$ to their constant pro-object, one can pullback along this inclusion to obtain a well-behaved definition of $\mathscr{P}$-coalgebras in $\mathbf{D}(\kk)$

\begin{definition}[$\mathscr{P}$-coalgebra]\label{def: infinity categorical P-coalgebras}
The $\infty$-category of $\mathscr{P}$\textit{-coalgebras} as the following pullback 
\[
\begin{tikzcd}[column sep=3.5pc,row sep=3.5pc]
\mathbf{Coalg}_\mathscr{P}(\mathbf{D}(\kk)) \arrow[r] \arrow[d] \arrow[dr, phantom, "\ulcorner", very near start]
&\mathbf{Coalg}_{\widehat{\mathbf{S}}^c_{\mathrm{pro}}(\mathscr{P})}(\mathbf{pro}(\mathbf{D}(\kk))) \arrow[d,"\mathrm{U}"]\\
\mathbf{D}(\kk) \arrow[r,"c"] 
&\mathbf{pro}(\mathbf{D}(\kk))
\end{tikzcd}
\]
in the $\infty$-category of $\infty$-categories. 
\end{definition}

The $\infty$-category of $\mathscr{P}$-coalgebras is presentable and comonadic over the base $\infty$-category $\mathbf{D}(\kk)$. When  $\mathscr{P}$ is the linearization of an $\infty$-operad enriched in spaces, this definition recovers the definitions of coalgebras over an $\infty$-operad considered in \cite[Chapter 2]{HigherAlgebra}, given by the opposite category of $\mathscr{P}$-algebras in $\mathbf{D}(\kk)^{\mathbf{op}}$. See \cite[Theorem B]{rectification} for more details.

\medskip

Moreover, let us recall the main theorem of \cite{rectification}, which compares the $\infty$-category of point-set coalgebras over a cofibrant dg operad localized at quasi-isomorphisms with coalgebras (in the sense above) over the underlying enriched $\infty$-operad of the dg operad. 

\begin{theorem}[{\cite[Main theorem]{rectification}}]\label{thm: rectification of P-coalgebras}
Let $\mathcal{P}$ be a cofibrant dg operad. There is an equivalence of $\infty$-categories
\[
\mathcal{P}\mbox{-}\mathsf{coalg}~[\mathsf{Q.iso}^{-1}] \simeq \mathbf{Coalg}_{\mathcal{P}}(\mathbf{D}(\kk))
\]

between dg $\mathcal{P}$-coalgebras up to quasi-isomorphism and coalgebras in $\mathbf{D}(\kk)$ over the induced enriched $\infty$-operad.     
\end{theorem}

In particular, this means that the point-set description of the cofree $\mathcal{P}$-coalgebra functor given in \cite{anelcofree2014} also describes the $\infty$-categorical cofree construction. For more details about this definition and its properties, we refer to \cite{rectification}.


\section{The point-set inclusion-restriction square}\label{Section: point-set inclusion-restriction square}


In this section, we construct a commutative square of Quillen adjunctions which we call the inclusion-restriction square. It intertwines two natural restriction and inclusion functors with the bar-cobar and the complete bar-cobar adjunctions. We furthermore show that this particular 1-categorical square of Quillen adjunctions induces an adjunction at the $\infty$-categorical level. The left and the right $\infty$-adjoints are both obtained by composing left and right Quillen functors, which is a somewhat surprising phenomenon. 

\begin{notation}
For the rest of this paper, $\mathcal{P}$ will always denote a dg operad and $\C$ will always denote a conilpotent dg cooperad. 
\end{notation}

Operads and cooperads throughout this paper will be subject to the following set of assumptions.

\begin{assumption}\label{assumption 1}\leavevmode
\begin{enumerate}
\item Their underlying dg symmetric sequences are \textit{reduced}, meaning they are trivial in arity $0$ and $\kk$ in arity $1$. 
\item Either the ground field $\kk$ is of characteristic zero, or the dg operads and conilpotent dg cooperads are non-symmetric. 
\item For all $n \geq 0$, the dg $\mathbb{S}_n$-module of arity $n$ operations is a degree-wise finite dimensional and bounded complex. Hence it is of finite total dimension over $\kk$.
\end{enumerate}
\end{assumption}

Let us make some remarks on these assumptions.
\begin{itemize}
    \item It follows from (1) that any dg operad is augmented, and any dg cooperad is both conilpotent and coaugmented. 
    \item It follows from (3) that the linear dual of a dg operad is a dg cooperad, and vice versa.
    \item These assumptions are stable under the operadic bar-cobar adjunction.
    \item Under assumption (2), most examples of interest satisfy (3). If (2) fails, then (3) is unnatural: for example, any $\mathbb E_\infty$ operad in positive characteristic is infinite-dimensional in each arity $>1$. In that setting (cf.~\cite[Section 2]{premierpapier}), one would need to relax assumption (3) so that the dg $\mathbb{S}_n$-module of arity $n$ operations is degree-wise finite dimensional and only bounded-below, in order to include examples like the quasi-planar dg cooperads of the form $\mathrm{B}(\mathcal{P} \otimes \mathcal{E})$, where $\mathcal{E}$ is the Barratt--Eccles operad \cite{BergerFresse}.
\end{itemize}


\begin{remark}[Starting with an operad or a cooperad.] For the rest of this section, we are going to start with dg operad $\mathcal{P}$, and we will use many times the conilpotent dg cooperad $\mathcal{P}^*$ given by its linear dual. This is a matter of convention, \textbf{the same constructions can be done} starting with a a conilpotent dg cooperad $\C$ and its linear dual dg cooperad $\mathcal{C}^*$. \end{remark}

\subsection{The restriction functor and its left adjoint}
The goal of this subsection is to recall some of the constructions of \cite[Section 3]{absolutealgebras}. In particular, we explain the adjunction between dg $\mathcal{P}^*$-algebras and dg $\mathcal{P}$-algebras given by restricting the absolute structure to finite sums, as well as some of its main properties. 

\begin{proposition}\label{Prop: absolution restriction adjunction}
There is an adjunction 
\[
\begin{tikzcd}[column sep=7pc,row sep=3pc]
            \mathcal{P}^*\mbox{-}\mathsf{alg}^{\mathsf{comp}}\arrow[r,"\mathrm{Res}"{name=F}, shift left=1.1ex] 
           &\mathcal{P}\mbox{-}\mathsf{alg}~, \arrow[l, shift left=.75ex, "c\mathrm{Abs}"{name=U}]
            \arrow[phantom, from=F, to=U, , "\dashv" rotate=90]
\end{tikzcd}
\]
between dg $\mathcal{P}^*$-algebras and dg $\mathcal{P}$-algebras, where the right adjoint $\mathrm{Res}$ is given by restricting the structure to finite sums of operations.
\end{proposition}

\begin{proof}
There is a natural inclusion of monads 
\[
\iota_{\mathcal{P}}: \bigoplus_{n \geq 0} \mathcal{P}(n) \otimes_{\mathbb{S}_n} (-)^{\otimes n} \hookrightarrow \prod_{n \geq 0} \mathrm{Hom}_{\mathbb{S}_n}(\mathcal{P}^*(n),(-)^{\otimes n})~,
\]
given by the composition of the canonical inclusion $V^* \otimes W \hookrightarrow \mathrm{Hom}(V,W)$, the norm map from coinvariants to invariants and the inclusion of the direct sum into the direct product. Pulling the structure map of a dg $\mathcal{P}^*$-algebra along $\iota$ induces the restriction functor $\mathrm{Res}$ from dg $\mathcal{P}^*$-algebras to dg $\mathcal{P}$-algebras. It is an accessible functor which preserves all limits between two presentable categories, hence it admits a left adjoint $c\mathrm{Abs}$. By post-composing this adjoint with the completion functor from dg $\mathcal{P}^*$-algebras to complete dg $\mathcal{P}^*$-algebras, we obtain the functor $c\mathrm{Abs}$ above and the desired adjunction. 
\end{proof}

\begin{theorem}[{\cite[Theorem 3.16]{absolutealgebras}}]\label{thm: res fully faithful}
Let $\mathcal{P}$ be a dg operad satisfying Assumption \ref{assumption 1}. The restriction functor 
\[
\mathrm{Res}: \mathcal{P}^*\mbox{-}\mathsf{alg}^{\mathsf{comp}} \longrightarrow \mathcal{P}\mbox{-}\mathsf{alg}
\]
is fully faithful. Therefore dg $\C$-algebras are a reflective subcategory of dg $\mathcal{P}$-algebras.
\end{theorem}

\begin{remark}\label{Rmk: I-adic completion not an equivalence}
Let us consider the case of commutative algebras, where $\mathcal{P} = \mathcal{C}om$. It is not true that complete absolute dg commutative algebras are equivalent to $I$-adically complete dg commutative algebras. This is because the $I$-adic completion is not in general idempotent, unless one restricts to the noetherian case. By contrast, complete absolute dg commutative algebras can always be described as algebras over an idempotent monad on the category of dg commutative algebras. See \cite[Proposition 3.24]{absolutealgebras} and the subsequent discussion for more details.
\end{remark}

\subsection{The inclusion functor and its right adjoint}
We now compare dg $\mathrm{B}\mathcal{P}$-coalgebras with dg $\Omega\mathcal{P}^*$-coalgebras, which encode, respectively, the conilpotent and non-necessarily conilpotent versions of the same coalgebra structures. 

\begin{proposition}
Let $\mathcal{P}$ be a dg operad satisfying Assumption \ref{assumption 1}. There is an adjunction 
\[
\begin{tikzcd}[column sep=7pc,row sep=3pc]
            \Omega\mathcal{P}^*\mbox{-}\mathsf{coalg}\arrow[r,"\mathrm{Sub}"{name=F}, shift left=1.1ex] 
           &\mathrm{B}\mathcal{P}\mbox{-}\mathsf{coalg}~, \arrow[l, shift left=.75ex, "\mathrm{Inc}"{name=U}]
            \arrow[phantom, from=F, to=U, , "\dashv" rotate=90]
\end{tikzcd}
\]
between dg $\mathrm{B}\mathcal{P}$-coalgebras and dg $\Omega\mathcal{P}^*$-coalgebras, where the left adjoint $\mathrm{Inc}$ is given by viewing a dg $\mathrm{B}\mathcal{P}$-coalgebra as a dg $\Omega\mathcal{P}^*$-coalgebra.
\end{proposition}

\begin{proof}
As in the proof of Proposition \ref{Prop: absolution restriction adjunction}, there is a natural inclusion
\[
\iota_{\mathrm{B}\mathcal{P}}: \bigoplus_{n \geq 0} \mathrm{B}\mathcal{P}(n) \otimes_{\mathbb{S}_n} (-)^{\otimes n} \hookrightarrow \prod_{n \geq 0} \mathrm{Hom}_{\mathbb{S}_n}(\Omega\mathcal{P}^*(n),(-)^{\otimes n})~,
\]
since $(\Omega \mathcal{P}^*)^*$ is isomorphic to $\mathrm{B}\mathcal{P}$. The latter fact follows from the dualizability hypothesis in Assumption \ref{assumption 1}. This natural transformation of functors induces an inclusion of comonads 
\[
\iota_{\mathrm{B}\mathcal{P}}': \bigoplus_{n \geq 0} \mathrm{B}\mathcal{P}(n) \otimes_{\mathbb{S}_n} (-)^{\otimes n} \hookrightarrow \mathrm{L}(\Omega\mathcal{P}^*)(-)~, 
\]
since it can be checked to be compatible with the pullback square that defines the cofree functor $\mathrm{L}(\Omega\mathcal{P}^*)(-)$ in Theorem \ref{Thm: existence of the cofree coalgebra}. Pushing forward the structural morphism of a dg $\mathrm{B}\mathcal{P}$-coalgebra along the inclusion $\iota_{\mathrm{B}\mathcal{P}}'$ induces the inclusion functor $\mathrm{Inc}$. It clearly preserves all colimits as these are computed in the underlying category of chain complexes, and both categories are presentable, thus it admits a right adjoint $\mathrm{Sub}$. 
\end{proof}

\begin{remark}
The functor $\mathrm{Sub}$ is given by taking the maximal dg $\mathrm{B}\mathcal{P}$-subcoalgebra of a given dg $\Omega\mathcal{P}^*$-coalgebra. It generalizes taking the maximal conilpotent subcoalgebra of a non-necessarily conilpotent coalgebra. 
\end{remark}

\begin{proposition}
Let $\mathcal{P}$ be a dg operad satisfying Assumption \ref{assumption 1}. The inclusion functor 
\[
\mathrm{Inc}: \mathrm{B}\mathcal{P}\mbox{-}\mathsf{coalg} \longrightarrow \Omega\mathcal{P}^*\mbox{-}\mathsf{coalg}
\]
is fully faithful. Therefore dg $\mathrm{B}\mathcal{P}$-coalgebras are a coreflective subcategory of dg $\Omega\mathcal{P}^*$-coalgebras.
\end{proposition}

\begin{proof}
It is obviously faithful, as it does not change the underlying chain complex. Since the inclusion $\iota_{\mathrm{B}\mathcal{P}}'$ is a monomorphism, it is also full, as any morphism of dg $\Omega\mathcal{P}^*$-coalgebras between two dg $\mathrm{B}\mathcal{P}$-coalgebras is also a morphism of dg $\mathrm{B}\mathcal{P}$-coalgebras. 
\end{proof}

\subsection{The algebraic inclusion-restriction square}
When we intertwine the restriction and the inclusion adjunctions with the two versions of Koszul duality induced by the classical bar-cobar and the complete bar-cobar adjunctions, we obtain a commuting square of adjoint functors. 

\begin{proposition}\label{Prop: inclusion-restriction square}
Let $\mathcal{P}$ be a dg operad satisfying Assumption \ref{assumption 1}. The square of adjunctions 
\[
\begin{tikzcd}[column sep=5pc,row sep=5pc]
\mathcal{P}\mbox{-}\mathsf{alg}\arrow[r,"\mathrm{B}_\pi"{name=B},shift left=1.1ex] \arrow[d,"c\mathrm{Abs}"{name=SD},shift left=1.1ex ]
&\mathrm{B}\mathcal{P}\mbox{-}\mathsf{coalg}, \arrow[d,"\mathrm{Inc}"{name=LDC},shift left=1.1ex ] \arrow[l,"\Omega_\pi"{name=C},,shift left=1.1ex]  \\
\mathcal{P}^*\mbox{-}\mathsf{alg}^{\mathsf{comp}} \arrow[r,"\widehat{\mathrm{B}}_\iota"{name=CC},shift left=1.1ex]  \arrow[u,"\mathrm{Res}"{name=LD},shift left=1.1ex ]
&\Omega\mathcal{P}^*\mbox{-}\mathsf{coalg} ~, \arrow[l,"\widehat{\Omega}_\iota"{name=CB},shift left=1.1ex] \arrow[u,"\mathrm{Sub}"{name=TD},shift left=1.1ex] \arrow[phantom, from=SD, to=LD, , "\dashv" rotate=180] \arrow[phantom, from=C, to=B, , "\dashv" rotate=90]\arrow[phantom, from=TD, to=LDC, , "\dashv" rotate=180] \arrow[phantom, from=CC, to=CB, , "\dashv" rotate=90]
\end{tikzcd}
\] 
commutes in the following sense: the compositions of the left adjoints from the top right to the bottom left of the square are naturally isomorphic.
\end{proposition}

\begin{proof}
Let us check that the diagram commutes. First, there is a natural isomorphism 
\[
\mathrm{S}(\mathcal{P})(V) \cong \widehat{\mathrm{S}}^c(\mathcal{P}^*)(V),
\]
for any dg module $V$, as the restriction functor commutes with the forgetful functors of dg $\mathcal{P}$-algebras and dg $\mathcal{P}^*$-algebras to dg modules. This gives a natural isomorphism of graded $\mathcal{P}^*$-algebras
\[
c\mathrm{Abs} \circ \Omega_\pi(D) \cong \widehat{\Omega}_\iota \circ \mathrm{Inc}(D)~, 
\]
for any dg $\mathrm{B}\mathcal{P}$-coalgebra $D$. It can be checked that this isomorphism commutes with the differentials, as the restriction of both differentials to the generators $D$ is given by 
\[
\small{
\begin{tikzcd}[column sep= 2.5 pc,row sep=2pc]
\displaystyle D \arrow[r,"\Delta_D"] 
&\displaystyle \bigoplus_{n \geq 0} \mathrm{B}\mathcal{P}(n) \otimes_{\mathbb{S}_n} D^{\otimes n} \arrow[r, "\pi~ \circ~ \mathrm{id}"]
&\displaystyle \bigoplus_{n \geq 0} \mathcal{P}(n) \otimes_{\mathbb{S}_n} D^{\otimes n} \arrow[r,"\iota_{\mathcal{P}}(D)"]
&\displaystyle \prod_{n \geq 0} \mathrm{Hom}_{\mathbb{S}_n}(\mathcal{P}^*(n),D^{\otimes n})~, 
\end{tikzcd}}
\]
where $\Delta_D$ is the structural morphism of $D$ and $\pi: \mathrm{B}\mathcal{P} \longrightarrow \mathcal{P}$ the canonical twisting morphism. 
\end{proof}

\subsection{Homotopical inclusion-restriction squares.} The inclusion-restriction square is compatible with the model structures on these categories in the following sense. 

\begin{lemma}\label{lemma: restriction and inclusion are Quillen functors}Let $\mathcal{P}$ be a dg operad satisfying Assumption \ref{assumption 1}.
\begin{enumerate}
\item The inclusion adjunction 
\[
\begin{tikzcd}[column sep=5pc,row sep=3pc]
            \mathrm{B}\mathcal{P}\mbox{-}\mathsf{coalg} \arrow[r,"\mathrm{Inc}"{name=F},shift left=1.1ex ] &   \Omega\mathcal{P}^*\mbox{-}\mathsf{coalg} \arrow[l,"\mathrm{Sub}"{name=U},shift left=1.1ex ]
            \arrow[phantom, from=F, to=U, , "\dashv" rotate=-90]
\end{tikzcd}
\]
is a Quillen adjunction, where the model structure on dg $\mathrm{B}\mathcal{P}$-coalgebras is that of Proposition \ref{Prop: left Bousfield localization of conilpotent coalgebras}, where weak equivalences are given by quasi-isomorphisms.

\item The restriction adjunction 
\[
\begin{tikzcd}[column sep=5pc,row sep=3pc]
            \mathcal{P}\mbox{-}\mathsf{alg} \arrow[r,"c\mathrm{Abs}"{name=F},shift left=1.1ex ] &\mathcal{P}^*\mbox{-}\mathsf{alg}^{\mathsf{comp}} \arrow[l,"\mathrm{Res}"{name=U},shift left=1.1ex ]
            \arrow[phantom, from=F, to=U, , "\dashv" rotate=-90]
\end{tikzcd}
\]
is a Quillen adjunction, where the model structure on dg $\mathcal{P}^*$-algebras is that of Proposition \ref{Prop: right Bousfield localization of absolute algebras}, where weak equivalences are given by quasi-isomorphisms.
\end{enumerate}

\end{lemma}

\begin{proof}
It is clear that $\mathrm{Inc}$ preserves cofibrations and weak equivalences since they are given by monomorphisms and quasi-isomorphisms on both sides. Similarly, $\mathrm{Res}$ preserves fibrations and weak equivalences since they are given by epimorphisms and quasi-isomorphisms on both sides.
\end{proof}

\begin{theorem}[The inclusion-restriction rectangle]\label{thm: full rectangle infinity categorical}
Let $\mathcal{P}$ be a dg operad satisfying Assumption \ref{assumption 1}. There is a rectangle of adjunctions of $\infty$-categories
\[
\begin{tikzcd}[column sep=3pc,row sep=5pc]
\mathcal{P}\mbox{-}\mathsf{alg}~[\mathsf{Q.iso}^{-1}]\arrow[r,"\mathrm{B}_\pi"{name=B},shift left=1.1ex] \arrow[d,"\mathbb{L}c\mathrm{Abs}"{name=SD},shift left=1.1ex ]
&\mathrm{B}\mathcal{P}\mbox{-}\mathsf{coalg}~[\mathsf{W.eq}^{-1}] \arrow[l,"\Omega_\pi"{name=C},shift left=1.1ex] \arrow[r,"\mathrm{Id}"{name=DDD},shift left=1.1ex] 
&\mathrm{B}\mathcal{P}\mbox{-}\mathsf{coalg}~[\mathsf{Q.iso}^{-1}], \arrow[l,"\mathbb{R}\mathrm{Id}"{name=FFF},shift left=1.1ex] \arrow[d,"\mathrm{Inc}"{name=LDC},shift left=1.1ex ] \\
\mathcal{P}^*\mbox{-}\mathsf{alg}^{\mathsf{comp}}~[\mathsf{Q.iso}^{-1}] \arrow[r,"\mathbb{L}\mathrm{Id}"{name=GGG},shift left=1.1ex]  \arrow[u,"\mathrm{Res}"{name=LD},shift left=1.1ex ]
&\mathcal{P}^*\mbox{-}\mathsf{alg}^{\mathsf{comp}}~[\mathsf{W.eq}^{-1}] \arrow[r,"\widehat{\mathrm{B}}_\iota"{name=CC},shift left=1.1ex]  \arrow[l,"\mathrm{Id}"{name=HHH},shift left=1.1ex ]
&\Omega\mathcal{P}^*\mbox{-}\mathsf{coalg}~[\mathsf{Q.iso}^{-1}] ~,\arrow[l,"\widehat{\Omega}_\iota"{name=CB},shift left=1.1ex] \arrow[u,"\mathbb{R}\mathsf{Sub}"{name=TD},shift left=1.1ex] \arrow[phantom, from=SD, to=LD, , "\dashv" rotate=180] \arrow[phantom, from=C, to=B, , "\simeq" rotate=0]\arrow[phantom, from=TD, to=LDC, , "\dashv" rotate=180] \arrow[phantom, from=CC, to=CB, , "\simeq" rotate=0]\arrow[phantom, from=DDD, to=FFF, , "\dashv" rotate=-90] \arrow[phantom, from=GGG, to=HHH, , "\dashv" rotate=-90]
\end{tikzcd}
\] 
where the bar-cobar adjunction and the complete bar-cobar adjunction are adjoint equivalences of $\infty$-categories. This rectangle is commutative, meaning that there is a natural weak equivalence between the left adjoint functors from the top left to the bottom right. 
\end{theorem}

\begin{proof}
The existence of this rectangle and the fact that it descends to the $\infty$-categorical setting follows directly from Proposition \ref{Prop: inclusion-restriction square} and Lemma \ref{lemma: restriction and inclusion are Quillen functors}. Let us check that it is indeed commutative and let us construct a natural weak equivalence between the following two composites
\[
\widehat{\mathrm{B}}_\iota \circ \mathrm{Id} \circ \mathbb{L}c\mathrm{Abs} \simeq \mathrm{Inc} \circ \mathrm{Id} \circ \mathrm{B}_\pi~. 
\]
We will omit non-derived identity functors in order to simplify the equations. The idea is to explicitly compute the left hand side. Let $A$ be a dg $\mathcal{P}$-algebra. First, we can take $\Omega_\pi \mathrm{B}_\pi A$ as a functorial cofibrant resolution of $A$, thus the derived functor $\mathbb{L}c\mathrm{Abs}(A)$ can be computed as
\[
\mathbb{L}c\mathrm{Abs}(A) \simeq \widehat{\Omega}_\iota \circ \mathrm{Inc} \circ \mathrm{B}_\pi(A)~. 
\]
Therefore the left hand side is given by 
\[
\widehat{\mathrm{B}}_\iota \circ \mathbb{L}c\mathrm{Abs}(A) \simeq \widehat{\mathrm{B}}_\iota \circ \widehat{\Omega}_\iota \circ \mathrm{Inc} \circ  \mathrm{B}_\pi(A)~, 
\]
and the counit 
\[
\epsilon_{\mathrm{Inc}~ \circ ~ \mathrm{B}_\pi(A)}: \widehat{\mathrm{B}}_\iota \circ \widehat{\Omega}_\iota \circ \mathrm{Inc} \circ  \mathrm{B}_\pi(A) \qi \mathrm{Inc} \circ \mathrm{B}_\pi(A)
\]
provides us with the natural weak equivalence we were looking for, as it is a natural quasi-isomorphism. 
\end{proof}

\begin{notation}
In order to relieve notational burden, we will from now on omit the non-derived identity functors that appear in Theorem \ref{thm: full rectangle infinity categorical} and its consequences, if there are no ambiguities.
\end{notation}

\subsection{The \texorpdfstring{$\infty$}{infinity}-categorical bar-cobar adjunction} Using the commutative rectangle of Theorem \ref{thm: full rectangle infinity categorical}, we obtain an adjunction at the $\infty$-categorical level between dg $\mathcal{P}$-algebras and their Koszul dual dg $\Omega \mathcal{P}^*$-coalgebras. As we will explain in Section \ref{Section: Lurie's bar-cobar}, this adjunction is a generalization of Lurie's adjunction between augmented monoids and coaugmented comonoids constructed in \cite[Theorem 5.2.2.17.]{HigherAlgebra}. 

\begin{definition}[The $\infty$-categorical bar construction relative to $\mathcal{P}$]
The $\infty$-categorical \textit{bar construction} $\mathbf{bar}_{\mathcal{P}}$ is the functor 
\[
\mathbf{bar}_{\mathcal{P}} \coloneqq \mathrm{Inc} \circ \mathrm{B}_\pi: \mathcal{P} \mbox{-} \mathsf{alg}~[\mathsf{Q.iso}^{-1}] \longrightarrow \Omega\mathcal{P}^*\mbox{-}\mathsf{coalg}~[\mathsf{Q.iso}^{-1}]
\]
from the $\infty$-category of dg $\mathcal{P}$-algebras localized at quasi-isomorphisms to the $\infty$-category of dg $\Omega\mathcal{P}^*$-coalgebras localized at quasi-isomorphisms. 
\end{definition}

\begin{definition}[The $\infty$-categorical cobar construction relative to $\mathcal{P}$]
The $\infty$-categorical \textit{cobar construction} $\mathbf{cobar}_{\mathcal{P}}$ is the functor 
\[
\mathbf{cobar}_{\mathcal{P}} \coloneqq \mathrm{Res} \circ \widehat{\Omega}_\iota: \Omega\mathcal{P}^*\mbox{-}\mathsf{coalg}~[\mathsf{Q.iso}^{-1}] \longrightarrow \mathcal{P}\mbox{-}\mathsf{alg}~[\mathsf{Q.iso}^{-1}]
\]
from the $\infty$-category of dg $\Omega\mathcal{P}^*$-coalgebras localized at quasi-isomorphisms to the $\infty$-category of dg $\mathcal{P}$-algebras localized at quasi-isomorphisms. 
\end{definition}

\begin{remark}
Strictly speaking, the functor $\mathbf{cobar}_{\mathcal{P}}$ is given as $\mathrm{Res} \circ  \mathbb{L}\mathrm{Id} \circ  \widehat{\Omega}_\iota$. However, the image of the complete cobar functor is always cofibrant in the category of complete dg $\mathcal{P}^*$-algebras with the transferred model structure; thus the derived identity functor $\mathbb{L}\mathrm{Id}$ is the identity functor and therefore we omit it for simplicity.
\end{remark}

\begin{theorem}\label{thm: infty categorical operadic bar cobar}
Let $\mathcal{P}$ be a dg operad satisfying Assumption \ref{assumption 1}. There is an adjunction of $\infty$-categories 
\[
\begin{tikzcd}[column sep=5pc,row sep=3pc]
            \mathcal{P}\mbox{-}\mathsf{alg}~[\mathsf{Q.iso}^{-1}] \arrow[r,"\mathbf{bar}_{\mathcal{P}}"{name=F},shift left=1.1ex ] & \Omega\mathcal{P}^*\mbox{-}\mathsf{coalg}~[\mathsf{Q.iso}^{-1}] \arrow[l,"\mathbf{cobar}_{\mathcal{P}}"{name=U},shift left=1.1ex ]
            \arrow[phantom, from=F, to=U, , "\dashv" rotate=-90]
\end{tikzcd}
\]
between the $\infty$-category of dg $\mathcal{P}$-algebras localized at quasi-isomorphisms and the $\infty$-category of dg $\Omega\mathcal{P}^*$-coalgebras localized at quasi-isomorphisms. 
\end{theorem}

\begin{proof}
Follows immediately from Theorem \ref{thm: full rectangle infinity categorical}. 
\end{proof}

\begin{remark}
This adjunction also agrees with the adjunction constructed in \cite{heuts2024} over a field of characteristic zero. In general, there are two variants of this adjunction over a positive characteristic field and \textit{op.cit.} considers one of them. See Subsection \ref{Subsection: positive char extension}. 
\end{remark}


\section{Point-set and \texorpdfstring{$\infty$}{infinity}-categorical bar-cobar adjunctions}\label{Section: Francis-Gaitsfory enhanced bar-cobar}


The goal of this section is to compare the point-set operadic bar-cobar adjunction of paragraph \ref{subsubsection: classical bar-cobar adjunction} with the $\infty$-categorical enhanced bar-cobar adjunction constructed in \cite{FrancisGaitsgory}. This latter bar-cobar adjunction goes from algebras over an (enriched) $\infty$-operad and (conilpotent, divided powers) coalgebras over its operadic bar construction. Contrary to what one would expect, the former adjunction does not present the latter. This is due to the failure of rectification statements for conilpotent coalgebras over a cooperad, as pointed out in \cite[Section 5.3]{chen25}, which follows from \cite[Chapter 6, Remark 2.6.7]{GaitsgoryRozenblyumVolII}. Nevertheless, we construct a way to compare these two adjunctions and explain what the underlying chain complexes of each of these functors are.

\subsection{The enhanced bar-cobar adjunction of Francis--Gaitsgory}\label{Subsection: enhanced bar-cobar adjunction} In this subsection, we recall some of the results in \cite[Section 3]{FrancisGaitsgory}. Let $\mathscr{P}$ be an augmented (enriched) $\infty$-operad in $\mathbf{D}(\kk)$. The augmentation map $\epsilon: \mathscr{P} \longrightarrow \mathbb{1}$ induces an adjunction 
\[
\begin{tikzcd}[column sep=5pc,row sep=3pc]
         \mathbf{Alg}_{\mathscr{P}}(\mathbf{D}(\kk)) \arrow[r, shift left=1.1ex, "\mathbf{indec}"{name=F}] &\mathbf{D}(\kk)~, \arrow[l, shift left=.75ex, "\mathbf{triv}"{name=U}]
            \arrow[phantom, from=F, to=U, , "\dashv" rotate=-90]
\end{tikzcd}
\]

between the $\infty$-category of $\mathscr{P}$-algebras and the underlying $\infty$-category of chain complexes over $\kk$. The    functor $\mathbf{triv}$ endows any chain complex with a trivial $\mathscr{P}$-algebra structure and its left adjoint $\mathbf{indec}$ is the functor of derived indecomposables of a $\mathscr{P}$-algebra. The \textit{homology} of $\mathbf{indec}$ is called the (topological) André--Quillen homology, see \cite{Basterra99, Mandell03}. This adjunction has a universal property: it exhibits $\mathbf{D}(\kk)$ as the stabilization of the $\infty$-category of $\mathscr{P}$-algebras \cite{BasterraMandell05}. When $\mathscr{P}$ is the $\mathbb{E}_\infty$-operad in $\mathbf{D}(\kk)$, this adjunction coincides with a particular case of the cotangent complex (see \cite[Section 7.3]{HigherAlgebra} and in particular \cite[Remark 7.3.0.1]{HigherAlgebra}). 

\medskip

Let us explain how the enhanced bar-cobar adjunction appears from the above data. The adjunction $\mathbf{indec} \dashv \mathbf{triv}$ induces a comonad $\mathbf{T} := \mathbf{indec} \circ \mathbf{triv}$ on $\mathbf D(\kk)$. By its universal property, the adjunction $\mathbf{indec} \dashv \mathbf{triv}$ factors through the $\infty$-category of coalgebras over the comonad $\mathbf{T}$, and we get 

\[
\begin{tikzcd}[column sep=5pc,row sep=2.5pc]
&~~~~~~~~~~ \mathbf{Alg}_{\mathscr{P}}(\mathbf{D}(\kk)) \arrow[dd, shift left=1.1ex, "\mathbf{B}^{\mathbf{enh}}"{name=F}] \arrow[ld, shift left=.75ex, "\mathbf{indec}"{name=C}]\\
\mathbf{D}(\kk)  \arrow[ru, shift left=1.5ex, "\mathbf{triv}"{name=A}]  \arrow[rd, shift left=1ex, "\mathbf{cofree}"{name=B}] 
& \\
&\hspace{2.5pc} \mathbf{Coalg}_{\mathbf{T}}(\mathbf{D}(\kk)) ~, \arrow[uu, shift left=.75ex, "\bm{\Omega}^{\mathbf{enh}}"{name=U}] \arrow[lu, shift left=.75ex, "\mathbf{U}"{name=D}] 
\end{tikzcd}
\]
where right adjoints are drawn on top and on the right. Moreover, the comonad $\mathbf{T}$ can be explicitly identified. 

\begin{proposition}\label{prop: identification of the derived indecomposables comonad}
There is a natural weak equivalence of $\infty$-categorical comonads 
\[
\mathbf{T} \simeq \mathbf{S}(\mathbf{bar}(\mathscr{P}))(-) \simeq \bigoplus_{n \geq 0} \mathbf{bar}(\mathscr{P})(n) \otimes_{\mathbb{S}_n} (-)^{\otimes n}
\]
between the comonad $\mathbf{T}$ induced by the adjunction $\mathbf{indec} \dashv \mathbf{triv}$ and the comonad given by the Schur functor of the $\infty$-cooperad $\mathbf{bar}(\mathscr{P})$. 
\end{proposition}

\begin{proof}
Let us compute the composition $\mathbf{indec} \dashv \mathbf{triv}$. Let $V$ be a chain complex. Since the $\infty$-category of $\mathscr{P}$-algebras is monadic, any $\mathscr{P}$-algebra admits a resolution which is the geometric realization of a simplicial object made of free $\mathscr{P}$-algebras. In particular, so does the $\mathscr{P}$-algebra $\mathbf{triv}(V)$, which can be obtain as the geometric realization
\[
\mathbf{triv}(V) \simeq \left|~\mathbf{bar}(\mathbf{S}(\mathscr{P}),\mathbf{S}(\mathscr{P}),\mathbf{triv}(V))_\bullet ~ \right|~,
\]
where on the right hand side we consider the geometric realization of the simplicial object 
\[\hspace{-0.75cm}
\small{
\begin{tikzcd}
\mathbf{S}(\mathscr{P})(\mathbf{triv}(V)) \arrow[r]
&\mathbf{S}(\mathscr{P})\circ \mathbf{S}(\mathscr{P})(\mathbf{triv}(V)) \arrow[l,shift left=1.1ex] \arrow[l,shift right=1.1ex] \arrow[r,shift left=1ex] \arrow[r,shift right=1ex] 
&\mathbf{S}(\mathscr{P}) \circ \mathbf{S}(\mathscr{P}) \circ \mathbf{S}(\mathscr{P})(\mathbf{triv}(V)) \arrow[l,shift left=2ex] \arrow[l,shift right=2ex] \arrow[l] \arrow[r,shift left= 2ex] \arrow[r,shift right=2ex]  \arrow[r]
&\cdots \arrow[l,shift left=3ex] \arrow[l,shift right=3ex]  \arrow[l,shift left=1ex] \arrow[l,shift right=1ex]
\end{tikzcd}}~.\]
Since $\mathbf{indec}$ commutes with colimits, we have that 
\[
\begin{aligned}
\mathbf{indec} \circ \mathbf{triv}(V) &\simeq \left|~\mathbf{indec}\left(\mathbf{bar}(\mathbf{S}(\mathscr{P}),\mathbf{S}(\mathscr{P}),\mathbf{triv}(V))_\bullet\right) ~ \right| \\ &\simeq \left|~\left(\mathbf{bar}(\mathbf{id},\mathbf{S}(\mathscr{P}),\mathbf{triv}(V))_\bullet\right) ~ \right| \\ &\simeq \mathbf{S}(\mathbf{bar}(\mathscr{P}))(V)~, 
\end{aligned}
\]
where $\mathbf{id}$ is the identity endofunctor. The last equivalence follows from the fact that $\mathbf{S}(-)$ also commutes with colimits, and that the structure maps of $\mathbf{triv}(V)$ are all trivial. 
\end{proof}

Hence we get an adjunction 
\[
\begin{tikzcd}[column sep=5pc,row sep=3pc]
         \mathbf{Alg}_{\mathscr{P}}(\mathbf{D}(\kk)) \arrow[r, shift left=1.1ex, "\mathbf{B}^{\mathbf{enh}}"{name=F}] &\mathbf{Coalg}_{\mathbf{bar}(\mathscr{P})}^{\mathbf{conil}}(\mathbf{D}(\kk))~, \arrow[l, shift left=.75ex, "\bm{\Omega}^{\mathbf{enh}}"{name=U}] \arrow[phantom, from=F, to=U, , "\dashv" rotate=-90]
\end{tikzcd}
\]
called the \textit{enhanced bar-cobar adjunction} in \cite{FrancisGaitsgory}, between $\mathscr{P}$-algebras and conilpotent divided powers $\mathbf{bar}(\mathscr{P})$-coalgebras in the sense of Definition \ref{def: infinity cat C-coalgebras}. This adjunction is the universal factorization of the indecomposables-trivial adjunction over coalgebras over a comonad. 

\begin{remark}
The original Francis--Gaitsgory conjecture, disproved in \cite{heuts2024}, stated that this adjunction is an equivalence of $\infty$-categories when one restricts the enhanced bar construction to \textit{pro-nilpotent} $\mathscr{P}$-algebras. See \cite[Section 3.4]{FrancisGaitsgory} for the original statement. 
\end{remark}

\subsection{A comparison between strict and $\infty$-categorical conilpotent coalgebras}\label{subsection: comparison of conilpotent coalgebras} Let us start by explaining why the point-set operadic bar-cobar adjunction cannot present its $\infty$-categorical analogue given by the enhanced bar-cobar adjunction. Concretely, one would like to say that the following composite adjunction 

\begin{equation}\label{equation: point-set candidate for fg bar-cobar}
\begin{tikzcd}[column sep=5pc,row sep=3pc]
         \mathcal{P}\mbox{-}\mathsf{alg}~[\mathsf{Q.iso}^{-1}] \arrow[r, shift left=1.1ex, "\mathrm{B}_{\pi}"{name=F}] 
          &\mathrm{B}\mathcal{P}\mbox{-}\mathsf{cog}~[\mathsf{W}^{-1}] \arrow[l, shift left=.75ex, "\Omega_{\pi}"{name=U}] \arrow[r, shift left=1.1ex, "\mathrm{Id}"{name=A}]
          &\mathrm{B}\mathcal{P}\mbox{-}\mathsf{cog}~[\mathsf{Q.iso}^{-1}] \arrow[l, shift left=.75ex, "\mathbb{R}\mathrm{Id}"{name=B}]
            \arrow[phantom, from=F, to=U, , "\simeq" rotate=0] \arrow[phantom, from=A, to=B, , "\dashv" rotate=-90]
\end{tikzcd}
\end{equation}

coincides with the enhanced bar-cobar adjunction

\begin{equation}\label{equation: fg bar-cobar}
\begin{tikzcd}[column sep=5pc,row sep=3pc]
         \mathbf{Alg}_{\mathcal{P}}(\mathbf{D}(\kk)) \arrow[r, shift left=1.1ex, "\mathbf{B}^{\mathbf{enh}}"{name=F}] &\mathbf{Coalg}_{\mathbf{bar}(\mathcal{P})}^{\mathbf{conil}}(\mathbf{D}(\kk))~. \arrow[l, shift left=.75ex, "\bm{\Omega}^{\mathbf{enh}}"{name=U}] \arrow[phantom, from=F, to=U, , "\dashv" rotate=-90]
\end{tikzcd}
\end{equation}
However, if this was true, then because the adjunction (\ref{equation: point-set candidate for fg bar-cobar}) is a composite of an equivalence together with a left Bousfield localization, it would imply that $\mathbf{bar}(\mathcal{P})$-coalgebras are a reflective sub-$\infty$-category of $\mathcal{P}$-algebras. However, as pointed out in \cite[Theorem 5.3.1]{chen25}, this cannot be the case by the counterexample given in \cite[Chapter 6, Remark 2.6.7]{GaitsgoryRozenblyumVolII}, where one takes $\mathcal{P}$ to be the commutative operad $\mathcal{C}om$. 

\begin{proposition}\label{prop: adjunction between strict and loose conilpotent coalgebras}
Let $\mathcal{P}$ be a dg operad satisfying Assumptions \ref{assumption 1}. There is an adjunction of $\infty$-categories
\[
\begin{tikzcd}[column sep=5pc,row sep=3pc]
         \mathrm{B}\mathcal{P}\mbox{-}\mathsf{cog}~[\mathsf{Q.iso}^{-1}] \arrow[r, shift left=1.1ex, "\mathbf{loose}"{name=F}] &\mathbf{Coalg}_{\mathbf{bar}(\mathcal{P})}^{\mathbf{conil}}(\mathbf{D}(\kk))~, \arrow[l, shift left=.75ex, "\mathbf{strict}"{name=U}] \arrow[phantom, from=F, to=U, , "\dashv" rotate=-90]
\end{tikzcd}
\]
between dg $\mathrm{B}\mathcal{P}$-coalgebras up to quasi-isomorphism and $\mathbf{bar}(\mathcal{P})$-coalgebras, where $\mathbf{bar}(\mathcal{P})$ is the underlying enriched $\infty$-cooperad of $\mathrm{B}\mathcal{P}$. 
\end{proposition}

\begin{proof}
This adjunction follows from the fact that cofree-forgetful adjunction, at the point-set level, between dg $\mathrm{B}\mathcal{P}$-coalgebras and chain complexes, induces the comonad $\mathrm{S}(\mathrm{B}\mathcal{P})(-)$ on the underlying category of chain complexes. This $1$-comonad presents the $\infty$-categorical comonad $\mathbf{S}(\mathbf{bar}(\mathcal{P}))(-)$ on $\mathbf{D}(\kk)$, since by Remark \ref{rmk: point-set bar and infinity cat bar} the conilpotent dg cooperad $\mathrm{B}\mathcal{P}$ presents the enriched $\infty$-cooperad $\mathbf{bar}(\mathcal{P})$. Hence, by the (dual) universal property of the Eilenberg-Moore category, there exists such an adjunction. 
\end{proof}

The above adjunction is not, in general, an equivalence of $\infty$-categories, which prevents us from giving point-set models for the enhanced bar-cobar adjunction. The core issue is that the cofree-forgetful adjunction between dg $\mathrm{B}\mathcal{P}$-coalgebras and chain complexes is not comonadic in the $\infty$-categorical sense in general. A characterization of the $\infty$-category of dg $\mathrm{B}\mathcal{P}$-coalgebras up to quasi-isomorphisms is the following: it is the localization of the $\infty$-category of $\mathcal{P}$-algebras with respect to maps which are sent by $\mathrm{B}_\pi$ to a quasi-isomorphism. In other words, with respect to $\mathcal{P}$-algebra maps which induce an equivalence in (topological) André--Quillen homology. 

\subsection{The operadic bar-cobar adjunction and the enhanced bar-cobar adjunction}
Let $\mathcal{P}$ be a dg operad satisfying Assumptions \ref{assumption 1}. We show that the composition of the point-set adjunction (\ref{equation: point-set candidate for fg bar-cobar}) together with the adjunction of Proposition \ref{prop: adjunction between strict and loose conilpotent coalgebras} given indeed the enhanced bar-cobar adjunction (\ref{equation: fg bar-cobar}). See also \cite[Section 5.3]{chen25} for similar statements.

\begin{theorem}\label{thm: point-set models for enhanced bar-cobar}
Let $\mathcal{P}$ be a dg operad satisfying Assumptions \ref{assumption 1}. The $\infty$-categorical adjunction induced by the bar-cobar adjunction with respect to $\pi$ composed with the left Bousfield localization with respect to quasi-isomorphisms and with the comparison adjunction of Proposition \ref{prop: adjunction between strict and loose conilpotent coalgebras} 

\hspace{-2pc}
\begin{tikzcd}[column sep=3pc,row sep=3pc]
         \mathcal{P}\mbox{-}\mathsf{alg}~[\mathsf{Q.iso}^{-1}] \arrow[r, shift left=1.1ex, "\mathrm{B}_{\pi}"{name=F}] 
          &\mathrm{B}\mathcal{P}\mbox{-}\mathsf{cog}~[\mathsf{W}^{-1}] \arrow[l, shift left=.75ex, "\Omega_{\pi}"{name=U}] \arrow[r, shift left=1.1ex, "\mathrm{Id}"{name=A}]
          &\mathrm{B}\mathcal{P}\mbox{-}\mathsf{cog}~[\mathsf{Q.iso}^{-1}] \arrow[l, shift left=.75ex, "\mathbb{R}\mathrm{Id}"{name=B}]  \arrow[r, shift left=1.1ex, "\mathbf{loose}"{name=DDD}]
            \arrow[phantom, from=F, to=U, , "\simeq" rotate=0] \arrow[phantom, from=A, to=B, , "\dashv" rotate=-90]
    &\mathbf{Coalg}_{\mathbf{bar}(\mathcal{P})}^{\mathbf{conil}}(\mathbf{D}(\kk)) \arrow[l, shift left=.75ex, "\mathbf{strict}"{name=FFF}] \arrow[phantom, from=DDD, to=FFF, , "\dashv" rotate=-90]
\end{tikzcd}


is naturally equivalent to the $\infty$-categorical enhanced bar-cobar adjunction 
\[
\begin{tikzcd}[column sep=5pc,row sep=3pc]
         \mathbf{Alg}_{\mathcal{P}}(\mathbf{D}(\kk)) \arrow[r, shift left=1.1ex, "\mathbf{B}^{\mathbf{enh}}"{name=F}] &\mathbf{Coalg}_{\mathbf{bar}(\mathcal{P})}^{\mathbf{conil}}(\mathbf{D}(\kk))~. \arrow[l, shift left=.75ex, "\bm{\Omega}^{\mathbf{enh}}"{name=U}] \arrow[phantom, from=F, to=U, , "\dashv" rotate=-90]
\end{tikzcd}
\]

\end{theorem}

\begin{proof}
The dg operad $\mathcal{P}$ is in particular augmented, the augmentation induces a Quillen adjunction $\mathrm{indec} \dashv \mathrm{triv}$ at the point-set level. The derived adjunction 
\[
\begin{tikzcd}[column sep=5pc,row sep=3pc]
            \mathsf{Ch}(\kk)~[\mathsf{Q.iso}^{-1}] \arrow[r,"\mathrm{triv}"{name=F},shift left=1.1ex ] &   \mathcal{P}\mbox{-}\mathsf{alg}~[\mathsf{Q.iso}^{-1}] \arrow[l,"\mathbb{L}\mathrm{indec}"{name=U},shift left=1.1ex ]
            \arrow[phantom, from=F, to=U, , "\dashv" rotate=90]
\end{tikzcd}
\]
presents the adjunction $\mathbf{indec} \dashv \mathbf{triv}$ of the underlying $\infty$-operad of $\mathcal{P}$. Moreover, both adjunctions in the above theorem are universal factorizations of the adjunction $\mathbf{indec} \dashv \mathbf{triv}$ through the same $\infty$-category of coalgebras. Therefore, since they satisfy the same universal property, they must coincide. See e.g.~\cite{Heine17} for the $\infty$-categorical version of the universal property of the Eilenberg--Moore category of algebras over a monad. \end{proof}

\begin{remark}
Over a positive characteristic field, Theorem \ref{thm: point-set models for enhanced bar-cobar} also holds in the following, more general context: given a dg operad $\mathcal{P}$ which is $\mathbb{S}$-projective, the quasi-planar bar-cobar adjunction of \cite[Section 3.14]{premierpapier} composed with the Bousfield left localization of $\mathrm{B}(\mathcal{P}\otimes \mathcal{E})$ as in \cite[Proposition 48]{premierpapier} also fits in the same diagram, where at the $\infty$-categorical level we also consider the enhanced bar-cobar adjunction with respect of the underlying enriched $\infty$-operad of $\mathcal{P}$.
\end{remark}

\section{A dual version of the \texorpdfstring{$\infty$}{infinity}-categorical enhanced bar-cobar adjunction}


In this section, we introduce a dual version of the enhanced bar-cobar adjunction constructed in \cite{FrancisGaitsgory}, which is constructed by considering the universal factorization of the derived primitives functor of a given class of coalgebras (not necessarily conilpotent) through algebras over a monad. We prove that this results in an adjunction between coalgebras over an (enriched) $\infty$-operad and algebras over its dual $\infty$-cooperad. Moreover, and contrary to what happens with the enhanced bar-cobar adjunction, we give point-set models for this new construction. 

\subsection{A complete version of the enhanced bar-cobar adjunction}\label{Subsection: Complete enhanced bar-cobar adjunction} 
Let $\mathscr{P}$ be an augmented (enriched) $\infty$-operad in $\mathbf{D}(\kk)$. There is an associated $\infty$-category of $\mathscr{P}$-coalgebras, whose objecs are the coalgebras (non-conilpotent, without divided powers) defined by the operations in $\mathscr{P}$. For example, if $\mathscr{P}$ is the $\mathbb{E}_1$-operad or the $\mathbb{E}_\infty$-operad, then $\mathscr{P}$-coalgebras correspond to $\mathbb{E}_1$- or $\mathbb{E}_\infty$-coalgebras in the sense of \cite{HigherAlgebra}. We refer to Subsection \ref{subsubsection: infinity cat of coalgebras over an operad} or \cite[Section 3]{rectification} for more details about this definition. 

\medskip

The augmentation map $\epsilon: \mathscr{P} \longrightarrow \mathbb{1}$ induces another adjunction
\[
\begin{tikzcd}[column sep=5pc,row sep=3pc]
         \mathbf{Coalg}_{\mathscr{P}}(\mathbf{D}(\kk)) \arrow[r, shift left=1.1ex, "\mathbf{prim}"{name=F}] &\mathbf{D}(\kk)~, \arrow[l, shift left=.75ex, "\mathbf{triv}"{name=U}]
            \arrow[phantom, from=F, to=U, , "\dashv" rotate=90]
\end{tikzcd}
\]
this time between the $\infty$-category of $\mathscr{P}$-coalgebras and the derived $\infty$-category of chain complexes over $\kk$. The functor $\mathbf{triv}$ endows any chain complex with a trivial $\mathscr{P}$-coalgebra structure and its right adjoint $\mathbf{prim}$ is the functor of derived primitives of a $\mathscr{P}$-coalgebra. Its homology groups should be understood as (topological) coAndré-Quillen homology groups for types of coalgebras, of which (topological) coHochschild homology groups in the sense of \cite{HessShipley21} and \cite{BayindirPeroux23} are particular examples. See, in particular, \cite[Definition 2.10]{BayindirPeroux23}. 

\medskip

The adjunction $\mathbf{triv} \dashv \mathbf{prim}$ induces a monad on the derived $\infty$-category of chain complexes, which we denote by $\mathbf{Q}_{\mathscr{P}} = \mathbf{prim} \circ \mathbf{triv}$. Thus, the derived primitives adjunction factors universally through the $\infty$-category of $\mathbf{Q}_{\mathscr{P}}$-algebras, and we obtain the following triangle of adjunctions
\[
\begin{tikzcd}[column sep=5pc,row sep=2.5pc]
&\hspace{1pc}\mathbf{Alg}_{\mathbf{Q}_{\mathscr{P}}}(\mathbf{D}(\kk)) \arrow[dd, shift left=1.1ex, "\bm{\widehat{\Omega}}^{\mathbf{enh}}"{name=F}] \arrow[ld, shift left=.75ex, "\mathbf{U}"{name=C}]\\
\mathbf{D}(\kk)  \arrow[ru, shift left=1.5ex, "\mathbf{free}_{\mathbf{Q}_{\mathscr{P}}}"{name=A}]  \arrow[rd, shift left=1ex, "\mathbf{triv}"{name=B}] 
& \\
&\hspace{2.5pc}\mathbf{Coalg}_{\mathscr{P}}(\mathbf{D}(\kk)) ~, \arrow[uu, shift left=.75ex, "\bm{\widehat{\mathrm{B}}}^{\mathbf{enh}}"{name=U}] \arrow[lu, shift left=.75ex, "\mathbf{prim}"{name=D}] 
\end{tikzcd}
\]
where left adjoints are drawn on top and on the right. In particular, this gives a universal adjunction 
\[
\begin{tikzcd}[column sep=5pc,row sep=3pc]
         \mathbf{Coalg}_{\mathscr{P}}(\mathbf{D}(\kk)) \arrow[r, shift left=1.1ex, "\bm{\widehat{\mathrm{B}}}^{\mathbf{enh}}"{name=F}] &\mathbf{Alg}_{\mathbf{Q}_{\mathscr{P}}}(\mathbf{D}(\kk))~, \arrow[l, shift left=.9ex, "\bm{\widehat{\Omega}}^{\mathbf{enh}}"{name=U}] \arrow[phantom, from=F, to=U, , "\dashv" rotate=90]
\end{tikzcd}
\]

which factors the derived primitives-trivial structure adjunction, which we call the \textit{complete enhanced bar-cobar adjunction}. Although the triangle constructed above is dual to the one constructed in Subsection \ref{Subsection: enhanced bar-cobar adjunction}, Proposition \ref{prop: identification of the derived indecomposables comonad} does not a priori dualize and we can not identify the monad $\mathbf{Q}_{\mathscr{P}}$ right away. The reason is that the cosimplicial diagram that resolves a generic $\mathscr{P}$-coalgebra is made of cofree $\mathscr{P}$-coalgebras, which unlike algebras, admit no easy description. Nevertheless, using point-set models, we will be able to present the $\infty$-categorical $\mathbf{triv} \dashv \mathbf{prim}$ adjunction, its induced monad and thus the $\infty$-category of $\mathbf{Q}_{\mathscr{P}}$, which is going to be identified with the $\infty$-category of algebras over the $\infty$-cooperad $\mathbf{bar}(\mathscr{P})$ in the sense of Definition \ref{def: infinity cat C-algebra}.

\subsection{Absolute algebras and limit towers of $\infty$-categories}\label{subsection: absolute algebras inside the limit}
Let $\C$ be a conilpotent dg cooperad which satisfies Assumptions \ref{assumption 1}. It can be written as the colimit of the following tower of cooperads
\[
\tau^{\leq 1}\C \rightarrowtail \tau^{\leq 2}\C \rightarrowtail \cdots \rightarrowtail \tau^{\leq k}\C \rightarrowtail \cdots \rightarrowtail \colim_{k \in \mathbb{N}} \tau^{\leq k}\C \cong \C~,
\] 
where $\tau^{\leq k}\C$ is the truncation of $\C$ at arities strictly larger than $k$, for all $k \geq 1$. This tower induces an \textit{arity-filtration} on any dg $\C$-algebra $A$, whose successive quotients are given by 
\[
\begin{tikzcd}[column sep=3.5pc,row sep=3.5pc]
\widehat{\mathrm{S}}^c(\C)(A) \arrow[d,"\gamma_A",swap] \arrow[r,"\widehat{\mathrm{S}}^c(i_k)"]  \arrow[dr, phantom,"\ulcorner" rotate=-180, very near end]
&\widehat{\mathrm{S}}^c(\tau^{\leq k}\C)(A)\arrow[d]  \\
A \arrow[r]
&A/\mathrm{F}^k A~,
\end{tikzcd}
\]
where $i_k: \tau^{\leq k}\C \hookrightarrow \C$ is the canonical inclusion. Let us denote by 
\[
\upsilon_A: A \longrightarrow (\widehat{A})_{\mathrm{arity}} \coloneqq \lim_{k \in \mathbb{N}} A/\mathrm{F}^k A
\]
the canonical map from $A$ to its \textit{arity-completion}. 

\begin{lemma}\label{lemma: arity-filtration map to completion surjective}
Let $\C$ be a conilpotent dg cooperad which satisfies Assumptions \ref{assumption 1} and let $A$ be a dg $\C$-algebra. The map 
\[
\upsilon_A: A \twoheadrightarrow \lim_{k \in \mathbb{N}} A/\mathrm{F}^k A
\]
is a degree-wise epimorphism.
\end{lemma}

\begin{proof}
The idea of the proof is completely analogous to \cite[Proposition 4.24]{linearcoalgebras}. One must replace the weight filtration with the arity filtration in \cite[4.20, 4.21, 4.22 and 4.23]{linearcoalgebras}. 
See also \cite[Section 3.6]{premierpapier} for similar arguments. 
\end{proof}

\begin{proposition}
Let $\C$ be a conilpotent dg cooperad which satisfies Assumptions \ref{assumption 1} and let $A$ be a complete dg $\C$-algebra. The map 
\[
\upsilon_A: A \xrightarrow[]{\simeq} \lim_{k \in \mathbb{N}} A/\mathrm{F}^k A
\]
is an isomorphism, thus any complete dg $\C$-algebra is arity-complete.
\end{proposition}

\begin{proof}
By Lemma \ref{lemma: arity-filtration map to completion surjective}, it suffices to show that $\upsilon_A$ is a monomorphism. The key point is the following: since $\C$ satisfies Assumptions \ref{assumption 1}, then it is cogenerated in arities $\geq 2$. This implies that 
\[
\mathrm{F}^k A \subset \mathrm{W}^{k-1} A~, 
\]
since operations in arities $\leq k$ can be of weight at most $k-1$. This, in turn, implies that 
\[
\bigcap_{k \geq 1} \mathrm{F}^k A \subset \bigcap_{k \geq 1}  \mathrm{W}^{k-1} A = \{0\}~, 
\]
since we assumed $A$ to be complete. Hence the kernel of the map $\upsilon_A$ is zero and it is also a monomorphism. 
\end{proof}

Therefore, any complete dg $\C$-algebra $A$ can be written as a limit
\[
A \cong \lim_{k \in \mathbb{N}} A/\mathrm{F}^k A~.
\]
Notice that $A/\mathrm{F}^k A$ is the unit of the adjunction 
\[
\begin{tikzcd}[column sep=5pc,row sep=3pc]
            \mathcal{C}\mbox{-}\mathsf{alg}^{\mathsf{comp}} \arrow[r,"(i_k)_!"{name=F},shift left=1.1ex ] &   \tau^{\leq k}\mathcal{C}\mbox{-}\mathsf{alg}^{\mathsf{comp}} \arrow[l,"(i_k)^*"{name=U},shift left=1.1ex ]
            \arrow[phantom, from=F, to=U, , "\dashv" rotate=-90]
\end{tikzcd}
\]

induced by the inclusion $i_k: \tau^{\leq k} \C \longrightarrow \C$, that is, we have an isomorphism $(i_k)^*(i_k)_!(A) \cong A/\mathrm{F}^k A$ for any complete dg $\C$-algebra $A$. 

\begin{lemma}\label{lemma: identification of the truncated category}
Let $k \geq 1$. There is an equivalence of categories
\[
\tau^{\leq k}\mathcal{C}\mbox{-}\mathsf{alg}^{\mathsf{comp}} \cong \tau^{\leq k}\mathcal{C}^*\mbox{-}\mathsf{alg}~. 
\]
\end{lemma} 

\begin{proof}
The natural inclusion of monads 
\[
\iota_{\tau^{\leq k}\mathcal{C}}: \bigoplus_{n \geq 0}^k \tau^{\leq k}\mathcal{C}^*(n) \otimes_{\mathbb{S}_n} (-)^{\otimes n} \hookrightarrow \prod_{n \geq 0}^k \mathrm{Hom}_{\mathbb{S}_n}(\tau^{\leq k}\mathcal{C}(n),(-)^{\otimes n})~,
\]
is in fact an isomorphisms since finite sums and products agree. 
\end{proof}

\begin{theorem}\label{thm: arity decomposition of the infinity category of complete C algebras}
There is an adjunction of $\infty$-categories 
\[
\begin{tikzcd}[column sep=6pc,row sep=3pc]
            \mathcal{C}\mbox{-}\mathsf{alg}^{\mathsf{comp}}~[\mathsf{Q.iso}^{-1}] \arrow[r,"(\mathbb{L}(i_k)_!(-))_{k \geq 1}"{name=F},shift left=1.1ex ] &\lim_{k \geq 1}^h \tau^{\leq k}\mathcal{C}^*\mbox{-}\mathsf{alg}~[\mathsf{Q.iso}^{-1}]~, \arrow[l,"\lim_{k \geq 1}^h (i_k)^*(-)"{name=U},shift left=1.1ex ]
            \arrow[phantom, from=F, to=U, , "\dashv" rotate=-90]
\end{tikzcd}
\]
where the left adjoint $(\mathbb{L}(i_k)_!(-))_{k \geq 1}$ is fully faithful. 
\end{theorem}

\begin{proof}
Up to the equivalence in Lemma \ref{lemma: identification of the truncated category}, the adjunction 
\[
\begin{tikzcd}[column sep=5pc,row sep=3pc]
            \mathcal{C}\mbox{-}\mathsf{alg}^{\mathsf{comp}} \arrow[r,"(i_k)_!"{name=F},shift left=1.1ex ] &   \tau^{\leq k}\mathcal{C}\mbox{-}\mathsf{alg}^{\mathsf{comp}} \arrow[l,"(i_k)^*"{name=U},shift left=1.1ex ]
            \arrow[phantom, from=F, to=U, , "\dashv" rotate=-90]
\end{tikzcd}
\]
is a Quillen adjunction for all $n \geq 1$, since the functor $(i_n)^*$ clearly preserves fibrations and quasi-isomorphisms.  This induces a functor of $\infty$-categories
\[
(\mathbb{L}(i_k)_!(-))_{k \geq 1}: \mathcal{C}\mbox{-}\mathsf{alg}^{\mathsf{comp}}~[\mathsf{Q.iso}^{-1}] \longrightarrow \lim_{k \geq 1} \tau^{\leq k}\mathcal{C}^*\mbox{-}\mathsf{coalg}~[\mathsf{Q.iso}^{-1}]~. 
\]
The functor sends a complete dg $\C$-algebra $A$ to the collection $(\mathbb{L}(i_k)_!(A))_{k \geq 1}$, which lies in the limit in the right hand side. It can be checked that it has a a right adjoint, which sends a collection $(A_k')_{k \geq 1}$ of $\tau^{\leq k}\mathcal{C}^*$-algebras to their (homotopy) limit $\lim_{k \geq 1}^h((i_k)^*(A_k'))_{k \geq 1}$. Let $A$ be a cofibrant complete dg $\C$-algebra, it can be written as a (1-categorical) limit 
\[
A \cong \lim_{k \in \mathbb{N}} (i_k)^*(i_k)_!(A)~, 
\]
where all the maps are surjections and all the algebras are fibrant. Hence it is also a homotopy limit. Furthermore, since $A$ is cofibrant, the object $(i_k)_!(A)$ computes $\mathbb{L}(i_k)_!(A)$. Therefore the unit of the $\infty$-categorical adjunction is a weak equivalence and $\mathbb{L}(i_k)_!$ is fully faithful. 
\end{proof}

\begin{remark}[Relationship with the Goodwillie tower] 
The limit 
\[
\lim_{n \geq 1} \tau^{\leq n}\mathcal{C}^*\mbox{-}\mathsf{alg}~[\mathsf{Q.iso}^{-1}]
\]
is precisely the limit of the Goodwillie tower of the $\infty$-category of dg $\mathcal{C}^*$-algebras, localized at quasi-isomorphisms, in the sense of \cite{heuts21}. See \cite[Section 6.1]{heuts21} for more details. 
\end{remark}

\begin{remark}
There is a more general version of Theorem \ref{thm: arity decomposition of the infinity category of complete C algebras}, where one does not assume the conditions in Assumptions \ref{assumption 1}. It gives an adjunction of $\infty$-categories
\[
\begin{tikzcd}[column sep=6pc,row sep=3pc]
            \mathcal{C}\mbox{-}\mathsf{alg}^{\mathsf{comp}}~[\mathsf{Q.iso}^{-1}] \arrow[r,"(\mathbb{L}(i_\omega)_!(-))_{\omega \geq 1}"{name=F},shift left=1.1ex ] &\lim_{\omega \geq 1}^h \mathscr{R}_\omega\mathcal{C}\mbox{-}\mathsf{alg}~[\mathsf{Q.iso}^{-1}]~, \arrow[l,"\lim_{\omega \geq 1}^h (i_\omega)^*(-)"{name=U},shift left=1.1ex ]
            \arrow[phantom, from=F, to=U, , "\dashv" rotate=-90]
\end{tikzcd}
\]
where $\mathscr{R}_\omega$ denotes the $\omega$-stage of the coradical filtration of the conilpotent dg cooperad $\C$, where the left adjoint functor is again fully faithful. However, in general the weight filtration induced by the coradical filtration of $\C$ and the arity filtration induced by the truncations of $\C$ differ, and one neither identify the $\infty$-categories in this limit with those of algebras over operads nor compare it to the Goodwillie tower of the $\infty$-category of dg $\mathcal{C}^*$-algebras up to quasi-isomorphisms.
\end{remark}

The limit considered in Theorem \ref{thm: arity decomposition of the infinity category of complete C algebras}
\[
\lim_{k \geq 1} \tau^{\leq k}\mathcal{C}^*\mbox{-}\mathsf{alg}~[\mathsf{Q.iso}^{-1}]
\]
is a priori taken in the $\infty$-category of presentable $\infty$-categories and right adjoints above $\mathbf{D}(\kk)$. It coincides with the limit in the $\infty$-category of presentable $\infty$-categories and right adjoints since the diagram is contractible and thus can also be computed in the $\infty$-category of $\infty$-categories. 

\begin{corollary}\label{cor: infinity monadicity of absolute algebras}
Let $\C$ be a conilpotent dg cooperad satisfying Assumptions \ref{assumption 1}. The Quillen adjunction 
\[
\begin{tikzcd}[column sep=5pc,row sep=3pc]
            \mathsf{Ch}(\kk) \arrow[r,"\widehat{\mathrm{S}}^c(\C)(-)"{name=F},shift left=1.1ex ] &   \mathcal{C}\mbox{-}\mathsf{alg}^{\mathsf{comp}} \arrow[l,"\mathrm{U}"{name=U},shift left=1.1ex ]
            \arrow[phantom, from=F, to=U, , "\dashv" rotate=-90]
\end{tikzcd}
\]
between complete dg $\C$-algebras, endowed with the model structure where weak equivalences are quasi-isomorphisms and fibrations degree-wise surjection, and the underlying category of chain complexes induces an adjunction at the level of $\infty$-categories
\[
\begin{tikzcd}[column sep=5pc,row sep=3pc]
            \mathbf{D}(\kk) \arrow[r,"\widehat{\mathrm{S}}^c(\C)(-)"{name=F},shift left=1.1ex ] &\mathcal{C}\mbox{-}\mathsf{alg}^{\mathsf{comp}}~[\mathsf{Q.iso}^{-1}] \arrow[l,"\mathrm{U}"{name=U},shift left=1.1ex ]
            \arrow[phantom, from=F, to=U, , "\dashv" rotate=-90]
\end{tikzcd}
\]
which is monadic.
\end{corollary}

\begin{proof}
The result essentially follows from Theorem \ref{thm: arity decomposition of the infinity category of complete C algebras}. Indeed, the right hand side is a limit diagram of presentable $\infty$-categories and right adjoints over $\mathbf{D}(\kk)$, where all the functors are monadic in the $\infty$-categorical sense by Lemma \ref{lemma: identification of the truncated category}. This limit is therefore still monadic over $\mathbf{D}(\kk)$ by \cite[Theorem 1.3]{Heine17}. Since complete dg $\C$-algebras up to quasi-isomorphisms are a coreflective $\infty$-subcategory, colimits are computed in the ambient $\infty$-category. Therefore the forgetful functor from complete dg $\C$-algebras up to quasi-isomorphisms to $\mathbf{D}(\kk)$ preserves geometric realizations of $\mathrm{U}$-split  simplicial objects, and since it is obviously conservative, it is monadic by the Barr--Beck--Lurie theorem of \cite[Theorem 4.7.0.3.]{HigherAlgebra}. 
\end{proof}

\begin{corollary}\label{cor: rectifictation of dg C algebras}
Let $\C$ be a conilpotent dg cooperad satisfying Assumptions \ref{assumption 1}. There is an equivalence of $\infty$-categories
\[
\mathcal{C}\mbox{-}\mathsf{alg}^{\mathsf{comp}}~[\mathsf{Q.iso}^{-1}] \simeq \mathbf{Alg}_{\C}(\mathbf{D}(\kk))
\]
between complete dg $\C$-algebras up to quasi-isomorphisms and algebras over the $\infty$-cooperad induced by the dg cooperad $\C$. 
\end{corollary}

\begin{proof}
Since both $\infty$-categories are monadic over the $\infty$-category of $\mathbf{D}(\kk)$, it suffices by the Barr--Beck--Lurie theorem (see \cite[Corollary 4.7.3.16.]{HigherAlgebra}) to show that the two monads are naturally weakly equivalent. This, again, is immediate, since the point-set dual Schur functor $\widehat{\mathrm{S}}^c(\C)$ preserves quasi-isomorphisms, and thus induces the $\infty$-categorical dual Schur functor of Paragraph \ref{subsubsection: infinity cat dual Schur functor}. 
\end{proof}

\subsection{Identifying the monad via point-set models}
Let $\mathcal{P}$ be a reduced dg operad. Since we want to give point-set models for the $\infty$-category of coalgebras over its underlying $\infty$-operad, we need to further assume that is is \textit{cofibrant} as an dg operad. Without loss of generality, we can assume that $\mathcal{P} = \Omega \C$ for some reduced conilpotent dg cooperad $\C$. We further impose that $\C$ satisfies Assumptions \ref{assumption 1}. We are first going to identify the monad $\mathbf{Q}$ introduced in Subsection \ref{Subsection: Complete enhanced bar-cobar adjunction} with the $\infty$-category of algebras over the cooperad $\C$, and then this will allows us to give point-set models for the complete enhanced bar-cobar adjunction. 

\medskip

In the context above, the dg operad $\Omega\C$ is augmented, and the augmentation induces a Quillen adjunction 
\[
\begin{tikzcd}[column sep=5pc,row sep=3pc]
            \mathsf{Ch}(\kk) \arrow[r,"\mathrm{triv}"{name=F},shift left=1.1ex ] &   \Omega\mathcal{C}\mbox{-}\mathsf{coalg} \arrow[l,"\mathrm{prim}"{name=U},shift left=1.1ex ]
            \arrow[phantom, from=F, to=U, , "\dashv" rotate=-90]
\end{tikzcd}
\]
at the point-set level between dg $\Omega \C$-coalgebras and chain complexes, where dg $\Omega \C$-coalgebras are endowed with a left transferred model structure along the cofree-forgetful adjunction, where cofibrations are given by monomorphisms and weak equivalences by quasi-isomorphisms. The functor $\mathrm{triv}$ endows a chain complex with a trivial dg $\Omega \C$-coalgebra structure and its adjoint $\mathrm{prim}$ computes the primitive elements in a dg $\Omega \C$-coalgebra.

\begin{lemma}\label{lemma: rectification of derived primitives}
Let $\C$ be a conilpotent dg cooperad satisfying Assumptions \ref{assumption 1}. The derived adjunction 
\[
\begin{tikzcd}[column sep=5pc,row sep=3pc]
            \Omega\mathcal{C}\mbox{-}\mathsf{coalg}~[\mathsf{Q.iso}^{-1}]  \arrow[r,"\mathbb{R}\mathrm{prim}"{name=F},shift left=1.1ex ] &   \mathbf{D}(\kk) \arrow[l,"\mathrm{triv}"{name=U},shift left=1.1ex ]
            \arrow[phantom, from=F, to=U, , "\dashv" rotate=-90]
\end{tikzcd}
\]
is a model for the $\infty$-categorical derived primitives-trivial adjunction 
\[
\begin{tikzcd}[column sep=5pc,row sep=3pc]
         \mathbf{Coalg}_{\Omega\mathcal{C}}(\mathbf{D}(\kk)) \arrow[r, shift left=1.1ex, "\mathbf{prim}"{name=F}] &\mathbf{D}(\kk)~, \arrow[l, shift left=.75ex, "\mathbf{triv}"{name=U}]
            \arrow[phantom, from=F, to=U, , "\dashv" rotate=90]
\end{tikzcd}
\]

between coalgebras over the underlying $\infty$-operad of $\Omega\mathcal{C}$ and the $\infty$-category of chain complexes over $\kk$.
\end{lemma} 

\begin{proof}
The main key point in the proof is the comparison between the $\infty$-categories of dg $\Omega\mathcal{C}$-coalgebras up to quasi-isomorphisms and coalgebras defined over the underlying $\infty$-operad of $\Omega\mathcal{C}$. In \cite{rectification}, we showed that since the dg operad $\Omega\mathcal{C}$ is \textit{cofibrant as an operad}, there is an equivalence of $\infty$-categories
\[
\Omega\mathcal{C}\mbox{-}\mathsf{coalg}~[\mathsf{Q.iso}^{-1}] \simeq \mathbf{Coalg}_{\Omega\mathcal{C}}(\mathbf{D}(\kk))~. 
\]
Under this equivalence, it is straightforward to see that the (derived) functor $\mathrm{triv}$ and the functor $\mathbf{triv}$ agree, and thus the two adjunction are naturally equivalent. 
\end{proof}

\begin{theorem}\label{thm: identification of the infinity monad Q}
Let $\C$ be a conilpotent dg cooperad satisfying Assumptions \ref{assumption 1}. There is a natural weak equivalence of monads 
\[
\mathbf{Q}_{\Omega\C} \simeq \widehat{\mathbf{S}}^c(\C) \simeq \prod_{n \geq 0} \left[\C(n), (-)^{\otimes n}\right]^{h\mathbb{S}_n}
\]
between the monad $\mathbf{Q}_{\Omega\C}$ introduced in \S \ref{Subsection: Complete enhanced bar-cobar adjunction}, which factors the derived primitives-trivial adjunction of $\Omega\C$-coalgebras, and the monad obtained by applying the $\infty$-categorical dual Schur functor to the underlying $\infty$-cooperad of $\C$. 
\end{theorem}

\begin{proof}
By \cref{lemma: rectification of derived primitives}, we can compute the monad induced by the $\mathbf{triv} \dashv \mathbf{prim}$ adjunction using the adjunction $\mathrm{triv} \dashv \mathbb{R}\mathrm{prim}$. For any chain complex $V$, we have that 
\[
\mathbb{R}\mathrm{prim} \circ \mathrm{triv}(V) \simeq \mathrm{prim} \circ \widehat{\mathrm{B}}_\iota \widehat{\Omega}_\iota(V) \simeq \widehat{\mathrm{S}}^c(\C)(V)~, 
\]
since the structure of $V$ is trivial. The latter functor models $\widehat{\mathbf{S}}^c(\C)$ with its monad structure. 
\end{proof}

\begin{remark}
If we start with a cofibrant dg operad $\mathcal{P}$ in general, we can always replace it by the cofibrant operad $\Omega\mathrm{B}\mathcal{P}$, so $\C$ can be taken to be $\mathrm{B}\mathcal{P}$ in the above theorem. Since by Remark \ref{rmk: point-set bar and infinity cat bar}, the point-set operadic bar construction presents Lurie's bar construction, we also get that, up to equivalence, the monad $\mathbf{Q}_\mathcal{P}$ can be identified to the monad obtained by applying the $\infty$-categorical dual Schur functor to the $\infty$-operad $\mathbf{bar}\mathcal{P}$.
\end{remark}

\begin{remark}
Another consequence of Theorem \ref{thm: identification of the infinity monad Q} is that the totalization of the cosimplicial  object 
\[
\small{
\begin{tikzcd}
\cdots \arrow[r,shift left= 2ex] \arrow[r,shift right=2ex]  \arrow[r]
&\mathbf{L}(\Omega\C)^{\circ 3}(\mathbf{triv}(V))  \arrow[l,shift left=3ex] \arrow[l,shift right=3ex]  \arrow[l,shift left=1ex] \arrow[l,shift right=1ex] \arrow[r,shift left=1ex] \arrow[r,shift right=1ex] 
&\mathbf{L}(\Omega\C)^{\circ 2}(\mathbf{triv}(V)) \arrow[l,shift left=2ex] \arrow[l,shift right=2ex] \arrow[l] \arrow[r]
&\mathbf{L}(\Omega\C)(\mathbf{triv}(V))  \arrow[l,shift left=1.1ex] \arrow[l,shift right=1.1ex] 
\end{tikzcd}}
\]


coincides with $\widehat{\mathbf{S}}^c(\C)(V)$ for any chain complex $V$, where here $\mathbf{L}(\Omega\C)$ denotes the $\infty$-categorical cofree $\Omega\C$-coalgebra comonad. Indeed, using analogous arguments as in the proof of Proposition \ref{prop: identification of the derived indecomposables comonad}, it can be seen that the totalization of this cosimplicial object agrees with the value of the monad induced by the adjunction $\mathbf{triv} \circ \mathbf{prim}$ associated to the $\infty$-category of $\Omega\C$-coalgebras. 
\end{remark}

\begin{theorem}\label{thm: point-set models for complete enhanced bar-cobar adjunctions}
Let $\C$ be a conilpotent dg cooperad satisfying Assumptions \ref{assumption 1}. The $\infty$-categorical adjunction induced by the complete bar-cobar adjunction with respect to $\iota$ composed with the right Bousfield localization with respect to quasi-isomorphisms
\[
\begin{tikzcd}[column sep=4pc,row sep=3pc]
           \Omega\mathcal{C}\mbox{-}\mathsf{coalg}~[\mathsf{Q.iso}^{-1}] \arrow[r, shift left=1.1ex, "\widehat{\Omega}_{\iota}"{name=F}] 
          &\C\mbox{-}\mathsf{alg}^{\mathsf{comp}}~[\mathsf{W}^{-1}] \arrow[l, shift left=.75ex, "\widehat{\mathrm{B}}_{\iota}"{name=U}] \arrow[r, shift left=1.1ex, "\mathrm{Id}"{name=A}]
          &\C\mbox{-}\mathsf{alg}^{\mathsf{comp}}~[\mathsf{Q.iso}^{-1}] \arrow[l, shift left=.75ex, "\mathbb{L}\mathrm{Id}"{name=B}]
            \arrow[phantom, from=F, to=U, , "\simeq" rotate=0] \arrow[phantom, from=A, to=B, , "\dashv" rotate=90]
\end{tikzcd}
\]
is naturally weakly equivalent to the $\infty$-categorical complete enhanced bar-cobar adjunction 
\[
\begin{tikzcd}[column sep=5pc,row sep=3pc]
         \mathbf{Coalg}_{\Omega\mathcal{C}}(\mathbf{D}(\kk)) \arrow[r, shift left=1.1ex, "\bm{\widehat{\mathrm{B}}}^{\mathbf{enh}}"{name=F}] &\mathbf{Alg}_{\C}(\mathbf{D}(\kk))~. \arrow[l, shift left=.9ex, "\bm{\widehat{\Omega}}^{\mathbf{enh}}"{name=U}] \arrow[phantom, from=F, to=U, , "\dashv" rotate=90]
\end{tikzcd}
\]
\end{theorem}

\begin{proof}
Directly follows from \cref{thm: identification of the infinity monad Q} together with \cref{cor: rectifictation of dg C algebras}, using the fact that both are universal factorizations of the adjunction $\mathbf{triv} \dashv \mathbf{prim}$ for $\Omega\mathcal{C}$-coalgebras. 
\end{proof}

\begin{remark}
Over a positive characteristic field, Theorem \ref{thm: point-set models for enhanced bar-cobar} should hold for any \textit{quasi-planar} conilpotent dg cooperad $\C$ satisfying \ref{assumption 1}, where one replaces complete dg $\C$-algebras with \textit{qp-complete} dg $\C$-algebras in the sense of \cite[Sections 3]{premierpapier}. 
\end{remark}


\section{The \texorpdfstring{$\infty$}{infinity}-categorical inclusion-restriction squares}


In this section, we put together the previous pieces to construct an $\infty$-categorical inclusion-restriction square, similar to the point-set inclusion-restriction square constructed in Section \ref{Section: point-set inclusion-restriction square}. We compare the resulting adjunction with the one constructed by Heuts in \cite{heuts2024}. Finally, we discuss how these results could be extended in positive characteristic.

\subsection{Comparing \texorpdfstring{$\infty$}{infinity}-categories of coalgebras} We establish different comparison functors and equivalences between point-set conilpotent (divided powers) coalgebras over a cooperad, conilpotent (divided powers) coalgebras over a $\infty$-cooperad, non-necessarily conilpotent (divided powers) coalgebras over a $\infty$-cooperad in the sense of \cite[Appendix A]{heuts2024}, and coalgebras over an $\infty$-operad, in the sense of \cite[Section 3]{rectification}. These results shall be useful in the subsequent sections. 

\begin{lemma}\label{lemma: equivalence between non-conilpotent coalgebras over a cooperad a la Heuts and coalgebras over the dual operad}
Let $\mathscr{C}$ be a reduced $\infty$-cooperad. There is an adjunction 
\[
\begin{tikzcd}[column sep=5pc,row sep=3pc]
         \mathbf{Coalg}_{\mathscr{C}}(\mathbf{D}(\kk)) \arrow[r, shift left=1.1ex, ""{name=F}] &\mathbf{Coalg}_{\mathscr{C}^*}(\mathbf{D}(\kk)) \arrow[l, shift left=.9ex, ""{name=U}] \arrow[phantom, from=F, to=U, , "\dashv" rotate=-90]
\end{tikzcd}
\]
between (divided powers) $\mathscr{C}$-coalgebras in the sense of \cite[Appendix A]{heuts2024} and $\mathscr{C}^*$-coalgebras in the sense of Definition \ref{def: infinity categorical P-coalgebras}. 

\medskip

Furthermore, if we assume that either $\mathscr{C}$ is non-symmetric or that $\kk$ is a characteristic zero field, then this adjunction adjunction is an equivalence when $\mathscr{C}$ has arity-wise finite dimensional homology. 
\end{lemma}

\begin{proof}
Both $\infty$-categories are defined as a pullbacks along the inclusion of $\mathbf{D}(\kk)$ into $\mathbf{pro}(\mathbf{D}(\kk))$ of $\infty$-categories of coalgebras in $\mathbf{pro}(\mathbf{D}(\kk))$, so it suffices to compare the latter. Divided powers $\mathscr{C}$-coalgebras in the sense of \cite[Appendix A]{heuts2024} in $\mathbf{pro}(\mathbf{D}(\kk))$ are coalgebras over the comonad (keeping the notation of \textit{op.cit.}) given by 
\[
\widehat{p\mathbf{Sym}}_{\mathscr{C}} = \lim_{k} p\left(\bigoplus_{n \geq 0}^{k} \left(\mathscr{C}(n) \otimes (-)^{\otimes n}\right)_{h\mathbb{S}_n}\right)~,
\]
where $p$ stands for the prolongation functor that sends an endofunctor of $\mathbf{D}(\kk)$ to an endofunctor of $\mathbf{pro}(\mathbf{D}(\kk))$. Very similarly, $\mathscr{C}^*$-coalgebras in $\mathbf{pro}(\mathbf{D}(\kk))$ are coalgebras over the comonad 
\[
\widehat{\mathbf{S}}^c_{\mathrm{pro}}(\mathscr{C}^*) =\lim_{k} p\left(\bigoplus_{n \geq 0}^k \left[\mathscr{C}^*(n),(-)^{\otimes n}\right]^{h\mathbb{S}_n}\right)~, 
\]
where $[-,-]$ refers to the internal hom of $\mathbf{D}(\kk)$. For any $k \geq 1$, there is a map of endofunctors
\[
\iota_{\mathscr{C}}^{\leq k}: \bigoplus_{n \geq 0}^{k} \left(\mathscr{C}(n) \otimes (-)^{\otimes n}\right)_{h\mathbb{S}_n} \longrightarrow \bigoplus_{n \geq 0}^k \left[\mathscr{C}^*(n),(-)^{\otimes n}\right]^{h\mathbb{S}_n} 
\]
given by the natural map $V^* \otimes W \longrightarrow [V,W]$ together with the double dual map and the norm map from coinvariants to invariants. It induces a map of comonads
\[
\lim_k p(\iota_{\mathscr{C}}^{\leq k}): \widehat{p\mathbf{Sym}}_{\mathscr{C}} \longrightarrow \widehat{\mathbf{S}}^c_{\mathrm{pro}}(\mathscr{C}^*)~. 
\]
And this morphism of comonads in $\mathbf{pro}(\mathbf{D}(\kk))$ induces the desired adjunction, whose pullback is the one stated in the lemma. 

\medskip

Moreover, if $\mathscr{C}$ has arity-wise finite dimensional homology, the maps $\iota_{\mathscr{C}}^{\leq k}$ are all equivalences since $\mathscr{C}(n)$ has total finite dimensional homology and since the norm maps are all equivalences. Hence the induced map of comonads is an equivalence and the two $\infty$-categories of coalgebras are equivalent.
\end{proof}

\begin{remark}
There are two main differences between the definition of coalgebras considered in \cite[Appendix A]{heuts2024} and the definition considered in \cite{rectification}: 

\begin{enumerate}
    \item When working in a general context, the first one encodes \textit{divided powers} non-necessarily conilpotent types of coalgebras, whereas the second does not involve \textit{divided powers}. 

    \item The first has an extra finiteness condition that the later does not have in the non-finite dimensional case. Indeed, assume that $\mathscr{C}(n)$ is infinite dimensional. Then, at a \textit{given arity}, decompositions in the first definition land on the tensor product 
    \[
    \Delta_C(n): C \longrightarrow \mathscr{C}(n) \otimes C^{\otimes n}~,
    \]
    hence are finite sums of pure tensors, whereas in the second definition, they land on the internal hom, so "infinite sums of pure tensors" are allowed.
\end{enumerate}
\end{remark}

The previous preliminary result allows us to construct a "forgetful functor" or "inclusion functor" from conilpotent divided powers coalgebras over an enriched $\infty$-cooperad and coalgebras over its linear dual  $\infty$-operad. Note, however, that it is not a priori obvious whether this $\infty$-categorical "inclusion functor" is fully faithful or not, as explained in \cite[Remark 11.2]{heuts2024}. This fully faithfulness was conjectured in \cite[Conjecture 2.8.4]{GaitsgoryRozenblyumVolII}. 

\begin{lemma}\label{lemma: forgetting divided powers and conilpotentcy}
Let $\mathscr{C}$ be a reduced $\infty$-cooperad. There is an adjunction 
\[
\begin{tikzcd}[column sep=5pc,row sep=3pc]
         \mathbf{Coalg}_{\mathscr{C}}^{\mathbf{conil}}(\mathbf{D}(\kk)) \arrow[r, shift left=1.1ex, "\mathbf{Inc}"{name=F}] &\mathbf{Coalg}_{\mathscr{C}^*}(\mathbf{D}(\kk)) \arrow[l, shift left=.9ex, "\mathbf{Sub}"{name=U}] \arrow[phantom, from=F, to=U, , "\dashv" rotate=-90]
\end{tikzcd}
\]
between conilpotent (divided powers) $\mathscr{C}$-coalgebras and $\mathscr{C}^*$-coalgebras. 
\end{lemma}

\begin{proof}
We compose the adjunction constructed between conilpotent divided powers coalgebras over $\mathscr{C}$ in the sense of Definition \ref{def: infinity cat C-coalgebras} and divided powers coalgebras over $\mathscr{C}$ in the sense of \cite[Appendix A]{heuts2024}, constructed in \cite[Section 7]{heuts2024} with the adjunction of Lemma \ref{lemma: equivalence between non-conilpotent coalgebras over a cooperad a la Heuts and coalgebras over the dual operad}. 
\end{proof}

\begin{remark}
The left adjoint functor could also be constructed inductively like in \cite[Section 7]{heuts2024}, using the analoguous decomposition of the $\infty$-category of $\mathscr{C}^*$-coalgebras as a limit, which holds because of \cite[Proposition 3.16]{rectification}. 
\end{remark}

Finally, we put together the comparison functors constructed so far to obtain a commuting triangle of functors between the point-set version of conilpotent divided powers coalgebras, their $\infty$-categorical counterparts and the $\infty$-category of all coalgebras of that given type.

\begin{proposition}\label{prop: commutative triangle of inclusions}
Let $\mathcal{P}$ be a dg operad satisfying Assumptions \ref{assumption 1}. The following triangle of adjunctions of $\infty$-categories
\[
\begin{tikzcd}[column sep=4pc,row sep=3.5pc]
\mathrm{B}\mathcal{P}\mbox{-}\mathsf{cog}~[\mathsf{Q.iso}^{-1}] \hspace{1pc} \arrow[r, shift left=1.5ex, "\mathbf{loose}"{name=A}]  \arrow[rd, shift left=0ex, shorten >=10pt, shorten <=10pt, "\mathrm{Inc}"{name=B}] 
&\mathbf{Coalg}_{\mathbf{bar}(\mathcal{P})}^{\mathbf{conil}}(\mathbf{D}(\kk)) \arrow[d, shift left=2ex, "\mathbf{Inc}"{name=F}] \arrow[l, shift left=.75ex, "\mathbf{strict}"{name=C}] \\
&\Omega\mathcal{P}^*\mbox{-}\mathsf{cog}~[\mathsf{Q.iso}^{-1}] \simeq \mathbf{Coalg}_{\mathbf{cobar}(\mathcal{P}^*)}(\mathbf{D}(\kk)) \arrow[u, shift left=1.1ex, "\mathbf{Sub}"{name=U}] \arrow[lu, shift left=2.5ex, shorten >=-5pt, shorten <=30pt, "\mathbb{R}\mathrm{Sub}"{name=D}] 
\end{tikzcd}
\]
commutes, where left adjoints are drawn on the top and on the left.  
\end{proposition}

\begin{proof}
The vertical right adjunction is given by Lemma \ref{lemma: forgetting divided powers and conilpotentcy}, where $(\mathbf{bar}(\mathcal{P}))^* \simeq \mathbf{cobar}(\mathcal{P}^*)$ since $\mathcal{P}$ satisfies Assumptions \ref{assumption 1}. The top horizontal adjunction is constructed in Proposition \ref{prop: adjunction between strict and loose conilpotent coalgebras}. The diagonal adjunction is given by the derived inclusion-maximal conilpotent subcoalgebra Quillen adjunction of Lemma \ref{lemma: restriction and inclusion are Quillen functors}, under the equivalence of $\infty$-categories
\[
\Omega\mathcal{P}^*\mbox{-}\mathsf{cog}~[\mathsf{Q.iso}^{-1}] \simeq \mathbf{Coalg}_{\mathbf{cobar}(\mathcal{P}^*)}(\mathbf{D}(\kk))~,
\]
which is given by Theorem \ref{thm: rectification of P-coalgebras}. 

\medskip

If we compse the cofree-forgetful adjunction between $\Omega\mathcal{P}^*$-coalgebras and chain complexes with either the $\mathrm{Inc} \dashv \mathbb{R}\mathrm{Sub}$ adjunction or the $\mathbf{Inc} \dashv \mathbf{Sub}$ adjunction, we get, in both cases, the same comonad $\mathbf{S}(\mathbf{bar}(\mathcal{P}))$ on chain complexes. Hence, since the $\infty$-category $\mathbf{Coalg}_{\mathbf{bar}(\mathcal{P})}^{\mathbf{conil}}(\mathbf{D}(\kk))$ is precisely the category of coalgebras over this comonad, it is the terminal object in the $\infty$-category of adjunctions that produce the comonad $\mathbf{S}(\mathbf{bar}(\mathcal{P}))$, and thus there is an essentially unique adjunction between $\mathrm{B}\mathcal{P}$-coalgebras and $\mathbf{Coalg}_{\mathbf{bar}(\mathcal{P})}^{\mathbf{conil}}(\mathbf{D}(\kk))$ which makes this triangle commute, which is $\mathbf{loose} \dashv \mathbf{strict}$.
\end{proof}

\subsection{The \texorpdfstring{$\infty$}{infinity}-categorical inclusion-restriction square} We use the point-set version of the inclusion-restriction square of Theorem \ref{thm: full rectangle infinity categorical} together with the comparisons established in  \cref{thm: point-set models for enhanced bar-cobar} and \cref{thm: point-set models for complete enhanced bar-cobar adjunctions} to construct an $\infty$-categorical version of the inclusion-restriction square. We suspect that a result like this should holds in a much more general context, e.g.~for enriched $\infty$-operads in any compactly  rigidly generated stable symmetric monoidal $\infty$-category. 

\begin{theorem}\label{thm: infinity categorical inclusion-restriction square}
Let $\mathscr{P}$ be a reduced $\infty$-operad with arity-wise bounded and finite dimensional homology groups. The following diagram of adjunctions 

\[
\begin{tikzcd}[column sep=5pc,row sep=5pc]
\mathbf{Alg}_{\mathscr{P}}(\mathbf{D}(\kk)) \arrow[r,"\bm{\mathrm{B}}^{\mathbf{enh}}"{name=B},shift left=1.1ex] \arrow[d,"\mathbf{cAbs}"{name=SD},shift left=1.1ex ]
&\mathbf{Coalg}_{\mathbf{bar}(\mathscr{P})}^{\mathbf{conil}}(\mathbf{D}(\kk)), \arrow[d,"\mathbf{Inc}"{name=LDC},shift left=1.1ex ] \arrow[l,"\bm{\Omega}^{\mathbf{enh}}"{name=C},,shift left=1.1ex]  \\
\mathbf{Alg}_{\mathscr{P}^*}(\mathbf{D}(\kk)) \arrow[r,"\bm{\widehat{\mathrm{B}}}^{\mathbf{enh}}"{name=CC},shift left=1.1ex]  \arrow[u,"\mathbf{Res}"{name=LD},shift left=1.1ex ]
&\mathbf{Coalg}_{\mathbf{cobar}(\mathscr{P}^*)}(\mathbf{D}(\kk)) ~, \arrow[l,"\bm{\widehat{\Omega}}^{\mathbf{enh}}"{name=CB},shift left=1.1ex] \arrow[u,"\mathbf{Sub}"{name=TD},shift left=1.1ex] \arrow[phantom, from=SD, to=LD, , "\dashv" rotate=180] \arrow[phantom, from=C, to=B, , "\dashv" rotate=-90]\arrow[phantom, from=TD, to=LDC, , "\dashv" rotate=180] \arrow[phantom, from=CC, to=CB, , "\dashv" rotate=-90]
\end{tikzcd}
\] 

commutes in the following sense: left adjoints from the top left to the bottom right are naturally weakly equivalent. 
\end{theorem}

\begin{proof}
The two horizontal adjunctions correspond to the enhanced bar-cobar adjunction, defined in Subsection \ref{Subsection: enhanced bar-cobar adjunction} and its dual version, defined in Subsection \ref{Subsection: Complete enhanced bar-cobar adjunction}. The right vertical adjunction is given by Lemma \ref{lemma: forgetting divided powers and conilpotentcy} and the left vertical adjunction is induced the canonical morphism of monads
\[
\iota_{\mathscr{P}}: \bigoplus_{n \geq 0} \left(\mathscr{P}(n) \otimes (-)^{\otimes n}\right)_{h\mathbb{S}_n} \longrightarrow \prod_{n \geq 0} \left[\mathscr{P}^*(n),(-)^{\otimes n}\right]^{h\mathbb{S}_n}
\]
in a completely analogous way to Proposition \ref{Prop: absolution restriction adjunction}. We are left to check that this square of adjunctions does, indeed, commute. 

\medskip

The first part of the proof consists in showing that any reduced $\infty$-operad $\mathscr{P}$ with arity-wise bounded and finite dimensional homology groups admits a point-set model $\mathcal{P}$, which is a dg operad stasifying Assumptions \ref{assumption 1}, so that we can apply Theorem \ref{thm: full rectangle infinity categorical} to $\mathcal{P}$. 

\medskip

Since the $\infty$-category of $\infty$-operads is presented by the category of dg operads localized at quasi-isomorphisms \cite{HinichRectification}, there exists a point-set model $\mathcal{Q}$ for this $\infty$-operad. Now, let us explain why, for any choice of $\mathcal{Q}$, we can construct another model $\mathcal{P}$ which is a reduced dg operad such that each $\mathcal{P}(n)$ is a bounded and degree-wise finite dimensional chain complex for all $n \geq 0$. First, we can apply the homotopy transfer theorem to $\mathcal{Q}$ to get an up-to-homotopy operad structure on $\mathrm{H}_*(\mathcal{Q})$ (see \cite{vanderlaanthesis,granaker}). This gives in particular a weak equivalence of conilpotent dg cooperads between $\mathrm{B}\mathrm{H}_*(\mathcal{Q})$ and $\mathrm{B}(\mathcal{Q})$. Therefore, $\Omega\mathrm{B}\mathrm{H}_*(\mathcal{Q})$ is quasi-isomorphic to $\Omega\mathrm{B}(\mathcal{Q})$ and thus to $\mathcal{Q}$. Moreover, each $\Omega\mathrm{B}\mathrm{H}_*(\mathcal{Q})(n)$ is a bounded and degree-wise finite dimensional chain complex since the homology of $\mathcal{Q}(n)$ is by assumption bounded and finite dimensional. Therefore we can set $\mathcal{P} = \Omega\mathrm{B}\mathrm{H}_*(\mathcal{Q})$, which satisfies Assumptions \ref{assumption 1}, and apply Theorem \ref{thm: full rectangle infinity categorical}. 

\medskip

In this way, by applying Theorem \ref{thm: full rectangle infinity categorical}, together with Proposition \ref{prop: commutative triangle of inclusions} and Theorems \ref{thm: point-set models for enhanced bar-cobar} and \ref{thm: point-set models for complete enhanced bar-cobar adjunctions}, we get the following natural equivalences: 
\[
\mathbf{Inc} \circ \bm{\mathrm{B}}^{\mathbf{enh}} \simeq \mathbf{bar}_{\mathcal{P}} \simeq \mathrm{Inc} \circ \mathrm{B}_\pi \simeq \widehat{\mathrm{B}}_\iota \circ \mathbb{L}\mathrm{Id} \circ \mathbb{L}c\mathrm{Abs} \simeq \bm{\widehat{\mathrm{B}}}^{\mathbf{enh}} \circ \mathbf{cAbs}~, 
\]
making the diagram commute, where we use the fact that the Quillen adjunction of Proposition \ref{Prop: absolution restriction adjunction} is a point-set model for its $\infty$-categorical counterpart, using directly Corollary \ref{cor: rectifictation of dg C algebras}.
\end{proof}

\begin{remark}
It follows from the above that the $\infty$-categorical bar-cobar adjunction which follows from Theorem \ref{thm: infty categorical operadic bar cobar} admits the following $\infty$-categorical description
\[
\mathbf{bar}_{\mathcal{P}} \simeq \mathbf{Inc} \circ \bm{\mathrm{B}}^{\mathbf{enh}} \quad \text{and} \quad \mathbf{cobar}_{\mathcal{P}} \simeq \mathbf{Res} \circ \bm{\widehat{\Omega}}^{\mathbf{enh}}~, 
\]
and, vice-versa, these two composites admit explicit point-set models given in Theorem \ref{thm: infty categorical operadic bar cobar}.  
\end{remark}

\subsection{A comparison with the adjunction constructed by Heuts} We compare the composite adjunction given by Theorem \ref{thm: infinity categorical inclusion-restriction square} with the adjunction considered in \cite{heuts2024}, which we first recall. This comparison shall also be useful in Section \ref{Section: Derived Koszul equivalences} in order to compare our results with those obtained in \textit{op.cit.} Heuts considers more generally operads and cooperads in a presentably symmetric monoidal stable $\infty$-category $\mathbf C$. We work in the special case of $\mathbf C= \mathbf D(\kk)$. 

\begin{theorem}\label{prop: comparison of the composition adjunction with heuts' adjunction}
Let $\mathcal{P}$ be a dg operad satisfying Assumptions \ref{assumption 1}. The adjunction of $\infty$-categories 
\[
\begin{tikzcd}[column sep=4pc,row sep=3pc]
          { \text{$\mathcal{P}$}}\mbox{-}\mathsf{alg}~[\mathsf{Q.iso}^{-1}] \arrow[r,"\mathbf{bar}_\mathcal{P}"{name=F},shift left=1.1ex ] & \Omega\mathcal{P}^* \mbox{-}\mathsf{coalg}~[\mathsf{Q.iso}^{-1}] \arrow[l,"\mathbf{cobar}_\mathcal{P}"{name=U},shift left=1.1ex ]
            \arrow[phantom, from=F, to=U, , "\dashv" rotate=-90]
\end{tikzcd}
\]
between the $\infty$-category of dg $\mathcal{P}$-algebras localized at quasi-isomorphisms and the $\infty$-category of dg $\Omega\mathcal{P}^*$-coalgebras localized at quasi-isomorphisms coincides with the adjunction considered in \cite[Theorem 2.1]{heuts2024} for the underlying $\infty$-operad of $\mathcal{P}$, which is defined as the following composition
\[
\begin{tikzcd}[column sep=5pc,row sep=3pc]
         \mathbf{Alg}_{\mathcal{P}}(\mathbf D(\kk)) \arrow[r, shift left=1.1ex, "\mathbf{B}^{\mathbf{enh}}"{name=F}] &\mathbf{Coalg}_{\mathbf{bar}(\mathcal{P})}^{\mathbf{conil}}(\mathbf D(\kk)) \arrow[l, shift left=.75ex, "\bm{\Omega}^{\mathbf{enh}}"{name=U}] \arrow[phantom, from=F, to=U, , "\dashv" rotate=-90] \arrow[r, shift left=1.1ex, "\mathbf{Inc}"{name=FF}]
&\mathbf{Coalg}_{\mathbf{bar}(\mathcal{P})}(\mathbf D(\kk))~.  \arrow[l, shift left=.75ex, "\mathbf{Sub}"{name=UU}] \arrow[phantom, from=FF, to=UU, , "\dashv" rotate=-90]
\end{tikzcd}
\]
\end{theorem}

\begin{proof}
Since $\mathcal{P}$ satisfies Assumptions \ref{assumption 1}, we can use Lemma \ref{lemma: equivalence between non-conilpotent coalgebras over a cooperad a la Heuts and coalgebras over the dual operad} to get an equivalence between the target coalgebra category considered in \cite{heuts2024} and ours: 
\[
\mathbf{Coalg}_{\mathbf{bar}(\mathcal{P})}(\mathbf D(\kk)) \simeq \mathbf{Coalg}_{\mathbf{cobar}(\mathcal{P}^*)}(\mathbf D(\kk))~. 
\]
Then, the result essentially follows from the respective definitions, using the Theorem \ref{thm: point-set models for enhanced bar-cobar} and Proposition \ref{prop: commutative triangle of inclusions}. 
\end{proof}

\subsection{About the positive characteristic situation}\label{Subsection: positive char extension}
The constructions of \cite{heuts2024} are valid in any presentable stable symmetric monoidal $\infty$-category $\mathbf{C}$. Let $\mathscr{P}$ be an $\infty$-operad enriched in $\mathbf{C}$. In this setting, the adjunction 
\[
\begin{tikzcd}[column sep=5pc,row sep=3pc]
         \mathbf{Alg}_{\mathscr{P}}(\mathbf{C}) \arrow[r, shift left=1.1ex, "\mathbf{B}^{\mathbf{enh}}"{name=F}] &\mathbf{Coalg}_{\mathbf{bar}(\mathscr{P})}^{\mathbf{pd},\mathbf{conil}}(\mathbf{C}) \arrow[l, shift left=.75ex, "\bm{\Omega}^{\mathbf{enh}}"{name=U}] \arrow[phantom, from=F, to=U, , "\dashv" rotate=-90] \arrow[r, shift left=1.1ex, "\mathbf{Inc}"{name=FF}]
&\mathbf{Coalg}_{\mathbf{bar}(\mathscr{P})}^{\mathbf{pd}}(\mathbf{C})~. \arrow[l, shift left=.75ex, "\mathbf{Sub}"{name=UU}] \arrow[phantom, from=FF, to=UU, , "\dashv" rotate=-90]
\end{tikzcd}
\]
land in $\infty$-categories of $\mathbf{bar}(\mathscr{P})$-coalgebras which now has a priori divided powers operations. These prevent the adjunction in Lemma \ref{lemma: equivalence between non-conilpotent coalgebras over a cooperad a la Heuts and coalgebras over the dual operad} from being an equivalence. 

\begin{notation}
We add the superscript $\mathbf{pd}$ in this subsection to keep track of divided powers operations.
\end{notation}

In this setting, we expect the above adjunction to fit in a more general inclusion-restriction square of the form
\begin{center}
\begin{equation}\label{eq: hypothetical incl-rest 1}
\begin{tikzcd}[column sep=5pc,row sep=5pc]
\mathbf{Alg}_{\mathscr{P}}(\mathbf{D}(\kk)) \arrow[r,"\bm{\mathrm{B}}^{\mathbf{enh}}"{name=B},shift left=1.1ex] \arrow[d,"\mathbf{cAbs}"{name=SD},shift left=1.1ex ]
&\mathbf{Coalg}_{\mathbf{bar}(\mathscr{P})}^{\mathbf{pd},\mathbf{conil}}(\mathbf{D}(\kk)), \arrow[d,"\mathbf{Inc}"{name=LDC},shift left=1.1ex ] \arrow[l,"\bm{\Omega}^{\mathbf{enh}}"{name=C},,shift left=1.1ex]  \\
\mathbf{Alg}_{\mathscr{P}^*}(\mathbf{D}(\kk)) \arrow[r,"\bm{\widehat{\mathrm{B}}}^{\mathbf{enh}}"{name=CC},shift left=1.1ex]  \arrow[u,"\mathbf{Res}"{name=LD},shift left=1.1ex ]
&\mathbf{Coalg}_{\mathbf{cobar}(\mathscr{P}^*)}^{\mathbf{pd}}(\mathbf{D}(\kk)) ~, \arrow[l,"\bm{\widehat{\Omega}}^{\mathbf{enh}}"{name=CB},shift left=1.1ex] \arrow[u,"\mathbf{Sub}"{name=TD},shift left=1.1ex] \arrow[phantom, from=SD, to=LD, , "\dashv" rotate=180] \arrow[phantom, from=C, to=B, , "\dashv" rotate=-90]\arrow[phantom, from=TD, to=LDC, , "\dashv" rotate=180] \arrow[phantom, from=CC, to=CB, , "\dashv" rotate=-90]
\end{tikzcd}
\end{equation}
\end{center}

where one starts with classical algebras over $\mathscr{P}$, and where divided powers appear in all the other $\infty$-categories in the square. Moreover, if one starts with divided powers algebras over $\mathscr{P}$, there should be another inclusion-restriction square of the form 

\begin{center}
\begin{equation}\label{eq: hypothetical incl-rest 2}
    \begin{tikzcd}[column sep=5pc,row sep=5pc]
\mathbf{Alg}_{\mathscr{P}}^{\mathbf{pd}}(\mathbf{D}(\kk)) \arrow[r,"\bm{\mathrm{B}}^{\mathbf{enh}}"{name=B},shift left=1.1ex] \arrow[d,"\mathbf{cAbs}"{name=SD},shift left=1.1ex ]
&\mathbf{Coalg}_{\mathbf{bar}(\mathscr{P})}^{\mathbf{conil}}(\mathbf{D}(\kk)), \arrow[d,"\mathbf{Inc}"{name=LDC},shift left=1.1ex ] \arrow[l,"\bm{\Omega}^{\mathbf{enh}}"{name=C},,shift left=1.1ex]  \\
\mathbf{Alg}_{\mathscr{P}^*}^{\mathbf{pd}}(\mathbf{D}(\kk)) \arrow[r,"\bm{\widehat{\mathrm{B}}}^{\mathbf{enh}}"{name=CC},shift left=1.1ex]  \arrow[u,"\mathbf{Res}"{name=LD},shift left=1.1ex ]
&\mathbf{Coalg}_{\mathbf{cobar}(\mathscr{P}^*)}(\mathbf{D}(\kk)) ~, \arrow[l,"\bm{\widehat{\Omega}}^{\mathbf{enh}}"{name=CB},shift left=1.1ex] \arrow[u,"\mathbf{Sub}"{name=TD},shift left=1.1ex] \arrow[phantom, from=SD, to=LD, , "\dashv" rotate=180] \arrow[phantom, from=C, to=B, , "\dashv" rotate=-90]\arrow[phantom, from=TD, to=LDC, , "\dashv" rotate=180] \arrow[phantom, from=CC, to=CB, , "\dashv" rotate=-90]
\end{tikzcd}
\end{equation}
\end{center}

Defining all of these $\infty$-categories and constructing such adjunctions is beyond the scope of the present paper. Let us simply explain how these two variants can be observed at the level of point-set models. 

\medskip

In positive characteristic, one could choose to construct a point-set inclusion-restriction square, like in Theorem \ref{thm: full rectangle infinity categorical}, starting with a \textit{quasi-planar} conilpotent dg cooperad $\C$ as defined in \cite{premierpapier}. Then, bottom horizontal adjunction in the aforementioned theorem would be a well-defined Quillen adjunction by \cite[Section 6]{premierpapier}. This Quillen adjunction should induce the bottom horizontal adjunction (\ref{eq: hypothetical incl-rest 2}), since dg $\Omega \C$-coalgebras up to quasi-isomorphisms encode $\infty$-categories of (non-necessarily conilpotent) coalgebras \textit{without} divided powers by the rectification result of \cite[Theorem C]{rectification}. On the top horizontal adjunctions of Theorem \ref{thm: full rectangle infinity categorical}, one would have algebras over the dg operad $\C^*$, which is now injective as an dg symmetric sequence since $\C$ is projective as a dg symmetric sequence. Therefore dg $\C^*$-algebras encode an $\infty$-category of divided powers algebras over an operad, see \cite{pdalgebras} for more on this type of algebras and their properties. However, dg $\C^*$-algebras a priori only have a semi-model structure, and it is not clear how to construct the top horizontal adjunction in (\ref{eq: hypothetical incl-rest 2}) at the point-set level.

\medskip

Another variant would be to start construct the inclusion-restriction square in Theorem \ref{thm: full rectangle infinity categorical} starting with a cofibrant dg operad $\mathcal{P}$. Then dg $\mathcal{P}$-algebras up to quasi-isomorphism would encode an $\infty$-category of classical algebras over its underlying $\infty$-operad since any cofibrant operad is projective as a dg symmetric sequence. In this context, the top horizontal adjunction of (\ref{eq: hypothetical incl-rest 1}) should be related to the Quillen adjunction constructed in \cite[Section 4]{premierpapier} in a similar way to what has been done in Section \ref{Section: Francis-Gaitsfory enhanced bar-cobar}. However, in this case $\Omega \mathcal{P}^*$ would be injective as a dg symmetric sequence and thus encode, in principle, an $\infty$-category of (non-necessarily conilpotent) divided powers coalgebras. However, in this case, we do not have a rectification result in the sense of \cite{rectification} because this dg operad would not be cofibrant as an operad. So constructing point-set models for the bottom horizontal adjunction in (\ref{eq: hypothetical incl-rest 1}) seems for the moment out of reach.

\begin{remark}
However \cite[Theorem 4.42]{pdalgebras} shows that, up to composing with the linear dual functor, one has point-set models for the linear dual of the bar construction as a  $\Omega \mathcal{P}^*$-algebra, which are divided powers algebras at the homotopical level.
\end{remark}


\section{Derived Koszul equivalences and good completions}\label{Section: Derived Koszul equivalences}


In this section, we define homotopy (co)completeness at the point-set level and we obtain an explicit point-set description of the equivalence in \cite[Theorem 2.1]{heuts2024}. We then discuss Heuts's notion of an operad having good completions (meaning that trivial algebras are homotopy complete) and we define the dual notion of a cofibrant operad having good cocompletions. We finally put together all our previous constructions to obtain several fully faithful functors in different directions. We also give counterexamples to some natural-sounding statements which however are ``too good to be true''.

\subsection{Homotopy (co)complete (co)algebras at the point-set level}
Let $\mathcal{P}$ be a dg operad satisfying Assumptions \ref{assumption 1}. We recall the definition of a homotopy complete dg $\mathcal{P}$-algebra, already considered in \cite{HarperHess13,CamposPetersen24,absolutealgebras}. We introduce the dual definition of a homotopy cocomplete dg $\Omega\mathcal{P}^*$-coalgebra. We show that these two definitions are in fact point-set descriptions of the notions introduced by \cite{heuts2024} of \textit{nilcomplete} algebras and \textit{conilcomplete} coalgebras. Almost tautologically, the $\infty$-categorical bar-cobar adjunction $\mathbf{bar}_{\mathcal{P}} \dashv \mathbf{cobar}_{\mathcal{P}}$ restricts to an equivalence between these two $\infty$-subcategories, which reproves \cite[Theorem 2.1]{heuts2024} where the ambient stable $\infty$-category is $\mathbf D(\kk)$. 

\subsubsection{Homotopy completeness}  Let us consider the restriction adjunction 
\[
\begin{tikzcd}[column sep=5pc,row sep=3pc]
            \mathcal{P}\mbox{-}\mathsf{alg} \arrow[r,"c\mathrm{Abs}"{name=F},shift left=1.1ex ] &\mathcal{P}^*\mbox{-}\mathsf{alg}^{\mathsf{comp}} \arrow[l,"\mathrm{Res}"{name=U},shift left=1.1ex ]
            \arrow[phantom, from=F, to=U, , "\dashv" rotate=-90]
\end{tikzcd}
\]
which is a Quillen adjunction by \cref{lemma: restriction and inclusion are Quillen functors}. Here we are considering the model structures where weak equivalences are given by quasi-isomorphisms on both sides. 

\begin{definition}[Homotopy complete dg $\mathcal{P}$-algebra]
Let $A$ be a dg $\mathcal{P}$-algebra. It is \textit{homotopy complete} if the derived unit of adjunction
\[
\mathbb{L}\eta_A: A \qi \mathrm{Res} \circ \mathbb{L}c\mathrm{Abs}(A)
\]
is a quasi-isomorphism.
\end{definition}

\begin{remark}
    The idea behind this definition is that the underived adjunction unit $\mathrm{Res}\circ c\mathrm{Abs}$ universally turns a $\mathcal P$-algebra to a complete $\mathcal P^*$-algebra, or ``absolute'' algebra. %
%
When $\mathcal{P} = \mathcal{C}om$, this notion of completeness coincides with the $I$-adic completion of augmented commutative algebras in the finitely generated case, and it is in fact  better behaved than the $I$-adic completion in the general case. See \cite[Section 3.4.1]{absolutealgebras} for more details. Coming back to the above definition, it is only natural to call a dg $\mathcal{P}$-algebra \textit{homotopy} complete in this sense when the \textit{derived} unit of adjunction is an equivalence. 

\end{remark}
\begin{remark}\label{point set description of homotopy complete}
    With our point-set models, we can make the condition of homotopy completeness extremely explicit on the point-set level. If $A$ is a $\mathcal P$-algebra, then the standard proof that $\Omega \mathrm B A \to A$ is a quasi-isomorphism uses that (disregarding the differential) we may write
\[ \Omega \mathrm B A = \bigoplus_{n>0} (\mathcal P \circ \mathrm B \mathcal P)(n) \otimes_{\mathbb S_n} A^{\otimes n}.\]
Now the chain complex $(\mathcal P \circ \mathrm B \mathcal P)(n)$ is acyclic for $n> 1$. It follows that there are  evident filtrations with respect to which $\Omega \mathrm B A \to A$ is a filtered quasi-isomorphism, which gives the result since the filtrations are bounded below and exhaustive. 
The condition of homotopy completeness of $A$ is now that the natural embedding of 
\[ \Omega \mathrm B A = \bigoplus_{k>0} \mathcal P(k) \otimes_{\mathbb S_k} \bigoplus_{i_1+\dots+i_k=n} \mathrm{Ind}_{\mathbb{S}_{i_1} \times \cdots \times \mathbb{S}_{i_k}}^{\mathbb{S}_{n}} \left( \mathrm{B}\mathcal{P}(i_1) \otimes \cdots \otimes \mathrm{B}\mathcal{P}(i_k) \right) \otimes_{\mathbb{S}_n} A^{\otimes n}\]
into
\[ \widehat \Omega \mathrm B A = \prod_{k>0} \mathcal P(k) \otimes_{\mathbb S_k} \bigoplus_{i_1+\dots+i_k=n} \mathrm{Ind}_{\mathbb{S}_{i_1} \times \cdots \times \mathbb{S}_{i_k}}^{\mathbb{S}_{n}} \left( \mathrm{B}\mathcal{P}(i_1) \otimes \cdots \otimes \mathrm{B}\mathcal{P}(i_k) \right) \otimes_{\mathbb{S}_n} A^{\otimes n}\]
is a quasi-isomorphism. It is still a filtered quasi-isomorphism, for the same filtration as considered before, but this filtration is no longer exhaustive on the target, and so the filtered quasi-isomorphism is not necessarily a quasi-isomorphism. In favorable situations homotopy completeness can be proven by introducing further filtrations, as in \cite[Section 5]{CamposPetersen24}.

\end{remark}

\begin{remark}
By definition, the derived complete absolution functor $\mathbb{L}c\mathrm{Abs}$ restricts to a fully faithful functor
\[
\mathbb{L}c\mathrm{Abs}: \mathcal{P}\mbox{-}\mathsf{alg}^{\mathsf{ho.comp}}~[\mathsf{Q.iso}^{-1}] \hookrightarrow \mathcal{P}^*\mbox{-}\mathsf{alg}^{\mathsf{comp}}~[\mathsf{Q.iso}^{-1}]
\]
between the $\infty$-subcategory of homotopy complete dg $\mathcal{P}$-algebras up to quasi-isomorphisms and the $\infty$-category of complete dg $\mathcal{P}^*$-algebras up to quasi-isomorphisms.  
\end{remark}

Let $\mathscr{P}$ be a enriched $\infty$-operad. Following \cite[Section 10]{heuts2024}, a $\mathscr{P}$-algebra $X$ is \textit{nilcomplete} if the canonical map
\[
X \longrightarrow \lim_k \mathbf{t}_k X~,
\]
is an equivalence, where $\mathbf{t}_k X$ is the truncation of $X$ at operations in arities bigger than $n$. The functors $\mathbf{t}_n(-)$ are obtained by considering the successive arity truncations of the enriched $\infty$-operad $\mathscr{P}$.

\begin{proposition}\label{prop: nilcompletion = homotopy completetion}
Let $\mathcal{P}$ be a dg operad satisfying Assumptions \ref{assumption 1}. Let $A$ be a dg $\mathcal{P}$-algebra. It is homotopy complete if and only if it is nilcomplete.
\end{proposition}

\begin{proof}
This result follows directly from  \cref{prop: comparison of the composition adjunction with heuts' adjunction}, since these two definitions are in fact two different descriptions of the unit of the $\infty$-categorical bar-cobar adjunction $\mathbf{bar}_{\mathcal{P}} \dashv \mathbf{cobar}_{\mathcal{P}}$. Indeed, on the one hand, the nilcompletion map is shown to be the unit of this adjunction in \cite[
Proposition 10.4]{heuts2024}. On the other hand, let $A$ be a dg $\mathcal{P}$-algebra. It is homotopy complete if the derived unit 
\[
\mathbb{L}\eta_A: A \qi \mathrm{Res} \circ \mathbb{L}c\mathrm{Abs}(A)
\]

is a quasi-isomorphism. Using the canonical cofibrant resolution $\Omega_\pi\mathrm{B}_\pi(A)$ of $A$, we compute that this derived unit is represented by the map 
\[
A \longrightarrow \mathrm{Res} \circ \widehat{\Omega}_\iota \circ \mathrm{Inc} \circ \mathrm{B}_\pi(A)~,
\]
and this also represents the (derived) unit of the $\infty$-categorical operadic bar-cobar adjunction $\mathbf{bar}_{\mathcal{P}} \dashv \mathbf{cobar}_{\mathcal{P}}$. Hence the two notions coincide. 
\end{proof}

\begin{remark}[An explicit comparison with \textit{nilcompletenes}]
There is actually a more illuminating way to understand Proposition \ref{prop: nilcompletion = homotopy completetion}. From the results of Subsection \ref{subsection: absolute algebras inside the limit}, we get a commuting triangle of Quillen adjunctions 
\[
\begin{tikzcd}[column sep=5pc,row sep=3.5pc]
&\hspace{1.5pc}\mathcal{P}^*\mbox{-}\mathsf{alg}^{\mathsf{comp}} \arrow[dd, shift left=1.1ex, "((i_k)_!(-))_{k \geq 1}"{name=F}] \arrow[ld, shift left=.75ex, "\mathrm{Res}"{name=C}]\\
\mathcal{P} \mbox{-}\mathsf{alg} \arrow[ru, shift left=1.5ex, "c\mathrm{Abs}"{name=A}]  \arrow[rd, shift left=1ex, "((i_k)_!(-))_{k \geq 1}"{name=B}] 
& \\
&\hspace{3.5pc} \lim_{k \geq 1} \tau^{\leq k}\mathcal{P}\mbox{-}\mathsf{alg} ~, \arrow[uu, shift left=.75ex, "\lim_{k \geq 1} (i_k)^*(-)"{name=U}] \arrow[lu, shift left=.75ex, "\lim_{k \geq 1} (i_k)^*(-)"{name=D}] 
\end{tikzcd}
\]
where left adjoints are drawn on the top and on the right, and where the vertical right adjunction is the one constructed in Subsection \ref{subsection: absolute algebras inside the limit}. Passing to derived functors, the first triangle of Quillen adjunctions gives the following commutative triangle of units of adjunctions 
\[
\begin{tikzcd}[column sep=2.5pc,row sep=2pc]
&\mathrm{Res} \circ \mathbb{L}c\mathrm{Abs}(A) \arrow[dd,"\eta_2"] \\\
A \arrow[ru, "\eta_1"{name=A}] \arrow[rd, "\eta_3"] 
& \\
&\lim_{k \in \mathbb{N}}^h (i_k)^*\mathbb{L}(i_k)_!(A) ~, 
\end{tikzcd}
\]
for every dg $\mathcal{P}$-algebra $A$. The map $\eta_1: A \longrightarrow \mathrm{Res} \circ \mathbb{L}c\mathrm{Abs}(A)$ is the map from $A$ to its homotopy completion. The map $\eta_3: A \longrightarrow \lim_{k \in \mathbb{N}}^h (i_k)^*\mathbb{L}(i_k)_!(A)$ is a model for the map from $A$ to its nilcompletion. Finally, the map $\eta_2$ is given by applying $\mathrm{Res}$ to the unit of the adjunction in Theorem \ref{thm: arity decomposition of the infinity category of complete C algebras} at the object $\mathbb{L}c\mathrm{Abs}(A)$, so it is a natural equivalence, again by Theorem \ref{thm: arity decomposition of the infinity category of complete C algebras}. So indeed, $A$ is homotopy complete if and only if it is nilcomplete. 
\end{remark}

\subsubsection{Homotopy cocompleteness.} Let $\mathcal{P}$ be an dg operad satisfying Assumptions \ref{assumption 1}. Let us consider the inclusion adjunction 
\[
\begin{tikzcd}[column sep=5pc,row sep=3pc]
            \mathrm{B}\mathcal{P}\mbox{-}\mathsf{coalg} \arrow[r,"\mathrm{Inc}"{name=F},shift left=1.1ex ] &   \Omega\mathcal{P}^*\mbox{-}\mathsf{coalg} \arrow[l,"\mathrm{Sub}"{name=U},shift left=1.1ex ]
            \arrow[phantom, from=F, to=U, , "\dashv" rotate=-90]
\end{tikzcd}
\]
which is a Quillen adjunction by Lemma \ref{lemma: restriction and inclusion are Quillen functors}. Here we are considering the model structures where weak equivalences are given by quasi-isomorphisms on both sides. 

\begin{definition}[Homotopy cocomplete dg $\Omega\mathcal{P}^*$-coalgebra]
Let $C$ be a dg $\Omega\mathcal{P}^*$-coalgebra. It is \textit{homotopy cocomplete} if the derived counit of the adjunction
\[
\mathbb{R}\epsilon_C: \mathrm{Inc} \circ \mathbb{R}\mathrm{Sub}(C) \stackrel\sim\longrightarrow C
\]
is a quasi-isomorphism.
\end{definition}

The idea of this definition is that dg $\mathrm{B}\mathcal{P}$-coalgebras are inherently conilpotent (hence cocomplete) since by definition any element in them can only have finite iterated decompositions. \textit{Homotopy} cocompleteness is the analogous notion at the derived level. 

\begin{remark}
By definition, the derived maximal conilpotent subcoalgebra $\mathbb{R}\mathrm{Sub}$ restricts to a fully faithful functor
\[
\mathbb{R}\mathrm{Sub}: \Omega\mathcal{P}^*\mbox{-}\mathsf{coalg}^{\mathsf{ho.cocomp}}~[\mathsf{Q.iso}^{-1}] \hookrightarrow \mathrm{B}\mathcal{P}\mbox{-}\mathsf{coalg}~[\mathsf{Q.iso}^{-1}]
\]
between the $\infty$-subcategory of homotopy cocomplete dg $\Omega\mathcal{P}$-coalgebras up to quasi-isomorphisms and the $\infty$-category of dg $\mathrm{B}\mathcal{P}$-coalgebras up to quasi-isomorphisms.
\end{remark}

Let $\mathscr{C}$ be an enriched $\infty$-cooperad. Following \cite[Section 10]{heuts2024}, a (non-necessarily conilpotent) $\mathscr{C}$-coalgebra $Y$ is \textit{conilcomplete} if the canonical map 
\[
\lim_k \mathbf{t}_k C \longrightarrow C~,
\]
is an equivalence, where $\mathbf{t}_n C$ is the truncation of $C$ at operations in arities bigger than $n$ and is obtained by considering the successive arity truncations of the enriched $\infty$-cooperad $\mathscr{C}$. 

\begin{proposition}\label{prop: conilcompletion = homotopy cocomplete}
Let $\mathcal{P}$ be a dg operad satisfying Assumptions \ref{assumption 1}. Let $C$ be a dg $\Omega\mathcal{P}^*$-coalgebra. It is homotopy cocomplete if and only if it is conilcomplete.
\end{proposition}

\begin{proof}
The proof is completely dual to that of Proposition \ref{prop: nilcompletion = homotopy completetion}.
\end{proof}

\begin{remark}[A more relaxed version of conilcompleteness]
Homotopy cocomplete dg $\Omega\mathcal{P}^*$-coalgebras fully faithfully embed into the \textit{strict} version of conilpotent coalgebras, enconded by the model category of dg coalgebras over the cooperad $\mathrm{B}\mathcal{P}$ up to quasi-isomorphisms. One could also define a more relaxed version of this notion asking the counit of the $\infty$-categorical adjunction 
\[
\begin{tikzcd}[column sep=5pc,row sep=3pc]
         \mathbf{Coalg}_{\mathrm{B}\mathcal{P}}^{\mathbf{conil}}(\mathbf{D}(\kk)) \arrow[r, shift left=1.1ex, "\mathbf{Inc}"{name=F}] &\mathbf{Coalg}_{\Omega\mathcal{P}^*}(\mathbf{D}(\kk)) \arrow[l, shift left=.9ex, "\mathbf{Sub}"{name=U}] \arrow[phantom, from=F, to=U, , "\dashv" rotate=-90]
\end{tikzcd}
\]
to be an equivalence. Since dg $\mathrm{B}\mathcal{P}$-coalgebras up to quasi-isomorphisms do not rectify conilpotent coalgebras over the underlying $\infty$-cooperad of $\mathrm{B}\mathcal{P}$, this gives, in general, a different notion of conilcompleteness. See Counterexample \ref{Counterexample: non fullyfaithfulness of strict functor} for an example where this notion disagrees with the definition of conilcompletess given earlier.
\end{remark}

\subsection{The equivalence induced by the \texorpdfstring{$\infty$}{infinity}-categorical bar-cobar adjunction.} One of the main results of \cite{heuts2024}, Theorem 2.1, establishes an equivalence of $\infty$-categories between nilcomplete $\mathscr{P}$-algebras and conilcomplete $\mathbf{bar}(\mathscr{P})$-coalgebras. This theorem is viewed as a replacement for the original conjecture of Francis--Gaitsgory in \cite{FrancisGaitsgory}, which is disproved by Heuts. In our framework, the $\infty$-categorical bar-cobar adjunction restricts almost tautologically to an equivalence between homotopy complete $\mathcal{P}$-algebras and homotopy cocomplete $\Omega\mathcal{P}^*$-coalgebras (which under Assumptions \ref{assumption 1}, can be identified with non-conilpotent $\mathbf{bar}(\mathscr{P})$-coalgebras), and thus we recover, from the point-set descriptions introduced so far \cite[Theorem 2.1]{heuts2024} in the case where $\mathbf C=\mathbf D(\kk)$. 

\begin{theorem}\label{thm: infinty categorical bar-cobar is an equivalence on homotopy (co)complete}
Let $\mathcal{P}$ be a dg operad satisfying Assumptions \ref{assumption 1}. The adjunction 
\[
\begin{tikzcd}[column sep=5pc,row sep=3pc]
            \mathcal{P}\mbox{-}\mathsf{alg}~[\mathsf{Q.iso}^{-1}] \arrow[r,"\mathbf{bar}_{\mathcal{P}}"{name=F},shift left=1.1ex ] & \Omega\mathcal{P}^*\mbox{-}\mathsf{coalg}~[\mathsf{Q.iso}^{-1}] \arrow[l,"\mathbf{cobar}_{\mathcal{P}}"{name=U},shift left=1.1ex ]
            \arrow[phantom, from=F, to=U, , "\dashv" rotate=-90]
\end{tikzcd}
\]
restricts to an equivalence of $\infty$-categories 
\[
\mathcal{P}\mbox{-}\mathsf{alg}^{\mathsf{ho.comp}}~[\mathsf{Q.iso}^{-1}] \simeq \Omega\mathcal{P}^*\mbox{-}\mathsf{coalg}^{\mathsf{ho.cocomp}}~[\mathsf{Q.iso}^{-1}]~,
\]
between the $\infty$-subcategory of homotopy complete dg $\mathcal{P}$-algebras and the $\infty$-subcategory of homotopy cocomplete dg $\Omega\mathcal{P}^*$-coalgebras.  
\end{theorem}

\begin{proof}
As already explained in the proofs of Propositions \ref{prop: nilcompletion = homotopy completetion} and \ref{prop: conilcompletion = homotopy cocomplete}, the unit and the counit of the $\infty$-categorical bar-cobar adjunction can easily be identified with the maps that define homotopy completeness and homotopy cocompleteness, hence the result follows. 
\end{proof}

\subsection{Operads with good (co)completions}\label{subsection: operads with good completions}
In this subsection, we recall the definition of an operad with good completions introduced in \cite[Section 13]{heuts2024}. 


\begin{definition}[Operad with good completion]
Let $\mathcal{P}$ be an dg operad. It has \textit{good completion} if the for any chain complex $V$, the dg $\mathcal{P}$-algebra given by $V$ with the trivial algebra structure is homotopy complete. 
\end{definition}

The point-set description of homotopy complete algebras given in \cref{point set description of homotopy complete} can be made even more explicit here. Let $V$ be a trivial $\mathcal P$-algebra.  Taking the completion of the map $\Omega \mathrm B V \to V$ gives $\widehat \Omega \mathrm B V \to \widehat V = V$ (since a trivial algebra is in particular complete), which implies that $\mathcal P$ has good completions if and only if the projection from 
\[ \widehat \Omega \mathrm B V = \prod_{k>0} \mathcal P(k) \otimes_{\mathbb S_k} \bigoplus_{i_1+\dots+i_k=d} \mathrm{Ind}_{\mathbb{S}_{i_1} \times \cdots \times \mathbb{S}_{i_k}}^{\mathbb{S}_{d}} \left( \mathrm{B}\mathcal{P}(i_1) \otimes \cdots \otimes \mathrm{B}\mathcal{P}(i_k) \right) \otimes_{\mathbb{S}_d} V^{\otimes d}\]
to the summand 
\[ \mathcal P(1) \otimes \mathrm B\mathcal P(1) \otimes V = V\]
is a quasi-isomorphism, for any chain complex $V$. 

\medskip 
The main example of operads with good completions are \emph{Koszul operads} in the sense of the following definition. 

\begin{definition}
    Let $\mathcal P$ be a dg operad satisfying Assumptions \ref{assumption 1}. We say that $\mathcal P$ is \emph{Koszul} if $H_*(\mathcal P)=0$ for $\ast\neq 0$, and $H_*(\mathrm B\mathcal P(n)) = 0$ for $\ast \neq n-1$. 
\end{definition}

\begin{remark}
    The definition given here is somewhat restrictive: it covers the standard examples of $\mathcal Ass$, $\mathcal Com$, and $\mathcal Lie$, but not e.g.~$\mathcal Pois_n$ which is usually considered as a Koszul operad. 
\end{remark}

   The following theorem is proven in \cite[Section 13.2]{heuts2024}, and by a somewhat different argument in \cite[Theorem 2.9.4]{GaitsgoryRozenblyumVolII}.

 \begin{theorem}[Gaitsgory--Rozenblyum, Heuts] Let $\mathcal P$ be a Koszul operad. Then $\mathcal P$ has good completions. \label{thm: homotopy complete trivial algebras}
 \end{theorem}

\begin{proof}
    For the reader's convenience, let us translate Heuts's proof to our context. Koszulness implies $\mathcal P \simeq H_*(\mathcal P)$ and $\mathrm B\mathcal P \simeq H_*(\mathrm B\mathcal P)$, so we may and will replace both $\mathcal P$ and $\mathrm B \mathcal P$ with their homologies in what follows, and assume them both to be concentrated in a single degree. Consider the Koszul complex 
    \[ K(d) := \bigoplus_{k\geq 0} \mathcal P(k)\otimes_{\mathbb S_k} \bigoplus_{i_1+\dots+i_k=d} \mathrm{Ind}_{\mathbb{S}_{i_1} \times \cdots \times \mathbb{S}_{i_k}}^{\mathbb{S}_{d}}  \left( \mathrm{B}\mathcal{P}(i_1) \otimes \cdots \otimes \mathrm{B}\mathcal{P}(i_k) \right)\]
 with its Koszul differential. It has a descending filtration by the arity of $\mathcal P$, with graded pieces 
    \[ \operatorname{Gr}_F^p K(d) = \mathcal P(p)\otimes_{\mathbb S_p} \bigoplus_{i_1+\dots+i_p=d} \mathrm{Ind}_{\mathbb{S}_{i_1} \times \cdots \times \mathbb{S}_{i_p}}^{\mathbb{S}_{d}}  \left( \mathrm{B}\mathcal{P}(i_1) \otimes \cdots \otimes \mathrm{B}\mathcal{P}(i_p) \right).\]
    By Koszulness, $\operatorname{Gr}_F^p K(d)$ is concentrated in degree $d-p$. It follows that for each $p$, the complexes $F^p K(d)$ and $K(d)/F^p K(d)$ are simply the ``brutal'' truncations of $K(d)$. If $d>1$, the Koszul complex $K(d)$ is acyclic, so its brutal truncations have homology concentrated in a single degree. In particular, $H_*(K(d)/F^p K(d))$ is concentrated in degree $d-p$. 

\medskip

But now $\Omega \mathrm B V = \bigoplus_d K(d) \otimes_{\mathbb S_d} V^{\otimes d}$, and 
\begin{align*}
    \widehat \Omega \mathrm B V &= \varprojlim_p \bigoplus_d K(d)/F^p K(d) \otimes_{\mathbb S_d} V^{\otimes d} \\
    &= K(1) \otimes V \oplus \varprojlim_p \bigoplus_{d>1} K(d)/F^p K(d) \otimes_{\mathbb S_d} V^{\otimes d}.
\end{align*} 
Since $K(d)/F^p K(d)$ and $K(d)/F^{p-1} K(d)$ have homology concentrated in different degrees for $d>1$, the map between them is nullhomotopic, and then all the maps in the above cofiltered system are nullhomotopic. So the inverse limit vanishes, and $\mathcal P$ has good completions.
\end{proof}

It appears that there is no example in the literature of a dg operad which does \textit{not} have good completions. Let us record one here. 

\begin{counterexample}\label{counter-example: operad without good completion}
Consider the cooperad $\C_0$, whose underlying symmetric sequence is given by $\C_0(n) = \kk.c_n$ concentrated in degree $0$ for all $n \geq 1$, with the trivial $\mathbb{S}_n$-action and zero differential, and all cocompositions identically zero. Then the dg operad $\mathcal{P} = \Omega \C_0$ (which is simply the free dg operad on the symmetric sequence $s^{-1}\C_0$) does not have good completion.

\medskip

Let us explain why. Consider $V = \kk.v$, a 1-dimensional vector space over $\kk$, viewed as a trivial dg $\Omega \C_0$-algebra concentrated in degree $0$. We use the bar-cobar adjunction relative to the Koszul twisting morphism $\iota: \C_0 \longrightarrow \Omega \C_0$ to obtain a cofibrant resolution of $V$. Using this resolution, we deduce that $V$ is a homotopy complete dg $\Omega \C_0$-algebra if and only if the map 
\[
\iota_V: \bigoplus_{k \geq 1} \Omega \C_0 (k) \otimes_{\mathbb{S}_k} (\C_0 \circ V)^{\otimes k} \longrightarrow \prod_{k \geq 1} \Omega \C_0 (k) \otimes_{\mathbb{S}_k} (\C_0 \circ V)^{\otimes k}
\]
is a quasi-isomorphism. However, this is not the case. Indeed, on the right hand side, the formal power series
\[
\alpha = \sum_{k \geq 1} s^{-1}c_k \circ (c_1, \dots, c_1) (v, \dots, v)~, 
\]
is a degree $-1$ element which closed, which is not in the image of the differential, and which cannot be in the image of the inclusion $\iota_V$. So the map $\iota_V$ is not a quasi-isomorphism.
\end{counterexample}

\medskip

\subsubsection{Operads with good cocompletion.} We now define, for cofibrant dg operads, the notion of having good cocompletions. Notice that, up to equivalence, we can always choose an dg operad of the form $\Omega \C$, for some conilpotent dg cooperad $\C$. Moreover, given Assumptions \ref{assumption 1}, we can always choose it of the form $\Omega\mathcal{P}^*$, for some dg operad $\mathcal{P}$ satisfying Assumptions \ref{assumption 1}.

\begin{definition}[Operad with good cocompletion]
Let $\Omega\mathcal{P}^*$ be a cofibrant dg operad. It has \textit{good cocompletion} if the for any chain complex $V$, the dg $\Omega\mathcal{P}^*$-coalgebra given by $V$ with the trivial coalgebra structure is homotopy cocomplete. 
\end{definition}

We want to understand when $\Omega\mathcal{P}^*$ has good cocompletion, meaning that trivial dg $\Omega\mathcal{P}^*$ are homotopy cocomplete. The amounts to understanding when the canonical map
\[
\mathrm{B}_\pi \widehat{\mathrm{S}}^c(\mathcal{P}^*)(V) \rightarrowtail \widehat{\mathrm{B}}_\iota \widehat{\mathrm{S}}^c(\mathcal{P}^*)(V)
\]
is a quasi-isomorphism for any chain complex $V$. The complex $\widehat{\mathrm{B}}_\iota \widehat{\mathrm{S}}^c(\mathcal{P}^*)(V)$ involves the cofree construction, which quite hard to compute, so we start by reducing this question to a more manageable one.

\begin{lemma}\label{lemma: change of inclusions for cocompletion}
    A dg operad $\Omega \mathcal{P}^*$ has good cocompletion if and only if the canonical inclusion 
    \[
\mathrm{B}_\pi \mathrm{S}(\mathcal{P}^*)(V) \rightarrowtail \mathrm{B}_\pi \widehat{\mathrm{S}}^c(\mathcal{P}^*)(V)
\]
is a quasi-isomorphism for all chain complexes $V$. 
\end{lemma}

\begin{proof}
    Consider the following commutative square 
    \[
    \begin{tikzcd}[column sep=5pc,row sep=3pc]
    \mathrm{B}_\pi \widehat{\mathrm{S}}^c(\mathcal{P}^*)(V) \arrow[r,rightarrowtail]
    &\widehat{\mathrm{B}}_\iota \widehat{\mathrm{S}}^c(\mathcal{P}^*)(V) \arrow[d,"\simeq"] \\
    \mathrm{B}_\pi \mathrm{S}(\mathcal{P}^*)(V) \arrow[u,rightarrowtail] \arrow[ru,"\simeq"] \arrow[r,"\simeq"]
    &V
    \end{tikzcd}
    \]
    in chain complexes. The vertical right and bottom horizontal arrow are quasi-isomorphisms, therefore the diagonal map is also a quasi-isomorphism. Thus, by 2-out-of-3, the vertical left inclusion is a quasi-isomorphism if and only if the horizontal top one is. 
\end{proof}

\begin{remark}
    Yet another way to phrase things is to say that the canonical map 
    \[
    \mathrm{B}_\pi \widehat{\mathrm{S}}^c(\mathcal{P}^*)(V) \longrightarrow V
    \]
    is a quasi-isomorphism. 
\end{remark}

The first, quite obvious example that we have of a cofibrant dg operad with good cocompletion is the case of cobar constructions of truncated cooperads.

\begin{proposition}
Let $\mathcal{P}$ be a dg operad satisfying Assumptions \ref{assumption 1}. Then the dg operad $\Omega(\tau^{\leq k}\mathcal{P}^*)$ has good cocompletion for all $k \geq 1$.
\end{proposition}

\begin{proof}
In this case, the map 
\[
\mathrm{B}_\pi \mathrm{S}(\tau^{\leq k}\mathcal{P})(V) \rightarrowtail \mathrm{B}_\pi \widehat{\mathrm{S}}^c(\tau^{\leq k}\mathcal{P}^*)(V)
\]
is in fact an isomorphism for any $k \geq 1$, hence we can conclude by Lemma \ref{lemma: change of inclusions for cocompletion}. 
\end{proof}

This property seems to be much more rare in nature than operads with good completion, and in fact basic examples of cofibrant dg operads fail to have good cocompletion.

\begin{counterexample}\label{Counterexample: not good cocompletions} Let $\Omega \mathcal{C}om^*$ be the dg operad encoding $\mathcal{L}_\infty$-(co)algebras. Then $\mathcal{P}$ does \textit{not} have good cocompletion. Indeed, let $V$ be a countable infinite dimensional vector space concentrated in degree $0$. Then we have that 
\[
\widehat{\mathrm{S}}^c(\mathcal{C}om^*)(V) \cong \kk\llbracket x_1, x_2, x_3, \dots \rrbracket ~, 
\]
that is, the commutative algebra of formal power series in a countable infinite number of variables. Then it turns out that the map 
\[
\mathrm{B}_\pi \left(\kk\llbracket x_1, x_2, x_3, \dots \rrbracket \right) \longrightarrow V
\]
is not a quasi-isomorphism. This is because, while formal power series rings in a \textit{finite} number of variables are {smooth} over the base field, it is no longer the case for formal power series in an \textit{infinite} number of variables. See \cite[Remark 2.6.7]{GaitsgoryRozenblyumVolII} for more details on this counter-example. Notice that formal power series rings in infinitely many variables are also used by Heuts in \cite[Lemma 12.3]{heuts2024} as a counter-example to the original Francis--Gaitsgory conjecture, since they are $I$-adic completions which are not themselves $I$-adically complete. 

\medskip

The relationship between these two phenomena is the following: if the dg operad $\Omega \mathcal{C}om^*$ had good completion, then we would have an equivalence of $\infty$-categories by Proposition \ref{prop: absolute coincide with homotopy complete}
\[
\mathcal{C}om^*\mbox{-}\mathsf{alg}^{\mathsf{comp}}~[\mathsf{Q.iso}^{-1}] \simeq \mathcal{C}om\mbox{-}\mathsf{alg}^{\mathsf{ho.comp}}~[\mathsf{Q.iso}^{-1}]~, 
\]
between (complete) dg absolute commutative algebras up to quasi-isomorphisms and homotopy complete dg commutative algebras up to quasi-isomorphisms. Rectricting to degree $0$ (discrete objects) this equivalence would give an equivalence of 1-categories between absolute commutative algebras and $I$-adically completed commutative algebras, which cannot exist because the free absolute commutative algebra on an infinite number of variables $\kk\llbracket x_1, x_2, x_3, \dots \rrbracket$ is not $I$-adically complete. See also Remark \ref{Rmk: I-adic completion not an equivalence} and \cite[Section 3.4]{absolutealgebras} for more details. 
\end{counterexample}

\subsection{Derived inclusions and equivalences}\label{subsection: derived inclusions and equivalences}
In this subsection, we lay out the different fully faithful functors between the different $\infty$-categories that we have considered so far, and we explain why most of them are not in general equivalences of $\infty$-categories. We also explain some of the consequences of assuming that an operad has good (co)completions in this context.

\begin{proposition}\label{proposition: panorama of infinity cats}
Let $\mathcal{P}$ a dg operad satisfying Assumptions \ref{assumption 1}. There are the following fully faithful functors of $\infty$-categories
\[
\begin{tikzcd}[column sep=5pc,row sep=5pc]
\mathcal{P}\mbox{-}\mathsf{alg}^{\mathsf{ho.comp}}~[\mathsf{Q.iso}^{-1}] \arrow[r,"\mathrm{B}_\pi"{name=B},hookrightarrow] \arrow[d,"\mathbb{L}c\mathrm{Abs}"{name=SD},hookrightarrow] \arrow[rd, "\simeq"]
&\mathrm{B}\mathcal{P}\mbox{-}\mathsf{coalg}~[\mathsf{Q.iso}^{-1}] \\
\mathcal{P}^*\mbox{-}\mathsf{alg}^{\mathsf{comp}}~[\mathsf{Q.iso}^{-1}] 
&\Omega\mathcal{P}^*\mbox{-}\mathsf{coalg}^{\mathsf{ho.cocomp}}~[\mathsf{Q.iso}^{-1}] ~, \arrow[l,"\widehat{\Omega}_\iota"{name=CB},hookrightarrow] \arrow[u,"\mathbb{R}\mathrm{Sub}"{name=LD},hookrightarrow] 
\end{tikzcd}
\] 
where the top left and the bottom right are equivalent via the $\mathbf{bar}_{\mathcal{C}} \dashv \mathbf{cobar}_{\mathcal{C}}$ adjunction.
\end{proposition}

\begin{proof}
The functors are those of Theorem \ref{thm: full rectangle infinity categorical} and the fact that the top left and the bottom right $\infty$-categories are equivalent is given by Theorem \ref{thm: infinty categorical bar-cobar is an equivalence on homotopy (co)complete}. The two vertical functors are fully faithful by definition of homotopy complete $\C^*$-algebras and homotopy cocomplete $\Omega\C$-coalgebras. 

\medskip

Let us show that the functor $\mathrm{B}_\pi$ is fully faithful when restricted to homotopy complete $\mathcal{P}$-algebras. It is fully faithful (in fact, an equivalence) when it lands in the $\infty$-category of dg $\mathrm{B}\mathcal{P}$-coalgebras up to transferred weak equivalences, of which the $\infty$-category of dg $\mathrm{B}\mathcal{P}$-coalgebras up to quasi-isomorphisms is a localization. So it suffices to show that if $A$ is a homotopy complete $\mathcal{P}$-algebra, then $\mathrm{B}_\pi A$ is a local object with respect to quasi-isomorphism. There are weak equivalences of mapping spaces
\[
\begin{aligned}
\mathrm{Map}_{\mathrm{B}\mathcal{P}\mbox{-}\mathsf{coalg}~[\mathsf{W.eq}^{-1}]}(-,\mathrm{B}_\pi(A)) &\simeq \mathrm{Map}_{\mathcal{P}\mbox{-}\mathsf{alg}~[\mathsf{Q.iso}^{-1}]}(\Omega_\pi(-),A) \\ &\simeq \mathrm{Map}_{\mathcal{P}\mbox{-}\mathsf{alg}~[\mathsf{Q.iso}^{-1}]}(\Omega_\pi(-),\mathrm{Res}~\mathbb{L}c\mathrm{Abs}A) \\ &\simeq \mathrm{Map}_{\mathcal{P}^*\mbox{-}\mathsf{alg}^{\mathsf{comp}}~[\mathsf{Q.iso}^{-1}]}(\widehat{\Omega}_\iota(\mathrm{Inc}(-)),\mathbb{L}c\mathrm{Abs}A)~, 
\end{aligned}
\]
and this last mapping space clearly sends quasi-isomorphisms to weak equivalences of mapping spaces, since both $\widehat{\Omega}_\iota$ and $\mathrm{Inc}$ preserve quasi-isomorphisms. Hence $\mathrm{B}_\pi A$ is indeed local with respect to quasi-isomorphisms and the functor $\mathrm{B}_\pi$ is fully faithful. The proof for $\widehat{\Omega}_\iota$ is entirely analogous. 
\end{proof}

\begin{counterexample}
The fully faithful functor 
\[
\mathrm{B}_\pi: \mathcal{P}\mbox{-}\mathsf{alg}^{\mathsf{ho.comp}}~[\mathsf{Q.iso}^{-1}] \longrightarrow \mathrm{B}\mathcal{P}\mbox{-}\mathsf{coalg}~[\mathsf{Q.iso}^{-1}] 
\]
is not, in general, an equivalence of $\infty$-categories. Indeed, if it were, then the functor 
\[
\mathrm{Inc}: \mathrm{B}\mathcal{P}\mbox{-}\mathsf{coalg}~[\mathsf{Q.iso}^{-1}] \longrightarrow \Omega\mathcal{P}^*\mbox{-}\mathsf{coalg}~[\mathsf{Q.iso}^{-1}] 
\]
would necessarily have to be fully faithful as well. Now, take $\mathcal{P} = \mathcal{C}om$, and consider an infinite-dimensional vector space $V$, concentrated in degree $0$, endowed with the trivial $\mathrm{B}\mathcal{P}$-coalgebra structure. The derived counit 
\[
\mathbb{R}\mathrm{Sub}~\mathrm{Inc}(V) \longrightarrow V
\]
can be computed to be the map 
\[
\mathrm{B}_\pi \left(\kk\llbracket V \rrbracket \right) \longrightarrow V
\]
which, as explained in Counterexample \ref{Counterexample: not good cocompletions}, is not a quasi-isomorphism. 
\end{counterexample}

\subsubsection{The case of operads with good completions.} 
If we put together Theorem \ref{thm: full rectangle infinity categorical} with the commutative triangle of Proposition \ref{prop: commutative triangle of inclusions} and Theorem \ref{thm: point-set models for enhanced bar-cobar}, we get the following diagram of adjunctions 
\[
\begin{tikzcd}[column sep=3pc,row sep=3.5pc]
\mathcal{P}\mbox{-}\mathsf{alg}~[\mathsf{Q.iso}^{-1}] \arrow[r, shift left=1.1ex, "\mathrm{B}_\pi"{name=A}] 
&\mathrm{B}\mathcal{P}\mbox{-}\mathsf{cog}~[\mathsf{Q.iso}^{-1}] \arrow[l, shift left=1.1ex, "\mathbb{R}\Omega_\pi"{name=A}]
\arrow[r, shift left=1.5ex, "\mathbf{loose}"{name=A}]  \arrow[rd, shift left=0ex, shorten >=10pt, shorten <=10pt, "\mathrm{Inc}"{name=B}] 
&\mathbf{Coalg}_{~\mathbf{bar}(\mathcal{P})}^{\mathbf{conil}}(\mathbf{D}(\kk)) \arrow[d, shift left=2ex, "\mathbf{Inc}"{name=F}] \arrow[l, shift left=.75ex, "\mathbf{strict}"{name=C}] \\
&
&\Omega\mathcal{P}^*\mbox{-}\mathsf{cog}~[\mathsf{Q.iso}^{-1}]\arrow[u, shift left=1.1ex, "\mathbf{Sub}"{name=U}] \arrow[lu, shift left=2.5ex, shorten >=-5pt, shorten <=30pt, "\mathbb{R}\mathrm{Sub}"{name=D}] 
\end{tikzcd}
\]
where, once again, left adjoints are written on top. We state, with our notation, the following key result.

\begin{proposition}[{\cite[Proposition 11.1]{heuts2024}}]\label{prop: heuts' comparison between the two units}
Let $\mathcal{P}$ be a dg operad satisying Assumptions \ref{assumption 1}. Assume that $\mathcal{P}$ has good completions and let $A$ be a dg $\mathcal{P}$-algebra. Then the comparison map 
\[
\bm{\Omega}^{\mathbf{enh}} \bm{\mathrm{B}}^{\mathbf{enh}}(A) \longrightarrow \mathbf{cobar}_{\mathcal{P}} \mathbf{bar}_{\mathcal{P}}(A) 
\]
is an equivalence in the $\infty$-category of $\mathcal{P}$-algebras.
\end{proposition}

\begin{corollary}\label{cor: fully faithfulness of loose on homotopy complete}
Let $\mathcal{P}$ be a dg operad satisying Assumptions \ref{assumption 1} with good completions. The functor 
\[
\mathbf{loose}: \mathrm{B}\mathcal{P}\mbox{-}\mathsf{cog}~[\mathsf{Q.iso}^{-1}] \longrightarrow \mathbf{Coalg}_{~\mathbf{bar}(\mathcal{P})}^{\mathbf{conil}}(\mathbf{D}(\kk))
\]
is fully faithful on the essential image of homotopy complete $\mathcal{P}$-algebras via the operadic bar construction $\mathrm{B}_\pi$. 
\end{corollary}

\begin{proof}
Let $A$ be a homotopy complete $\mathcal{P}$-algebra. Then, by definition, the map 
\[
A \longrightarrow  \mathbf{cobar}_{\mathcal{P}} \mathbf{bar}_{\mathcal{P}}(A) \simeq \mathbb{R}\Omega_\pi ~ \mathbb{R}\mathrm{Sub} ~ \mathrm{Inc} ~ \mathrm{B}_\pi(A)
\]
is an equivalence. Hence, by Proposition \ref{prop: heuts' comparison between the two units}, the map 
\[
A \longrightarrow \bm{\Omega}^{\mathbf{enh}} \bm{\mathrm{B}}^{\mathbf{enh}}(A) \simeq \mathbb{R}\Omega_\pi ~ \mathbf{strict}~ \mathbf{loose}~ \mathrm{B}_\pi(A)
\]
is also an equivalence. Therefore, since $\mathrm{B}_\pi$ is fully faithful on all homotopy complete $\mathcal{P}$-algebras, we deduce that $\mathbf{loose}$ must also be fully faithful on the essential image of homotopy complete $\mathcal{P}$-algebras in $\mathrm{B}\mathcal{P}$-algebras via the functor $\mathrm{B}_\pi$. 
\end{proof}

\begin{example}
Let $\mathcal{P}$ be a connective dg operad. Then there are equivalences of $\infty$-categories
\[
\mathcal{P}\mbox{-}\mathsf{alg}_{\geq 1}~[\mathsf{Q.iso}^{-1}]  \simeq \mathrm{B}\mathcal{P}\mbox{-}\mathsf{cog}_{\geq 1}~[\mathsf{Q.iso}^{-1}] \simeq \mathbf{Coalg}_{~\mathbf{bar}(\mathcal{P})}(\mathbf{D}(\kk)))_{\geq 1}^{\mathbf{conil}}
\]
between $\mathcal{P}$-algebras in degrees $\geq 1$, $\mathrm{B}\mathcal{P}$-coalgebras in degrees $\geq 1$ and conilpotent coalgebras over the enriched $\infty$-cooperad $\mathrm{B}\mathcal{P}$ in degrees $\geq 1$. In particular, on these sub-categories the rectification of conilpotent $\mathrm{B}\mathcal{P}$-coalgebras holds. This can be interpreted as an example of Corollary \ref{cor: fully faithfulness of loose on homotopy complete}. Indeed, since $\mathcal{P}$ is connective, any $\mathcal{P}$-algebra in degrees $\geq 1$ is homotopy complete. An example of this are dg Lie algebras in degrees $\geq 1$ appearing in the work of Quillen in rational homotopy \cite{quillenrationalhomotopytheory}.
\end{example}

\begin{counterexample}\label{Counterexample: non fullyfaithfulness of strict functor}
As already pointed out in \cite[Theorem 5.3.1]{chen25}, the adjunction $\mathbf{loose} \dashv \mathbf{strict}$ is not, in general, an equivalence. Let us give yet another reinterpretation of the Counterexample \ref{counter-example: operad without good completion} to illustrate this. 

\medskip

Let $\mathcal{P} = \mathcal{C}om$. Since $\mathcal{C}om$ has good completions by Theorem \ref{thm: homotopy complete trivial algebras}, it implies, by \cite[Remark 11.2]{heuts2024}, that the functor $\mathbf{Inc}$ is fully faithful on trivial coalgebras over the enriched $\infty$-cooperad $\mathrm{B}\mathcal{P}$. This, in turn, means that the functor $\mathbf{Sub}$ is also fully faithful on any trivial $\Omega\mathcal{C}om^*$-coalgebra. However, as we saw in Counterexample \ref{counter-example: operad without good completion}, the operad $\Omega\mathcal{C}om^*$ does not have good cocompletions because infinite dimensional trivial $\Omega\mathcal{C}om^*$-coalgebras in degree $0$ are not homotopy cocomplete. This, in turn, means that $\mathbb{R}\mathrm{Sub}$ is not fully faithful on infinite dimensional trivial $\Omega\mathcal{C}om^*$-coalgebras. Which, finally, implies that the functor $\mathbf{strict}$ is not fully faithful either when restricted to infinite dimensional trivial coalgebras. Hence, this adjunction cannot be an equivalence. 
\end{counterexample}

\subsubsection{The case of operads with good cocompletions.} We show that when the cofibrant operad encoding coalgebras has good cocompletions, homotopy complete and absolute algebras on the other side of Koszul duality coincide.

\begin{proposition}\label{prop: absolute coincide with homotopy complete}
Let $\mathcal{P}$ be a dg operad satisying Assumptions \ref{assumption 1}. If the dg operad $\Omega \mathcal{P}^*$ has good cocompletions if and only if there are equivalences 
\[
    \mathcal{P}\mbox{-}\mathsf{alg}^{\mathsf{ho.comp}}~[\mathsf{Q.iso}^{-1}] \simeq \mathcal{P}^*\mbox{-}\mathsf{alg}^{\mathsf{comp}}~[\mathsf{Q.iso}^{-1}]  \simeq \Omega\mathcal{P}^*\mbox{-}\mathsf{coalg}^{\mathsf{ho.cocomp}}~[\mathsf{Q.iso}^{-1}]~. 
    \]
\end{proposition}

\begin{proof}
The adjunction $\widehat{\Omega}_\iota \dashv \mathbb{L}\widehat{\mathrm{B}}_\iota$ induces an equivalence at the level of $\infty$-categories if and only if, for any complete dg $\mathcal{P}^*$-algebra $B$, its image $\widehat{\mathrm{B}}_\iota(V)$ is homotopy cocomplete. Indeed, this is beacause we already know have that $\widehat{\mathrm{B}}_\iota$ is fully faithful. 

\medskip

First, we observe that $\Omega\mathcal{P}^*$ has good cocompletion precisely if, for any chain complex $V$, the canoncial map 
\[
\mathrm{B}_\pi \widehat{\mathrm{S}}^c(\mathcal{P}^*)(V) \longrightarrow \widehat{\mathrm{B}}_\iota \widehat{\mathrm{S}}^c(\mathcal{P}^*)(V)
\]
is a quasi-isomorphism. That is, if and only if the image by $\widehat{\mathrm{B}}_\iota$ of the free dg $\mathcal{P}^*$-algebra on $V$ is homotopy cocomplete. Therefore, if the equivalences hold, then $\Omega\mathcal{P}^*$ has good cocompletion.

\medskip

Now, in order to show that the image of any dg $\mathcal{P}^*$-algebra via this functor is homotopy cocomplete, we use that the $\infty$-category of complete dg $\mathcal{P}^*$-algebras up to quasi-isomorphisms is monadic over $\mathbf{D}(\kk)$ by Corollary \ref{cor: infinity monadicity of absolute algebras}, hence any algebra $B$ can be written as the geometric realization of a $\mathrm{U}$-split simplicial diagram of free dg $\mathcal{P}^*$-algebras. 

\medskip

The image of this geometric realization by $\widehat{\mathrm{B}}_\iota$ is again homotopy cocomplete since we have that
\[
\mathbb{R}\mathrm{Sub} \circ \widehat{\mathrm{B}}_\iota \simeq \mathrm{B}_\pi \circ \mathrm{Res}~, 
\]
where $\mathrm{Res}$ preserves geometric realization of a $\mathrm{U}$-split simplicial diagrams because it commutes with the forgetful functors and where $\mathrm{B}_\pi$ preserves all colimits since it is a left adjoint. Thus, the image of any dg $\mathcal{P}^*$-algebra $B$ via the complete bar functor is again homotopy cocomplete and the equivalences hold.
\end{proof}

\begin{remark}[Operads with good finite cocompletion]
As explained in Counterexample \ref{Counterexample: not good cocompletions}, basic examples of cofibrant operads fail to have good cocompletions. However, one could introduce the notion of a dg operad $\Omega\mathcal{P}^*$ with \textit{good finite cocompletion}, that is, such that the map 
  \[
\mathrm{B}_\pi \mathrm{S}(\mathcal{P})(V) \rightarrowtail \mathrm{B}_\pi \widehat{\mathrm{S}}^c(\mathcal{P}^*)(V)
\]
is a quasi-isomorphism for all perfect chain complexes $V$. The fact that $\Omega\mathcal{P}^*$ has finite good cocompletion should, in principle, induce a equivalence of $\infty$-categories 
\[
\mathcal{P}^*\mbox{-}\mathsf{alg}_{\mathsf{f.g}}^{\mathsf{comp}}~[\mathsf{Q.iso}^{-1}] \simeq \mathcal{P}\mbox{-}\mathsf{alg}_{\mathsf{f.g}}^{\mathsf{ho.comp}}~[\mathsf{Q.iso}^{-1}] 
\]
between finitely generated complete dg $\mathcal{P}^*$-algebras up to quasi-isomorphisms and finitely generated homotopy complete dg $\mathcal{P}^*$-algebras up to quasi-isomorphisms. Notice that examples like $\Omega \mathcal{C}om^*$ have good finite cocompletion but not good cocompletion. 
\end{remark}

\subsubsection{Truncated operads and point-set models for conilpotent coalgebras over their bar constructions.} Finally, we deal with arity-truncated dg operads and their bar constructions, and we show that the rectification of conilpotent coalgebras always holds in this case. 

\begin{proposition}\label{prop: loose is an equivalence for truncated operads}
Let $\mathcal{P}$ be a dg operad satisfying Assumptions \ref{assumption 1}. For any $k \geq 1$, the functor 
\[
\mathbf{loose}: \mathrm{B}(\tau^{\leq k}\mathcal{P})\mbox{-}\mathsf{cog}~[\mathsf{Q.iso}^{-1}] \longrightarrow \mathbf{Coalg}_{~\mathbf{bar}(\tau^{\leq k}\mathcal{P})}^{\mathbf{conil}}(\mathbf{D}(\kk))
\]
is an equivalence of $\infty$-categories, where $\tau^{\leq k}\mathcal{P}$ is the truncation of $\mathcal{P}$ at arity $k$.
\end{proposition}

\begin{proof}
This essentially follows from the fact that, for any $k \geq 1$, the point-set cobar construction $\Omega_\pi$ with respect to the twisting morphism $\pi: \mathrm{B}(\tau^{\leq k}\mathcal{P}) \longrightarrow \tau^{\leq k}\mathcal{P}$ preserves all quasi-isomorphisms. Hence the point-set bar-cobar adjunction induces is an equivalence of $\infty$-categories
\[
\tau^{\leq k}\mathcal{P}\mbox{-}\mathsf{alg}~[\mathsf{Q.iso}^{-1}] \simeq \mathrm{B}(\tau^{\leq k}\mathcal{P})\mbox{-}\mathsf{cog}~[\mathsf{Q.iso}^{-1}]~, 
\]
which combined with \cite[Proposition 13.13.]{heuts2024}, gives the desired result.
\end{proof}

\begin{remark}
Under these hypotheses, Proposition \ref{prop: absolute coincide with homotopy complete} also applies, so we also get that
\[
\mathrm{B}(\tau^{\leq k}\mathcal{P})\mbox{-}\mathsf{cog}~[\mathsf{Q.iso}^{-1}] \simeq \tau^{\leq k}\mathcal{P}^*\mbox{-}\mathsf{alg}^{\mathsf{comp}}~[\mathsf{Q.iso}^{-1}] \simeq \Omega\mathcal{P}^*\mbox{-}\mathsf{coalg}^{\mathsf{ho.cocomp}}~[\mathsf{Q.iso}^{-1}]~, 
\]
hence all $\infty$-categories considered are in fact equivalent except for dg $\Omega\mathcal{P}^*$-coalgebras up to quasi-isomorphisms.
\end{remark}


\section{Point-set models for bar-cobar adjunctions of \texorpdfstring{$\mathbb E_n$}{E_n}-algebras à la Lurie}\label{Section: Lurie's bar-cobar}


The goal of this section is to give point-set models for Lurie's bar-cobar adjunction between augmented $\mathbb{E}_1$-algebras and coaugmented $\mathbb{E}_1$-coalgebras as well as its generalization to the $\mathbb{E}_n$ case by Ayala--Francis, both in the case of chain complexes over a field of any characteristic. 
Both of these functors appear as the $\infty$-categorical bar cobar adjunction of Theorem \ref{thm: infty categorical operadic bar cobar} when one takes $\mathcal{P}$ to be a point-set model of the $\mathbb{E}_n$ operad.

\subsection{Lurie's bar-cobar adjunction and its generalization}
In \cite[Section 4.3]{HigherAlgebra}, Lurie constructed a universal adjunction between augmented monoids and coaugmented comonoids in any pointed monoidal $\infty$-category $\mathbf{D}$ admitting geometric realizations of simplicial objects and totalizations of cosimplicial objects:
\[
\begin{tikzcd}[column sep=5pc,row sep=3pc]
         \mathbf{Alg}_{\mathbb{E}_1}^{\mathbf{aug}}(\mathbf{D}) \arrow[r, shift left=1.1ex, "\mathbf{bar}"{name=F}] &\mathbf{Coalg}_{\mathbb{E}_1}^{\mathbf{coaug}}(\mathbf{D})~. \arrow[l, shift left=.75ex, "\mathbf{cobar}"{name=U}] \arrow[phantom, from=F, to=U, , "\dashv" rotate=-90]
\end{tikzcd}
\]
The idea behind this adjunction is the following: given an augmented monoid, one can construct canonically a simplicial object, whose realization is precisely the bar functor. Dually, given a coaugmented comonoid, one can construct canonically a  cosimplicial object, whose totalization is precisely the cobar functor. See Subsection \ref{subsubsection: Lurie's bar-cobar adjunction for operads and cooperads} for an example of this constructions in the case of $\infty$-operads and $\infty$-cooperads. 

\medskip

Later, Ayala and Francis constructed a generalization of this adjunction in \cite{AyalaFrancis19}. We follow the exposition of \cite[Section 3.2]{AyalaFrancis21}. Suppose now that $\mathbf{D}$ admits all sifted limits and colimits and that the monoidal structure distributes over both of them. Then, for any $n \geq 1$, there is a bar-cobar adjunction 
\[
\begin{tikzcd}[column sep=5pc,row sep=3pc]
         \mathbf{Alg}_{\mathbb{E}_n}^{\mathbf{aug}}(\mathbf{D}) \arrow[r, shift left=1.1ex, "\mathbf{bar}^{(n)}"{name=F}] &\mathbf{Coalg}_{\mathbb{E}_n}^{\mathbf{coaug}}(\mathbf{D})~. \arrow[l, shift left=.75ex, "\mathbf{cobar}^{(n)}"{name=U}] \arrow[phantom, from=F, to=U, , "\dashv" rotate=-90]
\end{tikzcd}
\]
between augmented $\mathbb{E}_n$-algebras in $\mathbf{D}$ and coaugmented $\mathbb{E}_n$-coalgebras in $\mathbf{D}$. 

\subsection{A homotopy-coherent comparison between two bar constructions}On first sight, one might be tempted to think that the above adjunctions coincide with the ones of Subsection \ref{Subsection: enhanced bar-cobar adjunction}, which were constructed in \cite{FrancisGaitsgory} in the particular case of the $\mathbb{E}_n$ operad. This cannot be the case, however, as the two categories of coalgebras on the right of these adjunctions are \textit{not} the same. Indeed, in the enhanced bar-cobar adjunction, one has \textit{conilpotent} coaugmented $\mathbb{E}_n$-coalgebras, whereas here one has \textit{all} coaugmented $\mathbb{E}_n$-coalgebras.  

\medskip

The question then becomes how to compare these two adjunctions and the notions of Koszul duality that they induce. It can be seen that at the level of the underlying objects in $\mathbf{D}$, both bar constructions coincide up to a shift, see \cite[Section 4.1]{BlansBlom24}. However, we need a full comparison of the two functors as homotopy coherent \textit{coalgebras}, a result which was recently obtained by Heuts and Land in the case when $\mathbf{D}$ is the $\infty$-category $\mathbf{Sp}$ of spectra. 

\begin{theorem}[\cite{HeutsLands}]\label{thm: heuts land comparison}
For any $n \geq 1$, there is an essentially unique equivalence of $\infty$-operads 
\[
(\mathrm{B}\mathbb{E}_n^{\mathrm{nu}})^\vee \qi s^{-n} \mathbb{E}_n^{\mathrm{nu}}
\]
such that the following square of $\infty$-categories commutes 
\[
\begin{tikzcd}[column sep=2.5pc,row sep=2.5pc]
\mathbf{Alg}_{\mathbb{E}_n^{\mathrm{nu}}}(\mathbf{Sp}) \arrow[r, "\mathbf{bar}^{(n)}"] \arrow[d, "\mathbf{B}^{\mathbf{enh}}"'] &
\mathbf{Coalg}_{\mathbb{E}_n^{\mathrm{nu}}}(\mathbf{Sp}) \arrow[d,"\Sigma^n"] \\
\mathbf{Coalg}_{\mathrm{B}\mathbb{E}_n^{\mathrm{nu}}}^{\mathbf{conil}}(\mathbf{Sp}) 
\arrow[r,"\mathbf{inc}"] 
&\mathbf{Coalg}_{(\mathrm{B}\mathbb{E}_n^{\mathrm{nu}})^\vee}(\mathbf{Sp}) \simeq \mathbf{Coalg}_{s^{-n}\mathbb{E}_n^{\mathrm{nu}}}(\mathbf{Sp})~, 
\end{tikzcd}
\]
where the functor $\mathbf{inc}$ is the inclusion of conilpotent coalgebras over the $\infty$-cooperad $\mathrm{B}\mathbb{E}_n^{\mathrm{nu}}$ into all coalgebras over its Spanier-Whitehead dual $\infty$-operad $(\mathrm{B}\mathbb{E}_n^{\mathrm{nu}})^\vee$. 
\end{theorem}

\begin{remark}
Notice that \textit{augmented} $\mathbb{E}_n$-algebras are canonically equivalent to algebras over the non-unital version  $\mathbb{E}_n^{\mathrm{nu}}$ of the $\mathbb{E}_n$ operad, and similarly for coaugmented coalgebras. 
\end{remark}

\begin{remark}
    The original result in \cite{HeutsLands} is stated using $\infty$-cooperads and the definition of non-conilpotent coalgebras over them given in \cite[Appendix A]{heuts2024}. We have slightly reformulated it in an equivalent way, using Lemma \ref{lemma: equivalence between non-conilpotent coalgebras over a cooperad a la Heuts and coalgebras over the dual operad} for the convinience of the reader. 
\end{remark}

\begin{remark}
By smashing with the Eilenberg-MacLane spectrum $\mathrm{H}\kk$, Theorem \ref{thm: heuts land comparison} also holds when we replace the $\infty$-category of spectra with the $\infty$-category $\mathbf{D}(\kk)$. 
\end{remark}

\subsection{Point-set models for everyone}
We construct concrete point-set models for Lurie's bar-cobar adjunction in the base $\infty$-category of chain complexes over a field of any characteristic. We also construct point-set models for Ayala--Francis bar-cobar adjunctions between non-unital $\mathbb{E}_n$-algebras and non-counital $\mathbb{E}_n$-coalgebras over a field of characteristic zero when $n \geq 2$. 

\begin{theorem}\label{thm: rectification of Lurie's bar-cobar}
Let $\kk$ be a field of any characteristic and let $\mathcal{A}ss$ be the associative operad. The adjunction 
\[
\begin{tikzcd}[column sep=5pc,row sep=3pc]
            \mathcal{A}ss\mbox{-}\mathsf{alg}~[\mathsf{Q.iso}^{-1}] \arrow[r,"\mathbf{bar}_{\mathcal{A}ss}"{name=F},shift left=1.1ex ] & \Omega\mathcal{A}ss^*\mbox{-}\mathsf{coalg}~[\mathsf{Q.iso}^{-1}] \arrow[l,"\mathbf{cobar}_{\mathcal{A}ss}"{name=U},shift left=1.1ex ]
            \arrow[phantom, from=F, to=U, , "\dashv" rotate=-90]
\end{tikzcd}
\]
is naturally weakly equivalent to the following composite adjunctions
\[
\begin{tikzcd}[column sep=5pc,row sep=3pc]
         \mathbf{Alg}_{\mathbb{E}_1}^{\mathbf{aug}}(\mathbf{D}) \arrow[r, shift left=1.1ex, "\mathbf{bar}"{name=F}] &\mathbf{Coalg}_{\mathbb{E}_1}^{\mathbf{coaug}}(\mathbf{D}) 
         \arrow[l, shift left=.75ex, "\mathbf{cobar}"{name=U}] \arrow[phantom, from=F, to=U, , "\dashv" rotate=-90] \arrow[r, shift left=1.1ex, "\Sigma"{name=D}] 
         &\mathbf{Coalg}_{s^{-1}\mathbb{E}_1}^{\mathbf{coaug}}(\mathbf{D})  \arrow[l, shift left=.75ex, "\Sigma^{-1}"{name=C}] \arrow[phantom, from=D, to= C, , "\simeq" rotate=0] 
\end{tikzcd}
\]
that is, to Lurie's bar-cobar adjunction up to a shift. 
\end{theorem}

\begin{proof}
    Recall that the existence of such an adjunction at the $\infty$-categorical level was proven in Theorem \ref{thm: infty categorical operadic bar cobar}, where these functors are given by the compositions 
    \[
    \mathbf{bar}_{\mathcal{A}ss} = \mathrm{Inc} \circ \mathrm{B}_\pi \quad \text{and} \quad \mathbf{cobar}_{\mathcal{A}ss} = \mathrm{Res} \circ \widehat{\Omega}_\iota 
    \]
    in the commutative square of Theorem \ref{thm: full rectangle infinity categorical}. 

    \medskip

    First, we have to check that each of these $\infty$-categories presents, repectively, the $\infty$-categories of augmented $\mathbb{E}_1$-algebras and coaugmented $\mathbb{E}_1$-coalgebras. The fact that dg associative algebras localized at quasi-isomorphisms present their augmented $\mathbb{E}_1$-algebras follows from standard rectification arguments, see \cite[Section 4.1.8]{HigherAlgebra}. Notice that the dg operad  $\Omega\mathcal{A}ss^*$ encodes (shifted) $\mathcal{A}_\infty$-(co)algebras, so this is where the shift appears. The fact that dg $\Omega\mathcal{A}ss^*$-coalgebras localized at quasi-isomorphisms present (shifted) coaugmented $\mathbb{E}_1$-coalgebras follows from Theorem \ref{thm: rectification of P-coalgebras}, which is the main result of \cite{rectification}. 

    \medskip
    
    Hence we are left to compare the adjunctions, since their source and targets are equivalent. In order to do so, it suffices to show that the left adjoints are weakly equivalent. This follows directly from Theorem \ref{thm: heuts land comparison} combined with Theorem \ref{thm: point-set models for enhanced bar-cobar} and Proposition \ref{prop: commutative triangle of inclusions}. 
\end{proof}

\begin{theorem}\label{thm: rectification of Ayala-Francis}
Let $n \geq 2$ and $\kk$ be a field of characteristic zero. Let $\mathcal{P}ois_n$ denote the operad of $n$-Poisson algebras. The adjunction 
\[
\begin{tikzcd}[column sep=5pc,row sep=3pc]
            \mathcal{P}ois_n\mbox{-}\mathsf{alg}~[\mathsf{Q.iso}^{-1}] \arrow[r,"\mathbf{bar}_{\mathcal{P}ois_n}"{name=F},shift left=1.1ex ] & \Omega\mathcal{P}ois_n^*\mbox{-}\mathsf{coalg}~[\mathsf{Q.iso}^{-1}] \arrow[l,"\mathbf{cobar}_{\mathcal{P}ois_n}"{name=U},shift left=1.1ex ]
            \arrow[phantom, from=F, to=U, , "\dashv" rotate=-90]
\end{tikzcd}
\]
is naturally weakly equivalent to the following composite adjunctions
\[
\begin{tikzcd}[column sep=5pc,row sep=3pc]
         \mathbf{Alg}_{\mathbb{E}_n}^{\mathbf{aug}}(\mathbf{D}) \arrow[r, shift left=1.1ex, "\mathbf{bar}^{(n)}"{name=F}] &\mathbf{Coalg}_{\mathbb{E}_n}^{\mathbf{coaug}}(\mathbf{D}) 
         \arrow[l, shift left=.75ex, "\mathbf{cobar}^{(n)}"{name=U}] \arrow[phantom, from=F, to=U, , "\dashv" rotate=-90] \arrow[r, shift left=1.1ex, "\Sigma^n"{name=D}] 
         &\mathbf{Coalg}_{s^{-n}\mathbb{E}_n}^{\mathbf{coaug}}(\mathbf{D})  \arrow[l, shift left=.75ex, "\Sigma^{-n}"{name=C}] \arrow[phantom, from=D, to= C, , "\simeq" rotate=0] 
\end{tikzcd}
\]
that is, to Ayala--Francis's bar-cobar adjunction up to an iterated shift. 
\end{theorem}

\bibliographystyle{alpha}
\bibliography{bibax}
\end{document}